\documentclass[leqno]{amsart}
\usepackage[english,frenchb]{babel}
\usepackage{lmodern}
\usepackage{enumerate}
\usepackage{amsmath}
\usepackage{amsfonts}
\usepackage{amssymb}
\usepackage{amsthm}
\usepackage{graphicx}
\usepackage{mathrsfs}
\usepackage{pxfonts}
\usepackage{datetime}
\usepackage{comment}
\usepackage{fourier-orns}
\usepackage{dsfont}
\usepackage[left=3cm,right=3cm,top=3cm,bottom=3cm]{geometry}
\usepackage[cyr]{aeguill}
\usepackage{fancyhdr}
\usepackage{mathabx}
\usepackage[latin1]{inputenc}
\usepackage[all,cmtip]{xy}

\usepackage{tikz-cd}
\usetikzlibrary{decorations.pathmorphing}

\usepackage{amsfonts,amsmath,graphicx}

\usepackage{geometry}
\newcommand{\bc}{\begin{center}}
\newcommand{\ec}{\end{center}}
\newcommand{\bq}{\begin{quote}}
\newcommand{\eq}{\end{quote}}
\newcommand{\bqtn}{\begin{quotation}}
\newcommand{\eqtn}{\end{quotation}}
\newcommand{\beq}{\begin{equation}}
\newcommand{\eeq}{\end{equation}}
\newcommand{\bearr}{\begin{eqnarray}}
\newcommand{\eearr}{\end{eqnarray}}
\newcommand{\bearrn}{\begin{eqnarray*}}
\newcommand{\eearrn}{\end{eqnarray*}}
\newcommand{\bi}{\begin{itemize}}
\newcommand{\ei}{\end{itemize}}
\newcommand{\be}{\begin{enumerate}}
\newcommand{\ee}{\end{enumerate}}
\newcommand{\bthe}{\begin{theorem}}
\newcommand{\ethe}{\end{theorem}}
\newcommand{\blem}{\begin{lemme}}
\newcommand{\elem}{\end{lemme}}
\newcommand{\bsolu}{\begin{solution}}
\newcommand{\esolu}{\end{solution}}
\newcommand{\bexer}{\begin{exercise}}
\newcommand{\eexer}{\end{exercise}}

\newcommand{\ca}{\alpha}
\newcommand{\cb}{\beta}
\newcommand{\cP}{\phi}
\newcommand{\cQ}{\chi}
\newcommand{\cR}{\psi}

\newcommand{\ba}{\begin{array}}
\newcommand{\ea}{\end{array}}

\usepackage{amsfonts,amsmath}
\newtheorem{theoreme}{Theorem}[section]
\newtheorem{theorem}[theoreme]{Theorem}
\newtheorem{lemme}[theoreme]{Lemma}
\newtheorem{lemma}[theoreme]{Lemma}
\newtheorem{proposition}[theoreme]{Proposition}
\newtheorem{definition}[theoreme]{Definition}

\newtheorem{corollaire}[theoreme]{Corollary}

\newtheorem{solution}[theoreme]{Solution}
\newtheorem{exercise}[theoreme]{Exercise}
\newcommand{\bdefi}{\begin{definition}}
\newcommand{\edefi}{\end{definition}}
\newcommand{\brk}{\begin{remarque}}
\newcommand{\erk}{\end{remarque}}
\newcommand{\bpp}{\begin{proposition}}
\newcommand{\epp}{\end{proposition}}
\newcommand{\bpf}{\begin{proof}}
\newcommand{\epf}{\end{proof}}
\newcommand{\bcor}{\begin{corollaire}}
\newcommand{\ecor}{\end{corollaire}}
\newcommand{\bsol}{\begin{solution}}
\newcommand{\esol}{\end{solution}}
\theoremstyle{definition}

\newtheorem{remarque}[theoreme]{Remark}
\allowdisplaybreaks
\title{A complete classification of cubic function fields over any finite field}
\author{Sophie Marques and Kenneth Ward}

\begin{document}
\large
\selectlanguage{english}
\maketitle
\begin{abstract}
We classify all cubic function fields over any finite field, particularly developing a complete Galois theory which includes those cases when the constant field is missing certain roots of unity. In doing so, we find criteria which allow one to easily read ramification and splitting data from the generating equation, in analogy to the known theory for Artin-Schreier and Kummer extensions. We also describe explicit irreducibility criteria, integral bases, and Galois actions in terms of canonical generating equations.
\end{abstract}

\noindent \quad {\footnotesize MSC Code (primary): 11T22}

\noindent \quad {\footnotesize MSC Codes (secondary):  11T55, 11R32, 14G15, 11R58}


\tableofcontents

\section*{Introduction}
Let $p > 0$ be a prime integer, $\mathbb{F}_q$ a finite field with $q=p^n$ elements, and $K$ a function field with field of constants $\mathbb{F}_q$. Let $L/K$ be a Galois cubic extension. In general, if the field characteristic is equal to 3, then Artin-Schreier theory is used to describe the extension $L/K$, which is given by an equation $y^3 - y = a$, for $a \in K$, whereas if the characteristic is not equal to 3 and the constant field contains a primitive third root of unity, then Galois theory is well understood via Kummer theory, which gives a generating equation $y^3 = a$, for $a \in K$ (see for example \cite{Sti,Vil}). The situation is more delicate when the constant field does not contain a primitive third root of unity. Our goal in this paper is to investigate \emph{all} Galois cubic extensions of function fields $K$ over finite fields in any characteristic, and in doing so, we are able to give a canonical Galois theory for cubic function fields over finite fields when Artin-Schreier and Kummer theory cannot be used. 

We refer to extensions with generating equation $X^3 =a$ as \emph{purely cubic}. If the characteristic is different from $3$, we give a simple criterion which determines whether or not a given extension of function fields is purely cubic. If such an extension $L/K$ is not purely cubic, then we show that it has a generation of the form $L = K(y)$, with $y$ such that $$y^3 -3y-a=0\quad\quad\quad\quad\quad (a \in K).$$ We will study this form in detail, as it is central to Galois theory when Artin-Schreier and Kummer theory are not useful. We show that it is possible, and in fact practical, to read the ramification and splitting directly from this form. In particular, this may be done in terms of the factorisation of $a$. This gives a version of \emph{standard form}, which is well-known for Kummer and Artin-Schreier extensions (and we will henceforth refer to our form as standard due to this analogy). We are also able to completely describe the following using our standard form in conjunction with Artin-Schreier and Kummer theory:

\begin{itemize} \item Irreducibility criteria for degree 3 polynomials; \item Galois cubic extensions of function fields in any characteristic; \item Galois actions and their connections to splitting; and\item Algorithms for providing integral bases.\end{itemize}

We emphasise that the advantage of this approach is the ability to obtain all of the above information concisely using our standard form. When joined with classical Artin-Schreier and Kummer theory, the results which follow therefore provide a complete study of Galois structure for cubic function fields over finite fields.

For ease of reading, we have divided this paper in the following way. Section 1 is devoted to reducing general minimal polynomials of cubic extensions $L/K$ to standard forms. In this way, we recover Artin-Schreier and Kummer theory. For completeness, Section 2 examines the structure of the Galois closure of $L/K$ when the extension is not necessarily Galois. Section 3 reviews the Artin-Schreier and its ramification theory, and Section 4 does just the same for Kummer extensions. Section 5 addresses the matter of cubic extensions of function fields over finite fields when Artin-Schreier and Kummer theory cannot be used, which we note is the crux and \emph{raison d'\^{e}tre} of this paper. Owing to the length of the arguments, we relegate a portion of the full proofs of Section 5 to the Appendix.

\section*{Acknowledgements}

The authors thank Ben Blum Smith for valuable comments and discussions on Theorem \ref{qneq1mod3Galois}. \\ Kenneth Ward thanks the CAS Mellon Fund at American University for its generous support.

\section{Minimal polynomials for generators of cubic extensions}
In this section, we prove that any cubic extension has a primitive element whose minimal polynomial is either of the form
\begin{enumerate} 
\item $T(X)= X^3 -a$, for some $a\in K$ or
\item $T(X) = X^3 +X^2 +b$ for some $b \in K$, and that this is equivalent to the existence of a generator whose minimal polynomial is of the form
\begin{enumerate}
\item $T(X)= X^3-3X-a$, for some $a\in K$, in characteristic different from $3$.
\item $T(X) = X^3 +aX +a^2$, for some $a\in K$, in characteristic $3$. 
\end{enumerate}
\end{enumerate} 
We note that the first case yields a purely inseparable extension in characteristic $3$. 
\begin{definition} 
A cubic extension $L/K$ is called \emph{purely cubic} if there exists a primitive element $y$ for $L/K$ such that the minimal polynomial of $y$ over $K$ is of the form $y^3 =a$, with $a \in K$.
\end{definition}
 If the characteristic is different from $3$, we will give a simple criterion (Corollary 1.4) which determines whether or not a given extension of function fields is purely cubic. The following Lemma proves that one can find a generator with no square term. This is well known and originally due to Tartaglia, as in his contribution to the \emph{Ars Magna}, ``On the cube and first power equal to the number" \cite[Chapter XV, p. 114]{Cardano}.  (For a more modern reference, see also \cite[p. 112]{modernreference}.)
\begin{lemme}\label{linear}
Suppose that $p\neq 3$, and let $L/K$ be a separable extension of degree $3$. Then there exists an explicit primitive element $y$ with characteristic polynomial of the form $T(X)= X^3 + a X + b$, with $a,b \in K$. 
\end{lemme} 
The discriminant of the polynomial $T(X)= X^3 + a X + b$ of Lemma \ref{linear} is equal to $-4a^3 - 27 b^2$. We will use this form of the discriminant later. In the following Theorem, we show that one may find a generator of the form announced at the outset of this section. This will be crucial in the sequel; one major goal of this analysis is to use the generator of a cubic extension to study ramification, reducibility, and integral basis construction.
\begin{theoreme} \label{quadratic}
Let $p\neq 3$. Let $L/K$ be a separable extension of degree $3$ with generating equation $$T(y) = y^3 + e y^2 + fy + g = 0,$$ where $e,f,g \in K$. 
\begin{enumerate}[(a)] \item Suppose that $3eg \neq f^2$. Then there exists a primitive element $z$ of $L/K$ with characteristic polynomial of the form $T(X)= X^3 + X^2 + b$, with $b \in K$. \item Suppose that $3eg = f^2$. Then there exists a primitive element $z$ of $L/K$ with characteristic polynomial of the form $T(X)=X^3 - b$, with $b \in K$. \end{enumerate} Furthermore, in each case, this primitive element is explicitly determined.
\end{theoreme} 
\begin{proof} 
\begin{enumerate}[(a)] \item 
Let $y$ a generating element of $L/K$ then there exist $e,f, g \in K$, $g \neq 0$ such that the minimal polynomial of $y$ is of the form:
$$S(X) = X^3 + e X^2 + fX + g$$ 
This may be seen by performing the rational transformation $$y \rightarrow y'=\frac{y}{\frac{f}{3g} y +1},$$
as $g\neq 0$ by irreducibility of $S(X)$. It follows that $y=\frac{y'}{1 - \frac{fy'}{3g}}$, which yields $$\left(\frac{y'}{1 - \frac{fy'}{3g}}\right)^3 + e \left(\frac{y'}{1 - \frac{fy'}{3g}}\right)^2 + f\left(\frac{y'}{1 - \frac{fy'}{3g}}\right) + g = 0.$$ Multiplication by $(1 - \frac{fy'}{3g})^3$ then yields \begin{align}\label{gamma}\notag 0 &= y'^3 + ey'^2 \left({1 - \frac{fy'}{3g}}\right) + f y' \left({1 - \frac{fy'}{3g}}\right)^2 + g \left({1 - \frac{fy'}{3g}}\right)^3 \\&= \left(1 -\frac{ef}{3g}+ \frac{2f^3}{27g^2}\right) y'^3 + \left(e-\frac{f^2}{3g}\right) y'^2 + g.\end{align} Let $\gamma =1 -\frac{ef}{3g}+ \frac{2f^3}{27g^2}$. If $f=0$, then $\gamma=1\neq 0$, and otherwise,
\begin{align*} \gamma &=1 + e\left(-\frac{f}{3g}\right) + f\left(-\frac{f}{3g}\right)^2 + g \left(-\frac{f}{3g}\right)^3 \\ &=\left(-\frac{f}{3g}\right)^3\left(\left(-\frac{3g}{f}\right)^3 + e\left(-\frac{3g}{f}\right)^2 + f\left(-\frac{3g}{f}\right) + g \right)\\&= \left(-\frac{f}{3g}\right)^3S\left(-\frac{3g}{f}\right). \end{align*} As the polynomial $S(X)= X^3 + eX^2+fX + g$ is irreducible over $K$, it follows that $-\frac{3g}{f}$ cannot be a root of $S$, whence $$\gamma =\left(-\frac{f}{3g}\right)^3S\left(-\frac{3g}{f}\right) \neq 0.$$ Therefore, in any case, $\gamma \neq 0$. We let $\alpha = e-\frac{f^2}{3g}$. Dividing \eqref{gamma} by $\gamma$, we obtain \begin{equation}\label{gammazero} y'^3 + \frac{\alpha}{\gamma} y'^2 + \frac{g}{\gamma} =0.\end{equation} By assumption, $3eg \neq f^2$, so that $\alpha \neq 0$. We let $z= \frac{\gamma}{\alpha}y'$. Hence, with 
$$b=\frac{\gamma^2 g }{\alpha^3} =  \frac{ \left( 1 -\frac{ef}{3g}+ \frac{2f^3}{27g^2}\right)^2 g }{\left(  e-\frac{f^2}{3g}\right)^3} =\frac{ ( 27 g^2 -9efg +2f^3)^2}{27(3ge-f^2)^3} $$ and this choice of $z$, we obtain $z^3 + z^2 + b = 0$, as desired.
\item This follows as in the previous case by \eqref{gammazero} and $b = -\frac{g}{\gamma}$.
\end{enumerate}
\end{proof}
We note that the discriminant of the polynomial $T(X)$ of Lemma \ref{quadratic}(a) is equal to $-4b - 27 b^2$.
\begin{corollaire} \label{linearthree}
Let $p\neq 3$. Let $L/K$ be a separable extension of degree $3$ with generating equation $$T(y) = y^3 + e y^2 + fy + g = 0,$$ where $e,f,g \in K$.
  \begin{enumerate}[(a)] \item Suppose that $3eg \neq f^2$. Then there exists a primitive element $z$ for $L/K$ with characteristic polynomial of the form $T(X)= X^3 -3 X - b$, with $b \in K$. 
\item Suppose that $3eg = f^2$. Then there exists a primitive element $z$ for $L/K$ with characteristic polynomial of the form $T(X)=X^3 - b$, with $b \in K$. \end{enumerate}
In each case, this primitive element is explicitly determined.
\end{corollaire}\label{characteristic}
\begin{proof} \begin{enumerate}[(a)] \item By Lemma \ref{quadratic}, if $3eg \neq f^2$, then there exists an explicitly determined primitive element $z$ for $L/K$ with characteristic polynomial of the form $P(X)= X^3 + X^2 + a$, with $a \in K$. We then perform the linear change of variable $z= z' - \frac{1}{3}$ and \begin{align*}0&= \left(z' - \frac{1}{3}\right)^3 + \left(z'  - \frac{1}{3}\right)^2 +a\\ &= z'^3 - z'^2 +\frac{1}{3}z'-\frac{1}{27} + z'^2 -\frac{2}{3}z' +\frac{1}{9}+a\\ &=  z'^3-\frac{1}{3}z'+\frac{2}{27}+a.\end{align*}
Letting $z'= \frac{z''}{3}$, we therefore obtain
 $$0=3^3\left[\left(\frac{z''}{3}\right)^3-\frac{1}{3}\frac{z''}{3}  +\frac{2}{27}+a\right] = z''^3 -3 z'' +2+ 27 a.$$ Letting $$b= -2 -27a=-2-\frac{ ( 27 g^2 -9efg +2f^3)^2}{(3ge-f^2)^3} ,$$ it follows that $z''$ is the desired primitive element. 
 \item This case follows immediately from Lemma \ref{quadratic}(b).
\end{enumerate}
\end{proof}
We note that the discriminant of the polynomial $P(X)$ of Corollary \ref{characteristic}(a) is equal to $-27(-4 + b^2)$. If the characteristic is equal to $3$, we find two forms for the minimal polynomial of a generator of a cubic extension: 
\begin{theoreme} 
Suppose the characteristic of $k$ is $3$. Let $L/K$ be a extension of degree $3$. 
Then either:
\begin{enumerate}
\item $L/K$ is separable, there is a primitive element $z$ such that its minimal polynomial is equal to $T(X)= X^3 + X^2 +b$, and this primitive element is explicitly determined; or 
\item $L/K$ is purely inseparable and there is a primitive element $z$ such that its minimal polynomial is equal to $T(X)= X^3  +b$.
\end{enumerate} 
\end{theoreme}
\begin{proof} 
Let $y$ a generating element of $L/K$ then there exist $e,f, g \in K$, $g \neq 0$ such that the minimal polynomial of $y$ is of the form:
$$S(X) = X^3 + e X^2 + fX + g$$ Clearly, $L/K$ is separable if, and only if, not both $e$ and $f$ are zero. 

Suppose then that $L/K$ is separable.
If $e=0$, then as $T(y)$ is irreducible, $g \neq 0$. Let $y'= 1/y$. Then
$$y'^3 + \frac{f}{g} y'^2 +\frac{1}{g}=0.$$ 
If $L/K$ separable then $f\neq 0$, and taking $z=  \frac{g}{f}y'$, we obtain $$z^3 +  z^2 +\frac{g^2}{f^3}=0.$$ 
If, on the other hand, $e\neq 0$, then let $y' = y -\frac{f}{e}$. As the field characteristic is equal to 3, then
\begin{align*} 0&=\left(y' +\frac{f}{e}\right)^3 + e \left(y' +\frac{f}{e}\right)^2 + f\left(y' +\frac{f}{e}\right) +g \\ 
&=y'^3 +\frac{f^3}{e^3}+ e y'^2 +2fy'+ \frac{f^2}{e}+ fy' + \frac{f^2}{e} +g \\ 
&=y'^3 + e y'^2 + \frac{2f^2}{e} +g+\frac{f^3}{e^3}. 
\end{align*}
With $z= \frac{1}{e} y'$, we therefore obtain $$0 = z^3 + z^2 + \frac{2f^2}{e^4} +\frac{g}{e^3}+\frac{f^3}{e^6}.$$
\end{proof} 
\begin{corollaire} \label{char3generator}
Let $p= 3$. Let $L/K$ be a separable extension of degree $3$. Then there is a primitive element $z$ such that its minimal polynomial is of the form $T(X)= X^3 + bX +b^2$. Furthermore, this primitive element is explicitly determined.
\end{corollaire}
\begin{proof} 
From the previous Theorem, there is a primitive element $y'$ such that  $$y'^3 +y'^2 + b=0.$$
Let $y''=1/y'$; then $$by''^3 +y'' +1=0.$$
Finally, letting $z=by''$, we obtain $$z^3 +b z +b^2=0.$$
\end{proof} 
\section{Galois closure of cubic extensions}
We begin this Section with a well-known result on the Galois group of a cubic extension (see for instance \cite[Theorems 1.1, 2.1, 2.6]{Con}): 
\begin{lemme}\label{disc} 
Let $p\neq 2$. Suppose that the cubic extension $L/K$ is separable. Then $L/K$ is Galois if, and only if, the discriminant is a square in $K$. Furthermore, in this case, $\text{Gal}(L/K) = A_3$.
\end{lemme}
We denote the minimal polynomial of a primitive element $y$ of $L/K$ as $$T(X)=X^3 + e X^2 + f X + g,\quad\quad e,f,g\in K.$$ As $T(X)$ is irreducible, it is of course necessary that $g \neq 0$. As $T(X)$ has degree $3$ in $X$, it follows that $T(X)$ is irreducible over $K$ if, and only if, it possesses no root in $K$. For ease of notation, we denote by $L^{G_K}/K$ the Galois closure of $L/K$. We have $\text{Gal}(L^{G_K}|K) \unlhd S_3$, as the Galois group $G_K$ permutes the roots of the minimal polynomial of a primitive element of $L/K$.
The following two results are general and hold for any characteristic; these give the explicit construction of the Galois closure (Ibid.).
\begin{lemma} \label{resolventl}
Let $L/K$ be a separable extension of degree $3$ with generating equation $$y^3 + e y^2 + fy + g = 0,$$ where $e,f,g \in K$ with roots $\alpha , \beta , \gamma$ in the splitting field $L^s$ of $L$. The two elements of $L^s$ $$\alpha^2 \beta + \beta^2 \gamma + \gamma^2 \alpha \quad\text{and}\quad  \beta^2 \alpha + \alpha^2 \gamma + \gamma^2 \beta$$ are roots of the polynomial $$R(X)= X^2 + (ef -3 g) X + ( e^3 g+ f^3 +9g^2 - 6 efg),$$ which has the same discriminant as $$T(X)=X^3 + e X^2 + fX + g.$$ The quadratic function $R(X)$ is called the \emph{quadratic resolvent} of $T(X)$.
\end{lemma}
\begin{theoreme} \label{resolvent}
Suppose $T(X) \in K[X]$ is a separable irreducible cubic polynomial which defines a cubic field extension $L/K$. Then $L^{G_K}/K$ has automorphism group over $K$ equal to: 
\begin{enumerate}[(a)]
\item $A_3$ if its quadratic resolvent $R(X)$ is reducible over $K$; and
\item $S_3$ if its quadratic resolvent $R(X)$ is irreducible over $K$.\end{enumerate} Furthermore, in case (b), the Galois closure $L^{G_K}$ of $L/K$ is equal to $K(\alpha ,\delta)$, where $\alpha$ is any root of $T(X)$, and $\delta$ is either root of the quadratic resolvent of $T(X)$. If $p \neq 2$, then $L^{G_K}$ is also equal to $K(\alpha ,\sqrt{D})$, where $D$ is the discriminant of $T(X)$. 
\end{theoreme} 
We now suppose that $L/K$ is a separable but non-Galois extension.\\

 \begin{itemize}
 \item[$\circ$] Suppose moreover that the field characteristic is not $2$ or $3$; then either
 \begin{itemize} 
 \item[$\square$] $L/K$ is purely cubic and the Galois closure is equal to $K(y,\xi)$, where $y$ generates a cubic extension whose generating equation is $X^3-a$ and $\xi$ is a primitive third root of unity. 
 $$
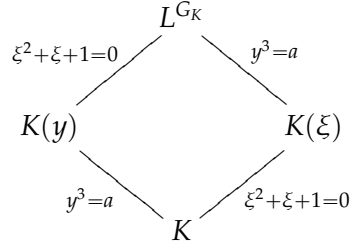
\xymatrix{ & L^{G_K} \ar@{-}[rd]^{y^3=a} \ar@{-}[ld]_{\xi^2+\xi+1=0}& \\ K(y) \ar@{-}[rd]_{y^3=a} && K(\xi) \ar@{-}[ld]^{\xi^2+\xi +1 =0}\\  & K&}$$
 Note that $K(\xi)/K$ is a constant extension, the constant field of $L^{G_K}$ is $\mathbb{F}_q(\xi)$, and $L^{G_K}/K(\xi)$ is a Kummer extension.   \\ 
 
 \item[$\square$] $L/K$ has a primitive element $y$ with minimal polynomial of the form $X^3 -3X-a$, for some $a \in K$ and the Galois closure is $K(y,\gamma)$ with $\gamma^2=-3(a^2-4)$.
  $$\xymatrix{ &L^{G_K} \ar@{-}[rd]^{y^3-3y=a} \ar@{-}[ld]_{\gamma^2=-3(a^2-4)}& \\ K(y) \ar@{-}[rd]_{y^3-3y=a} && K(\gamma)\ar@{-}[ld]^{\gamma^2=-3(a^2-4)}\\  & K&}$$
 \end{itemize} 
 Note that $\gamma^2=-3(a^2-4)$ is a Kummer extension, and thus the ramified places are those $\mathfrak{p}$ such that $(v_{\mathfrak{p}} (-3(a^2-4)),2)=1$. Thus, any place $\mathfrak{p}$ of $K$ which ramifies in $K(\gamma)$ must satisfy $v_{\mathfrak{p}} (a)=0$. Moreover, we will prove later that the ramified places of $y^3-3y=a$ are the places $\mathfrak{p}$ such that $v_{\mathfrak{p}} (a)<0$ and $3 \nmid v_{\mathfrak{p}} (a)$ (see Theorem 5.10). In particular, there are no places fully ramified in the Galois closure $L^{G_K}$.  \\ 
 
\item[$\circ$]  If the field characteristic is equal to $3$, then $L/K$ has a primitive element $y$ whose minimal polynomial is $X^3 -aX-a^2$ for some $a\in K$. Then the Galois closure $L^{G_K}/K$ of $L/K$ is $L(z, \gamma)$ where $\gamma^2 =a$, and $z = \gamma y$ generates a cubic Artin-Schreier equation with generating equation $z^3 -z=\gamma$. (See Theorem \ref{AS})
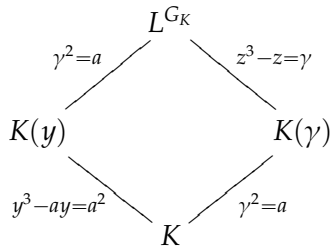
 $$\xymatrix{ & L^{G_K} \ar@{-}[rd]^{z^3-z=\gamma } \ar@{-}[ld]_{\gamma^2=a}& \\ K(y) \ar@{-}[rd]_{y^3-ay=a^2} && K(\gamma) \ar@{-}[ld]^{\gamma^2 =a}\\  & K&}$$
We note that $K(\gamma )/ K$ is Kummer extension, and that $L^{G_K} / K(\gamma)$ is an Artin-Schreier extension.\\ 

\item[$\circ$] If the field characteristic is equal to $2$, then either \\ 

 \begin{itemize} 
 \item[$\square$] $L/K$ is purely cubic and the Galois closure is $L(y,\xi)$, where $y$ generates a cubic extension whose generating equation is $X^3-a$ and $\xi$ is a primitive third root of unity. Indeed, the quadratic resolvent is $X^2+ a X +a^2$ and $\xi a$ is a root of this polynomial.
  $$
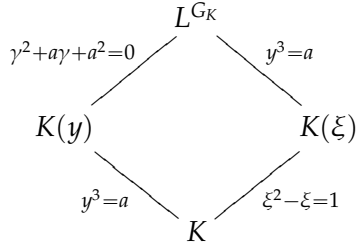
\xymatrix{ & L^{G_K} \ar@{-}[rd]^{y^3=a} \ar@{-}[ld]_{\gamma^2+ a \gamma +a^2=0}& \\ K(y) \ar@{-}[rd]_{y^3=a} && K(\xi) \ar@{-}[ld]^{\xi^2-\xi =1}\\  & K&}$$
  Note again that $K(\xi)/K$ is a constant extension, whence the constant field of $L^{G_K}$ is equal to $\mathbb{F}_q(\xi)$. \\
  
 \item[$\square$] $L/K$ has a primitive element $y$ with minimal polynomial of the form $P(X)=X^3 -3X-a$ for some $a \in K$ and the Galois closure is equal to $L(y,\gamma )$, where $\gamma$ is an Artin-Schreier generator of degree $2$ whose minimal polynomial is $X^2 -X = \frac{1+a^2}{a^2}$. Indeed, the quadratic resolvent of $P(X)$ is equal to $X^2+ aX+ (1+a^2)$.
   $$
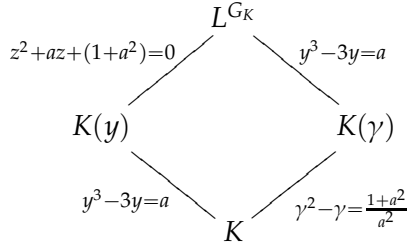
\xymatrix{ & L^{G_K} \ar@{-}[rd]^{y^3-3y=a} \ar@{-}[ld]_{z^2+ a z +(1+a^2)=0}& \\ K(y) \ar@{-}[rd]_{y^3-3y=a} && K(\gamma) \ar@{-}[ld]^{\gamma^2-\gamma =\frac{1+a^2}{a^2}}\\  & K&}$$
We note that $K(\gamma ) /K$ is an Artin-Schreier extension. 
 \end{itemize} 
\end{itemize} 

\section{Cubic extensions in characteristic $3$}  
As we have seen in Section $2$, separable cubic extensions of function fields in characteristic $3$ have primitive elements with minimal equation of the form $T(X) = X^3+bX+b^2$. If the extension $L/K$ is Galois, this extension is Artin-Schreier as is well-known. We reprove this basic result in Galois theory using our previous standard form.
\begin{theoreme} \label{AS}
Let $p= 3$, and let $L/K$ be a Galois extension of degree $3$. Then there is a primitive element $z$ such that its minimal polynomial is of the form $R(z)= z^3 -z -a$. Furthermore, this primitive element is explicitly determined.
\end{theoreme}
\begin{proof} 
By the previous result, we know that there is a primitive element $y'$ such that its minimal polynomial is of the form $$S(X)=X^3 +b X + b^2.$$ The discriminant of such a polynomial is equal to $-4b^3$. As $L/K$ is Galois, the discriminant is a square, thus $-b$ is a square, say $-b= a^2$. With $z= y'/a$, it follows that $z^3 -z =a$. 
\end{proof} 
Although the Artin-Schreier theory is well-known, we will draw certain important parallels to its fundamental results in other types of extensions, particularly those of the form $y^3-3y=a$ which occur when $q\equiv -1 \mod 3$. For the proof of the following classical result, we refer the reader to \cite[Proposition 5.8.6]{Vil}.
\begin{lemma} \label{genAS}
Let $\text{\emph{char}}(\mathbb{F}_q) = 3$, and let $L_i = K(z_i )/K$ $(i = 1, 2)$ be two cyclic extensions of degree $3$ of the form $z_i^3 - z_i = a_i \in K$, $i = 1, 2$. Then the following statements are equivalent:
\begin{enumerate}
\item $L_1 = L_2$.
\item $z_1 = j z_2 + b$ for $1 \leq j \leq 2$ and $b \in K$.
\item $a_1 = ja_2 + (b^3 - b)$ for $1 \leq j \leq 2$ and $b \in K$.
\end{enumerate}
\end{lemma}
The following Theorem gives ramification data, genus and Galois action for Artin-Schreier extensions (see \cite[Proposition 3.7.8]{Sti}). 
\begin{theoreme} \label{ASGal}
Let $K/\mathbb{F}_q$ be an algebraic function field of characteristic $3 $. Suppose that $a \in K $ is an element
which satisfies the condition $$a \neq w^3 - w \ for \ all \ w \in K.$$ Let $L= K(y)$ with $y^3 - y = a$. Such an extension $L/K$ is called an \emph{Artin-Schreier extension} of $K$. We then have:
\begin{enumerate}
\item $L/K$ is a cyclic Galois extension of degree $3$. The automorphisms of $L/K$ are given by $\sigma_l (y) = y + l$ ($l= 0, 1,2$).
\item A place $\mathfrak{p}$ of $K$ is ramified in $L$ if, and only if, there is an element $z \in K$ satisfying
$$v_{\mathfrak{p}} (a - (z^3 - z)) = - m_{\mathfrak{p}} < 0\quad\text{and}\quad m_{\mathfrak{p}} \not\equiv 0 \mod  3.$$ For a place $\mathfrak{p}$ of $K$ ramified in $L$, let $\mathfrak{P}$ denote the unique place of $L$ lying over $\mathfrak{p}$. Then the differential exponent $\alpha(\mathfrak{P}|\mathfrak{p})$ is given by $$\alpha(\mathfrak{P}|\mathfrak{p}) = 2(m_{\mathfrak{p}}+1).$$
\item If at least one place $\mathfrak{Q}$ satisfies $v_{\mathfrak{Q}} (a)>0$, then $\mathbb{F}_q$ is algebraically closed in $L$ and $$g_L = 3 g_K  -2 + \sum_{\mathfrak{p} \text{ ramified in $L$}} ( m_{\mathfrak{p}} + 1) \deg(\mathfrak{p}),$$
where $g_L$ (resp. $g_K$) is the genus of $L/\mathbb{F}_q$ (resp. $K/\mathbb{F}_q$) and $\deg(\mathfrak{p})$ is the degree of the place $\mathfrak{p}$ of $K$.
\end{enumerate}\end{theoreme}
\begin{remarque} Suppose $p=3$. Let $L/K$ be a cubic extension. By Corollary \ref{char3generator}, we know that there is a primitive element $z$ for $L$ with minimal polynomial $T(X) = X^3 -b X -b^2$, with $b \in K$. By the results of Section 2, we know that the Galois closure $L^{G_K}/ K$ can be expressed as a tower $L^{G_K}= K(y, \beta )/K(\beta ) / K$ with $y$ that has minimal polynomial $S(X)= X^3 - X -\beta$ where $\beta^2 =b$. As $[K(y, \beta ):K(\beta )]=3$ is coprime to $[K(\beta ):K]=2$, the ramified places of $L/K$ are the places of $K$ that ramify in $K(y, \beta )/K(\beta )$. These are precisely those places $\mathfrak{p}$ such that there is  place $\mathfrak{P}$ in $K(\beta )$ above $\mathfrak{p}$ and an element $a \in K(\beta )$ such that $$v_{\mathfrak{p}} (\beta - (a^3-a))=-m<0 \quad \text{and} \quad  gcd(m, 3)= 1.$$
\end{remarque} 
If $p = 3$, for a Galois cubic extension $L/\mathbb{F}_q(x)$, it is known (see \cite[Example 5.8.8]{Vil}) via a process originally due to Hasse that a generating equation of the form $P_3(X) = X^3 - X - a$ may be transformed into another generating equation $X^3 - X - c$ where the ramified places of $L/\mathbb{F}_q(x)$ are given by the places $\mathfrak{p}$ of $\mathbb{F}_q(x)$ for which $v_\mathfrak{p}(c) < 0$, and for each such place, $(v_\mathfrak{p}(c),3) = 1$. In this case, the equation $X^3 - X - c$ is said to be in \emph{standard form}. 

We recall when an Artin-Schreier extension is constant.

\begin{theorem} \label{ASconstant} Suppose that $p=3$, $L/\mathbb{F}_q(x)$ is cubic and Galois, it is thus an Artin-Schreier extension with a primitive element $y$ of $L/\mathbb{F}_q(x)$ and generating equation $X^3 - X - c$ in standard form. Then $c \in \mathbb{F}_q$ if, and only if, $L/\mathbb{F}_q(x)$ is a constant extension (whence $L = \mathbb{F}_{q^3}(x)$).
\end{theorem}
\begin{proof}
If $c \in \mathbb{F}_q$, then by definition of the generating equation $y^3 - y - c = 0$, $L/\mathbb{F}_q(x)$ is obtained by adjoining constants. For the converse, suppose that $c \notin \mathbb{F}_q$. Then there would exist a place $\mathfrak{p}$ of $\mathbb{F}_q(x)$ such that $v_\mathfrak{p}(c) < 0$. As the generating equation is in standard form, it follows that $\mathfrak{p} $ is (fully) ramified in $L$. As constant extensions are unramified (see for example \cite[Theorem 6.1.3]{Vil}), it follows that $L/\mathbb{F}_q(x)$ cannot be a constant extension.
\end{proof}

We conclude this section by giving an integral basis for Artin-Schreier extensions \cite[Theorem 9]{MadMad}.
\begin{theoreme}
Let $L$ be an Artin-Schreier extension of $\mathbb{F}_q (x)$ with generating equation $y^3-y = a$, where the factorisation of $a$ in $\mathbb{F}_q (x)$ is given by 
$$a = \frac{Q}{\prod_{i=1}^l P_i^{\lambda_i}},$$ where  $P_i \in \mathbb{F}_q[x]$ and $(\lambda_i ,3)=1$ for each $i=1,\ldots,l$. (This is known to exist; see \cite[Example 5.8.8]{Vil}.) Then $\{1, S_1 y, S_2 y^2\}$ is an integral basis of $L$ over $\mathbb{F}_q(x)$, where for each $j=1,2$,
$$S_j= \prod_{i=1}^l P_i^{r_{i,j}},$$ with $r_{i,j} =1 + \left\lfloor \frac{j \lambda_i}{3} \right\rfloor$ , where $\left\lfloor x \right\rfloor$ denotes the integral part of $x$.
\end{theoreme}

\section{Purely cubic extensions and Kummer extensions} 
In this section, we obtain a criterion for the coefficients of the minimal polynomial of a cubic extension $L/K$ to determine whether or not the extension is purely cubic.
\begin{theoreme} \label{purelycubic} Suppose $p\neq 3$. Let $L/K$ be a cubic extension with constant field $\mathbb{F}_q$ and a primitive element $z$ whose minimal polynomial is 
$$T(X)= X^3 + e X^2 +f X + g $$ where $e,f,g \in K$. If $3eg\neq f^2$, $L/K$ admits an explicitly determined primitive element $y$ such that $y^3  -3 y = a $ where
$$a = -2-\frac{ ( 27 g^2 -9efg +2f^3)^2}{27(3ge-f^2)^3}.$$
Then, $L/K$ is purely cubic if, and only if,  
\begin{itemize} 
\item[$\circ$] either $3eg=f^2$ or $a^2-4$ is a square in $K$, if $char(\mathbb{F}_q) \neq 2$. (The generator is explicitly determined.)
\item[$\circ$] either $eg=f^2$ or the polynomial $X^2 - X = \frac{1}{a^2}$ has a root in $K$, if $char(\mathbb{F}_q) = 2$. \\
\end{itemize} 
More precisely, then for any extension $L/K$ admitting a primitive element $y$ such that $y^3  -3 y = a$ with $a \in K$,
\begin{enumerate}
\item {\large If $p\neq 2$ and $a^2-4$ is a square, say, $a^2-4 = \delta^2$, then $u = y^2 + k y  -2 $ is such that $u^3= \delta^3 k$,  where $ k =  \frac{-a \pm \delta }{2}$.}
\item {\large If $p=2$ and $K= \mathbb{F}_q(x)$,  the polynomial $X^2 - X = \frac{1}{a^2}$ has a root in $K$ if and only if $a= \frac{Z^2}{W(Z+W)}$ for some $W$, $Z\in \mathbb{F}_q[x]$ with $(W, Z)=1$ and $u= y^2 +ky$ is such that $$u^3=  \frac{W^4 (Z+W)^2}{P^3},\quad\text{ where }\quad k = \frac{ W^7(Z+W)^5}{Z^{12}}.$$}
\end{enumerate}
\end{theoreme}
\begin{proof}
From Theorem \ref{linearthree}, we know that either $3eg=f^2$ and $L/K$ is purely cubic, or that $L/K$ admits a primitive element $y$ such that $y^3  -3 y = a $ with $$a= -2-\frac{ ( 27 g^2 -9efg +2f^3)^2}{27(3ge-f^2)^3} \in K.$$
The proof that follows does not use this form of $a$, and is therefore valid for any extension $L/K$ with a primitive element $y$   such that $y^3  -3 y = a $.

We note that any primitive element is of the form $ u=  j y^2 + k y + l  $ for $j,k \in K$ not both equal to $0$ and $l \in K$. Thus, we wish to determine which extensions admit a primitive element $u$ such that $u^3 =b$, for some $b \in K$. We have 
\begin{align*} u^3 = &( j y^2 + k y + l )^3 \\
=& j^3y^6+3kj^2y^5+3j(lj+k^2)y^4+k(k^2+6lj)y^3+3l(lj+k^2)y^2+3l^2ky+l^3\\
=& 9kj^2y^3+(9j^3+3jl^2+3k^2l+9lj^2+9k^2j+3j^2ak)y^2\\
&+(6j^3a+3l^2k+18lkj+3k^3+3j^2al+3jak^2)y+j^3a^2+ak^3+6kalj+l^3 \\ 
=& (9j^3+3jl^2+3k^2l+9lj^2+9k^2j+3j^2ak)y^2\\
&+(27kj^2+6j^3a+3l^2k+18lkj+3k^3+3j^2al+3jak^2)y\\
&+9j^2ak+ak^3+6kalj+l^3+j^3a^2.
\end{align*}
As $1,y, y^2$ form a basis of $L/K$, we then have that
$$\left\{ \begin{array}{llll} 
3j^3+jl^2+k^2l+3lj^2+3k^2j+j^2ak &=& 0 & (1) \\ 
9kj^2+2j^3a+l^2k+6lkj+k^3+j^2al+jak^2&=& 0 & (2)  \\ 
9j^2ak+ak^3+6kalj+l^3+j^3a^2 &=& b & (3) \\ 
\end{array} \right.$$ 
If $j$ were equal to $0$, then from $(1)$, either $k=0$ or $l=0$, whence from $(2)$, $k=0$, and thus that $u$ would not be a primitive element for $L$. Therefore, $j \neq 0$. Without loss of generality, we may suppose that $j=1$. Thus, we obtain
$$\left\{ \begin{array}{cclc} 
3+l^2+k^2l+3l+3k^2+ak &=& 0 &  (1) \\ 
9k+2a+l^2k+6lk+k^3+al+ak^2 &=& 0 & (2) \\ 
9ak+ak^3+6kal+l^3+a^2 &=& b & (3) \\ 
\end{array} \right.$$ 
Evaluating $k\cdot (1) -  (2)$ yields $$(k^3-3k-a)(l+2)=0.$$ As $k^3-3k-a\neq 0$ by irreducibility of the polynomial $T(X)= X^3 -3X-a$ over $K$, it follows that $l = -2$. By substitution in $(1)$, $(2)$, $(3)$ respectively,  we obtain 
$$\left\{ \begin{array}{lccc} 1+ka+k^2&=& 0 & (4)\\ 
k(1+ka+k^2) &=& 0 & (5)\\ 
ak^3-3ak+(-8+a^2) &=& b & (6)\\ 
\end{array} \right.$$ 
Evaluating $a\cdot (5) - (6)$, we obtain
$$a^2k^2+4ak-(-8+a^2)+b =0 \ (7).$$
Via $(7) - a^2\cdot (4)$, we find $$-a(a^2-4)k-2(a^2-4)+b=0.$$ Thus, $$k = \frac{-2(a^2-4)+b}{a(a^2-4)}.$$ 
By substitution of $$k = \frac{-2(a^2-4)+b}{a(a^2-4)}$$ in $(4)$, we have 
\begin{equation}\label{quaddisc} (a^2-4)^3-b(a^2-4)^2-b^2=0.\end{equation} 
This is a quadratic in $b$ with discriminant equal to $$\Delta= a^2 (a^2-4)^3.$$ 
\begin{enumerate}
\item  If $p\neq 2$, then the quadratic \eqref{quaddisc} has a solution, if and only if, $a^2 -4$ is a square, say $a^2 -4 = \delta^2$. Then, $$ b = -\frac{ (a^2-4)^2 \pm a  (a^2-4) \delta }{ 2} ,$$ and thus
$$ k = \frac{-2(a^2-4)+b}{a(a^2-4)} = \frac{-a \pm \delta }{2} $$
Note that 
$$b =   (a^2-4) \delta k .$$ 
\item If $p=2$ this quadratic, then \eqref{quaddisc} has a solution, if and only if, $$a^6- a^4X -X^2 =0$$ has a root $b$ in $K$, which is equivalent to that $X^2 - X - \frac{1}{a^2}$ has a root $c$ in $K$ with $c= a^2 b$. Now, we write 
$c= \frac{C}{D}$ and $a = \frac{P}{Q}$ with $P , Q , C, D \in  \mathbb{F}_q [x]$ and $(P,Q)=1$ and $(C,D)=1$, then we have 
$$\left( \frac{C}{D} \right)^2 + \frac{C}{D}= \frac{C^2 + CD}{D^2}=\frac{C(C + D)}{D^2} =\frac{ Q^2 }{P^2}$$ 
As $(C,D)=1$, we can assume without loss of generality that $ C(C+D) = Q^2$ and $D^2=P^2$. As $p=2$, this implies that $ D=P$ and $C(C+P)=Q^2$. The polynomials $C$ and $C+P$ are coprime as $(C,D)=1$; as $D=P$ and $\mathbb{F}_q[x]$ is a prime factorisation domain, we have that $C$ and $C+P$ are therefore square, and thus that there exist $W$ and $V \in \mathbb{F}_q$ with $(W,V)=1$ such that $C= W^2$ and $C+P= V^2$. As a consequence, $P = W^2 + V^2$ and $WV = Q$. Finally, putting $Z= W+V$, we get 
$$P= Z^2 , \quad Q=W(Z+W),  C=W^2 \quad and \quad D=P,$$
with $(Z,W)=1$. Hence, 
$$c= \frac{W^2}{Z^2},$$
$$b = \frac{c}{a^2} = \frac{W^2 Q^2}{Z^2 P^2}= \frac{W^4 (Z+W)^2}{P^3}$$
and 
$$ k = \frac{b}{a^3}= \frac{ W^7(Z+W)^5}{P^6}.$$
By construction, if $a$ is of this form, then the polynomial $X^2 - X - \frac{1}{a^2}$ has a root $c=\frac{W^2}{Z^2}$. The Theorem follows.
\end{enumerate} 
\end{proof} 
The following Corollary is also well known. We choose to use the previous Theorem to reprove it directly.
\begin{corollaire} \label{kummer}
Suppose that $K$ contains a primitive third root of unity. (This is equivalent to $q\equiv 1 \mod 3$.) Then a geometric cubic extension $L/K$ is Galois, if and only if, $L/K$ purely cubic extension. In this case, $L/K$ is called a \emph{Kummer extension}.
\end{corollaire}
\begin{proof} 
Suppose that $K$ contains a primitive third root of unity $\xi$. If $L/K$ is purely cubic, then we may find a primitive element $y$ such that its minimal polynomial is $y^3 = a$, for some $a \in K$. Clearly, $y$, $\xi y$, and $\xi^2 y$ are the roots of the minimal polynomial of $y$, and they are all contained in $L$. Thus, $L/K$ is Galois and $\text{Gal}(L/K) = \mathbb{Z}/3 \mathbb{Z}$.

Suppose now that $L/K$ is Galois, and let $y$ be a primitive element of $L/K$ with minimal equation $y^3 + e y^2 + fy +g=0$. By Theorem \ref{linearthree}, if $3eg= f^2$, then $L/K$ is purely cubic. Suppose then that $3eg\neq f^2$. Thus, there exists a primitive element $z$ with minimal polynomial $T(X) = X^3 - 3 X -a$. The discriminant of this polynomial is equal to $\Delta=-27(-4+a^2)$. 
\begin{itemize}
\item[$\circ$] If $p \neq 2$, then as the extension $L/K$ is Galois, the discriminant $\Delta$ is a square in $K$. As $K$ contains a primitive third root of unity, $-3$ is a square in $K$, 
whence $(a^2-4)$ must also be a square in $K$, and by the previous Theorem, it follows that $L/K$ is purely cubic. 
\item[$\circ$] If $p =2$, then the quadratic resolvent of $T(X)$ is equal to $X^2 +aX + (1+a^2)$. As $L/K$ is Galois, this polynomial has a root in $K$, say, $\delta$. From the previous Theorem, $L/K$ is purely cubic if, and only if, $X^2 - X - \frac{1}{a^2}$ has a root in $K$, which is equivalent to $S(X)=X^2+aX+1$ having a root in $K$. Note that there is a root of $S(X)$ of the form $\delta +u$, as then \begin{align*} S(\delta +u ) &= (\delta +u)^2 + a (\delta +u) + 1 \\&= (u + \delta)^2+ a (\delta +u) +1   \\&= u^2 +au^2 +a^2.\end{align*} The element $u=\xi a$ is a solution to this equation, where $\xi$ is a root of unity, whence $L/K$ is purely cubic. 
\end{itemize} 
\end{proof} 
\begin{corollaire}\label{purely} Let $p\neq 3$.
 A purely cubic extension $L/K$ is Galois if, and only if, $K$ contains a primitive third root of unity.
\end{corollaire} 
\begin{proof} 
Suppose that $L/K$ is purely cubic, so that $y^3 = \alpha$ for some $\alpha \in K$.
\begin{enumerate} \item Case 1: $p\neq 2$. By Lemma \ref{disc}, $L/K$ is Galois, if and only if, $d_{L/K} =-27\alpha^2$ equal to a square in $K$. This is equivalent to $-3$ being a square in $K$, which in turn is equivalent to $K$ containing a primitive third root of unity. 
\item Case 2: $p = 2$. By Lemma \ref{resolvent}, the extension $L/K$ is Galois if, and only if, the resolvent polynomial $R(X)= X^2 + \alpha  X + \alpha^2$ is reducible, which is true if, and only if, $X^2 +X +1$ is reducible. That is, $K$ contains a primitive third root of unity.
\end{enumerate}
Thus, in either case, the result follows.
\end{proof}
We will give the equivalent form of the next result later for extensions with generating equation $y^3 - 3y - a = 0$. For this reason, we recall some of the well-known results from the theory (for the proof, see \cite[Proposition 5.8.7]{Vil}).
\begin{lemma} \label{genK}
Let $q \equiv 1 \mod 3$. Let $L_i = K(z_i )$ ($i = 1, 2$) be two cyclic extensions of $K$ of degree 3,
given by generating equations $z_i^3 = a_i$ . The following statements are equivalent:
\begin{enumerate}
\item $L_1 = L_2$.
\item $z_1 = z_2^jc$ for all $1 \leq j \leq 2$ and $c \in K$.
\item $a_1 = a_2^jc^3$ for all $1 \leq j \leq 2$ and $c \in K$.
\end{enumerate}
\end{lemma}
The next Theorem gives ramification data, genus and Galois action for Kummer extensions \cite[Proposition 3.7.3]{Sti}.
\begin{theoreme} \label{KGal}
Let $K/\mathbb{F}_q$ be an algebraic function field of characteristic $p > 0$ with $q \equiv 1 \mod 3$. Suppose that $a \in K $ is an element which satisfies the condition
$$a\neq  w^3, \ for \ all \ w \in K. $$
Let $L= K(y)$ with $y^3 = a$, so that $L/K$ is a Kummer extension. We then have:
\begin{enumerate}
\item $L/K$ is a cyclic Galois extension of degree $3$. The automorphisms of $L/K$ are given by $\sigma (y) = \xi y$, with $\xi $ a primitive $3^{rd}$ root of unity.
\item A place $\mathfrak{p}$ of $K$ is ramified in $L/K$ if, and only if, $(v_{\mathfrak{p}} (a ),3)=1$. For a place $\mathfrak{p}$ of $K$ ramified in $L$, denote by $\mathfrak{P}$ the unique place of $L$ lying over $\mathfrak{p}$. Then the differential exponent $d(\mathfrak{P}|\mathfrak{p})$ is given by
$$d(\mathfrak{P}|\mathfrak{p}) = 2.$$
\item If at least one place $\mathfrak{Q}$ satisfies $v_{\mathfrak{Q}} (a)>0$, then $\mathbb{F}_q$ is algebraically closed in $L$, and $$g_L = 3 g_K  -2 + \sum_{(v_\mathfrak{p}(a),3)=1}  deg(\mathfrak{p}),$$
where $g_L$ (resp. $g_K$) is the genus of $L/\mathbb{F}_q$ (resp. $K/\mathbb{F}_q$).
\end{enumerate}
\end{theoreme}
We recall when an Kummer extension is constant.
\begin{theorem} \label{Kconstant} Suppose that $q \equiv 1 \text{ mod } 3$, and that $L/\mathbb{F}_q(x)$ is cubic and Galois, it is thus a Kummer extension and there is a primitive element $y$ of $L/\mathbb{F}_q(x)$ having irreducible polynomial $T(X)=X^3 -b$. Then, $L/\mathbb{F}_q(x)$ is a constant extension (whence $L = \mathbb{F}_{q^3}(x)$) if, and only if, $b = u \beta^3$ with $u\in \mathbb{F}_q$, $u$ is not a cube and $\beta \in \mathbb{F}_q(x)$.
\end{theorem}
\begin{proof}
If $b = u \beta^3$ with $u\in \mathbb{F}_q$ where $u \in \mathbb{F}_q$ is not a cube and $\beta \in \mathbb{F}_q(x)$, then $z=y/\beta$ generates $L/K$ and $z^3 = u$, by definition of the constant field and the generating equation $z^3 = u$ (Note that if $u$ was a cube then $X^3+u$ would not be irreducible), $L$ is obtained by adjoining constants to $\mathbb{F}_q(x)$, and $L/\mathbb{F}_q(x)$ is constant. Conversely, suppose that $b\neq u \beta^3$ with $u\in \mathbb{F}_q$ where $u$ is not a cube and $\beta \in \mathbb{F}_q(x)$. 
Then, as  the polynomial is irreducible, there exists a place $\mathfrak{p}$ of $\mathbb{F}_q(x)$ such that $(v_\mathfrak{p}(b),3) = 1$. By Kummer theory \cite[Theorem 5.8.12]{Vil}, it follows that $\mathfrak{p}$ is (fully) ramified in $L$. As constant extensions are unramified (see for example \cite[Theorem 6.1.3]{Vil}), we find that $L/\mathbb{F}_q(x)$ cannot be constant. 
\end{proof}
We finish this section by giving an integral basis for Kummer extensions (\cite[Theorem 3]{MadMad}).
\begin{theoreme}
Let $L$ be an Kummer extension of $\mathbb{F}_q (x)$, and let $y^3 = a$, where the factorisation of $a$ is given by $$a = \prod_{i=1}^l P_i^{\lambda_i},$$ where $P_i \in \mathbb{F}_q[x]$, $1\leq \lambda_i \leq 2$ for each $i=1,\ldots,l$.  
(This is known to exist;  see \cite[Example 5.8.9]{Vil}.) Then $\{1, \frac{y}{S_1},  \frac{y^2}{S_2}\}$ is an integral basis of $L$ over $\mathbb{F}_q(x)$, where for each $j=1,2$,
$$ S_j= \prod_{i=1}^l P_i^{r_{i,j}},$$ with $r_{i,j} = \left\lfloor \frac{j \lambda_i}{3} \right\rfloor$, where $ \left\lfloor x \right\rfloor$ denotes the integral part of $x$.
\end{theoreme}
\section{Extensions with generating equation $y^3 -3y-a=0$}
\subsection{The Galois criterion}
In the next Theorem, we investigate when a cubic extension with generating equation $y^3-3y -a=0$ is Galois, similarly to Theorem \ref{AS} and Corollary \ref{kummer} for Artin-Schreier and Kummer extensions. 
\begin{theorem}\label{qneq1mod3Galois} Let $q \equiv -1 \mod 3$. Then,  a cubic extension $L/\mathbb{F}_q(x)$ is Galois if, and only if,
\begin{enumerate}[(a)] 
\item  If $p\neq 2$, $L/\mathbb{F}_q(x)$ has a primitive element $z$ with minimal polynomial of the form 
$$T(X) = X^3 -3 X - b,$$ where $b = \frac{P}{Q}$, for some $P,Q \in \mathbb{F}_q[x]$ such that $(P,Q) = 1$, and
$$\left\{ \begin{array}{lll} P&=& 2(A^2- 3^{-1} B^2)\\
Q &=& A^2+ 3^{-1} B^2 \end{array}\right.$$ 
for some $A,B \in \mathbb{F}_q[x]$, $(A,B)=1$.

\item If $p=2$, there exists a primitive element $z$ of $L/\mathbb{F}_q(x)$ with minimal polynomial 
$$T(X)=X^3-3X-b,$$ where $b = \frac{P}{Q}$, for some $P,Q \in \mathbb{F}_q[x]$ such that $(P,Q) = 1$, and
$$\left\{ \begin{array}{lll} P&=& A^2\\
Q &=& A^2 + A B+ B^2 
,\end{array}\right.$$ 
for some $A$ and $B\in \mathbb{F}_q[x]$ and $(A,B)=1$.
\end{enumerate} 
\end{theorem}
 \begin{proof} 
 Let $q \equiv -1 \mod 3$. 
 \begin{enumerate}[(a)] 
\item Suppose that $L/\mathbb{F}_q(x)$ is Galois and $p \neq 2$. By Corollaries \ref{linear} and \ref{purely}, there exists a primitive element $z$ with minimal polynomial of the form $T(X) = X^3 -3  X - b$. It follows that the discriminant of $L/\mathbb{F}_q(x)$ satisfies $d_{L/\mathbb{F}_q(x)} = -27(-4+b^2)$. As in the statement of the Theorem, we write $b = P /Q$ with $P,Q \in \mathbb{F}_q[x]$ and $\gcd(P, Q) =1$. By Lemma \ref{disc}, $L/\mathbb{F}_q(x)$ is Galois if, and only if, there exists $R \in \mathbb{F}_q[x]$ such that $$R^2 = -27(-4Q^2 + P^2) = -27(-2Q + P)(2Q + P).$$  The polynomials $-2Q + P$ and $2Q + P$ are relatively prime; indeed, for if $f$ divides $-2Q+ P$ and $2Q+ P$, then $f$ divides both $4Q = (2Q+ P) -(-2Q + P)$ and $2P = (2Q + P) + (-2Q + P)$, contradicting $\gcd(P, Q) =1$. Thus, by unique factorisation in $\mathbb{F}_q[x]$, it follows that, up to elements of $\mathbb{F}_q^*$, $2Q + P$ and $-2Q + P$ are squares in $\mathbb{F}_q[x]$. Therefore, there exist $c, d \in \mathbb{F}_q^*$ with $cd$ equal to $-27^{-1}$ up to a square in $\mathbb{F}_q^*$ and $A, B\in \mathbb{F}_q[x]$ such that $$2Q + P= c A^2 \quad\quad\text{and}\quad\quad -2Q + P= d B^2.$$ Thus $2P= c A^2+ d B^2$ and $4Q =cA^2-  d B^2$, and $$4Q= cA^2-  d B^2= c^{-1} (c^2A^2-  cd B^2) = c^{-1}  ((cA)^2+3^{-1} (3^{-1}B)^2).$$ Therefore,
$$\left\{ \begin{array}{ccc} 4cQ  &= & A'^2+ 3^{-1} B'^2\\ 
2cP&= & A'^2- 3^{-1} B'^2,
\end{array} \right.$$
where $A'= cA$ and $B' = 3^{-1} B$.  Given $A',B'$ as before, for any $c \in \mathbb{F}_q$, $b$ takes the same value 
$$b=2\frac{A'^2- 3^{-1} B'^2}{A'^2+ 3^{-1} B]^2}.$$ Thus, without loss of generality, we have 
$$\left\{ \begin{array}{ccc} Q  &= & A''^2+ 3^{-1} B''^2\\ 
P&= & 2 (A''^2- 3^{-1} B''^2),
\end{array} \right.$$
for some $A''$ and $B''\in \mathbb{F}_q[x]$ with $(A'', B'' )=1$. 

Conversely, suppose that
$$\left\{ \begin{array}{ccc} Q &=&  A^2+ 3^{-1} B^2\\ 
P & =& 2( A^2- 3^{-1} B^2)\end{array} \right.$$ 
for some $A$ and $B $ in $\mathbb{F}_q[x]$, and $(A,B)=1$. Then \begin{align*} d_{L/\mathbb{F}_q(x)} &= -27 (-4+b^2)\\ 
&=-27 \left(-4 + 4\left( \frac{A^2- 3^{-1} B^2}{A^2+ 3^{-1} B^2}\right)^2\right)\\
& = -3 \times 6^2 \left( \frac{-(A^2+ 3^{-1} B^2)^2+ (A^2- 3^{-1} B^2)^2}{(A^2+ 3^{-1} B^2)^2}\right) \\
&=  -3 \times 6^2\left( \frac{-4\times 3^{-1} A^2 B^2}{(A^2+ 3^{-1} B^2)^2}\right) \\
&= 12^2\left(\frac{A^2 B^2}{(A^2+ 3^{-1} B^2)^2}\right),
\end{align*} whence $d_{L/\mathbb{F}_q(x)}$ is a square. By Lemma \ref{disc}, it follows that $L/\mathbb{F}_q(x)$ is Galois.
\item Suppose that $L/\mathbb{F}_q(x)$ is Galois, and that $p=2$. By Lemma \ref{quadratic} and Corollary \ref{purely}, there is a primitive element $z$ of $L/\mathbb{F}_q(x)$ with minimal polynomial of the form $T(X)= X^3 -3 X -b$. By Lemma \ref{resolvent}, we know that $L/\mathbb{F}_q(x)$ is Galois if, and only if, the resolvent polynomial $$R(X)= X^2 +3 b X + (-27+9b^2)$$ of $T(X)$ is reducible. This is the same as requiring that the polynomial $X^2 + X = 3/b^2-1$ is reducible, which (as this is a polynomial of degree 2) is equivalent to the existence of at least one $w \in K$ such that $w^2-w=3/b^2-1 = 1/b^2 +1$. As $p=2$, we have $(w+1/b) ^2 - (w+1/b) =1 + 1/b$. The latter is equivalent to the existence of $w' \in \mathbb{F}_q(x)$ such that 
$$w'^2 - w' =\frac{ b+1}{b}.$$ We write $b = \frac{P}{Q}$ where $P$, $Q\in \mathbb{F}_q[x]$ with $(P, Q)=  1$ and $w' =\frac{ C}{D}$ where $C$, $D\in \mathbb{F}_q[x]$ with $(C, D)=  1$. We thus find that $$ \frac{ C^2}{D^2} - \frac{ C}{D}= \frac{ P + Q}{P}$$ and 
$$ \frac{ C^2 -C D}{D^2} = \frac{ P + Q}{P}.$$ As $(C^2 -C D, D^2) = 1$ and $(P + Q, P)=1$, it follows that up to a constant not affecting the value of $b$, 
$$\left\{ \begin{array}{rcc} P &=& D^2\\ 
P + Q&=& C^2 -C D ;\end{array}\right.$$
equivalently,
$$\left\{ \begin{array}{rcl} P &=& D^2\\ 
 Q   &=& C^2 -C D +D^2. \end{array}\right.$$
Conversely, suppose $b= \frac{P}{Q}$, where
 $$\left\{ \begin{array}{rcl} P &=& D^2\\ 
 Q   &=& C^2 -C D +D^2 \end{array}\right.$$
where $C$, $D\in \mathbb{F}_q[x]$ with $(C, D)=  1$. Then the polynomial $X^2 - X=\frac{ b+1}{b}$ has $\frac{C}{D}$ as root and the Theorem follows.
\end{enumerate} 
\end{proof} 

\begin{remarque} \label{constantsremark} We note that under the assumptions and notation of the previous Theorem, $Q$ is constant if, and only if, $b$ is constant. 
Indeed, if $Q$ is constant, then with the notation of the previous Theorem, we write $Q=A^2+ 3^{-1} B^2$. Letting $$A = a_m x^m + a_{m-1} x^{m-1} + \cdots + a_0 \quad\quad\text{and}\quad\quad B = b_n x^n + b_{n-1} x^{n-1} + \cdots + b_0,$$ it follows from the fact that $Q$ is constant that if either of $m$ or $n$ is positive, then $m=n$.  This implies that $ a_n^2 +3^{-1} b_n^2 = 0$, and hence that $(-3)^{-1} = (b_n/a_n)^2$ where $b_n / a_n\in \mathbb{F}_q$, contradicting that $q \equiv -1 \mod 3$. \end{remarque} 
More precisely, we have the following Theorem, which characterises the Galois extensions of the form $y^3-3y=a$ in terms of the factorisation of the denominator of $a$. 
\begin{theoreme} \label{even}
Let $q \equiv -1 \mod 3$. Suppose that $L/\mathbb{F}_q(x)$ is a Galois cubic extension. By Theorem \ref{linear}, $L/\mathbb{F}_q(x)$ has a primitive element $z$ with minimal polynomial of the form $$T(X) = X^3 -3 X - b.$$ Write $b = \frac{P}{Q}$ with $P,Q \in \mathbb{F}_q[x]$ such that $(P,Q) = 1$.  Then, $$Q=w \prod_i Q_i ^{e_i}$$ with $Q_i $ unitary irreducible of even degree, $w \in \mathbb{F}_q^*$, and $e_i$ a positive integer. 
\end{theoreme} 
\begin{proof} By \cite[Theorem 3.46]{LiNi}, an irreducible polynomial over $\mathbb{F}_q$ of degree $n$ factors over $\mathbb{F}_{q^k}[x]$ into $\gcd(k,n)$ irreducible polynomials of degree $n/\gcd(k,n)$. It follows that an irreducible polynomial over $\mathbb{F}_{q}$ factors in $ \mathbb{F}_{q^2}$ into polynomials of smaller degree if, and only if, $\gcd(2,n) > 1$, i.e., $2 | n$. We note that the rings $\mathbb{F}_q[x]$ and $\mathbb{F}_{q^2}[x]= \mathbb{F}_q(u)[x]$ for some $u \in \mathbb{F}_{q^2}\backslash \mathbb{F}_q$ are both unique factorisation domains. Also, the irreducible polynomials in $\mathbb{F}_{q^2}[x]$ are those irreducible polynomials of odd degree in $\mathbb{F}_q[x]$ or the irreducible polynomials occurring as factors of irreducible polynomials of even degree in $\mathbb{F}_q[x]$.
\begin{enumerate} 
\item If $p\neq 2$, let $t$ be a root of $x^2 +3^{-1}$ in $\mathbb{F}_{q^2}$. From the previous Theorem, we know that, $$Q = A^2+ 3^{-1} B^2,$$ where $A,B \in \mathbb{F}_q[x]$ and $(A,B)=1$. We consider the norm map 
$$\begin{array}{cccl} N : & \mathbb{F}_{q^2}[x]&  \rightarrow & \mathbb{F}_q[x]\\ 
& \alpha +t \beta & \mapsto & \alpha^2+3^{-1} \beta^2 =(\alpha+ t \beta )(\alpha- t \beta )=(\alpha + t \beta)^{q+1}.\end{array}$$  
As, the element  $-t=t^{q}$ is the other root of the polynomial $X^2+3^{-1}=0$ in $\mathbb{F}_{q^2}$. Indeed, the coefficients of $S(X) = X^2+3^{-1}$ are in $\mathbb{F}_{q}$, whence 
$$S(t^q) = S(t)^q =0,$$ and $t\neq t^q$, as $t\notin \mathbb{F}_{q}$. 

Clearly, the map $N$ is multiplicative, and it follows that a polynomial $Q=w \prod_i Q_i ^{e_i}$ is of the form $N(A + t B)= Q$ with $A,B \in \mathbb{F}_q[x]$ coprime if, and only if, each $Q_i$ is of the form $N(A_i + t B_i)= Q_i$, with $A_i,B_i \in \mathbb{F}_q[x]$ coprime. \\
Indeed, observe that
\begin{itemize} 
\item[$\circ$] If $U$ is an irreducible polynomial of odd degree in $\mathbb{F}_q[x]$, then
$$N(U)= Q_i^2,$$
\item[$\circ$] If $U$ is an irreducible polynomial occurring as a factor of an irreducible polynomial of even degree in $\mathbb{F}_q[x]$, then it is of the form  $U= A+ t B$ with $A, B \in \mathbb{F}_q[x]$, $B \neq 0$ and $(A,B)=1$. Then
$$N(U)= A^2 + (-3)^{-1} B^2,$$
\end{itemize}
Thus, $N(A_i + t B_i)= Q_i$, with $A_i,B_i \in \mathbb{F}_q[x]$ coprime if only if $A_i + t B_i$ is an irreducible polynomial occurring as a factor of an irreducible polynomial of even degree in $\mathbb{F}_q[x]$. We therefore conclude the Theorem if $p\neq 2$.
\item If $p=2$, similarly, let $\xi$ be a primitive $3^{rd}$ root of unity in $\mathbb{F}_{q^2}$, whence $\xi^2 + \xi +1 =0$. From the previous Theorem, we know that, $$Q = A^2+AB+  B^2,$$ where $A,B \in \mathbb{F}_q[x]$ and $(A,B)=1$. We consider now the norm map 
$$\begin{array}{cccl} N: & \mathbb{F}_{q^2}[x]&  \rightarrow & \mathbb{F}_q[x]\\ 
& \alpha +\xi \beta & \mapsto & \alpha^2+ \alpha \beta + \beta^2 = ( \alpha + \xi \beta) ( \alpha + \xi^2 \beta) =( \alpha + \xi \beta)^{q+1} \end{array}$$  
Indeed, $\xi^2 = \xi^q$, as it is a root distinct from $\xi$ of the polynomial $X^2 + X + 1$, which has coefficients in $\mathbb{F}_q$.
As before, the map $N$ is multiplicative. 
We thus deduce the Theorem as in the previous case. 
\end{enumerate}
\end{proof} 
Any $Q$ of the form given in the previous Theorem is realisable as the denominator of $b$. Here, we give a recipe to construct Galois cubic extensions without primitive $3^{rd}$ roots of unity with a given $Q$ as in the previous Theorem. We begin with the following Lemma.
\begin{lemme} Suppose $q\equiv -1 \mod 3$. Given $w \in \mathbb{F}_q$, there are $q+1$ ways to write $w$ as 
\begin{enumerate}
\item $w= u^2 + (-3)^{-1} v^2$, for some $u , v\in \mathbb{F}_q$, if $p \neq 2$;
\item $w= u^2 + uv + v^2$, for some $u , v\in \mathbb{F}_q$, if $p=2$. 
\end{enumerate}
\end{lemme} 
\begin{proof} 
\begin{enumerate}
\item If $p\neq 2$, as in the previous Theorem, let $t$ be one of the solutions of the equation $X^2+3^{-1}=0$ in $\mathbb{F}_{q^2}$.
As, $-t= t^q$, we have,
$$w = (u+ t v)(u-tv) =(u+ t v)(u+t^qv) = (u+t v)^{q+1}\quad\text{and}\quad 1= w^{q-1}= (u+t v)^{q^2-1}.$$ 

If $w=1$ then $u+tv$ is a $q+1^{th}$ root of unity and there are $q+1$ of those in $\mathbb{F}_{q^2}$. Otherwise, $u+tv$ is a $q^2-1^{st}$ root of unity, say $u+tv= \zeta^m$ where $\zeta$ be a primitive $q^2-1^{st}$ root of unity and $m$ a positive integer. 
Also, $w=\zeta^{l}$ for some $l$ positive integer; moreover, $w^{q-1}=1$, whence $m(q-1)$ is divisible by $q^2-1$. It follows that $l = s(q+1)$ for some positive integer $s$, and $m=s, \ s+ (q-1),\ \ldots,\ s+q(q-1)$ give $q+1$ distinct $u+tv = \zeta^m$. Hence $q+1$ distinct ways to write $w$ as $w= N(u+ t v)=u^2+3^{-1} v^2$. 
\item If $p=2$, let $\xi$ be one of the solutions of the equation $X^2+X+1=0$ in $\mathbb{F}_{q^2}$. Then $\xi^2=\xi +1 = \xi^q$ is also the other root of the polynomial $X^2 +X+1$ thus the argument developed for $p\neq2$ can be applied in this case too. 
\end{enumerate}
\end{proof} 
\begin{lemme}\label{decomQ}
 Let $$Q=w \prod_{i=1}^r Q_i ^{e_i}$$ with $Q_i $ unitary irreducible in $\mathbb{F}_q[x]$ of even degree, $w \in \mathbb{F}_q^*$, and $e_i$ a positive integer. Then there are at most  $2^{r+1} (q+1)$, $(A,B) \in \mathbb{F}_q[x] \times \mathbb{F}_q[x]$ relatively prime such that 
 \begin{enumerate}
\item $Q= A^2 + (-3)^{-1} B^2$, for some $A , B\in \mathbb{F}_q$, if $p \neq 2$;
\item $Q= A^2 + AB + B^2$, for some $A , B \in \mathbb{F}_q$, if $p=2$. 
\end{enumerate}
\end{lemme} 
\begin{proof}
The degree of $Q_i$ is even, whence by \cite[Theorem 3.46]{LiNi}, each $Q_i$ factors in $\mathbb{F}_q(t)[x]$ into $gcd(2,n) = 2$ irreducible polynomials of degree $n/gcd(2,n) = n/2$. 
\begin{enumerate} 
\item Suppose that $p\neq 2$. 
We have that $$Q_i  = (A_i + t B_i) (A_i - t B_i) = A_i^2 + 3^{-1} B_i^2,$$ The factors of $Q_i$ are then $(A_i+tB_i)$ and $(A_i - tB_i)$, which are unique up to a unit. Also, as $Q_i$ is irreducible in $\mathbb{F}_q [x]$, we have that $A_i$ and $B_i$ are coprime. Suppose that $Q_i = N(C_i + tD_i) = N(A_i + tB_i)$. As $\mathbb{F}_{q^2}[x]$ is a unique factorisation domain, we have $$ C_i + t D_i= a_i (A_i \pm tB_i)\quad \text{ and } \quad C_i - t D_i= b_i (A_i \mp tB_i)$$
for some $a_i , b_i \in \mathbb{F}_{q^2}=\mathbb{F}_q(t)$, $a_ib_i=1$ with $A_i$ and $B_i\in \mathbb{F}_{q}[x]$. 
 We write $a_i= u_i +t v_i $ and $b_i = w_i + t z_i$ where $u_i, v_i , w_i, z_i \in \mathbb{F}_{q}$. As $a_i b_i =(u_i w_i -3^{-1} v_i z_i ) +t (v_i w_i + u_iz_i) =1$, it follows that
 $$\left\{ \begin{array}{ccc} 
1&=& u_iw_i -3^{-1} v_iz_i \\ 
0 &=&v_iw_i+ u_i z_i 
\end{array}\right. $$
 Thus, 
$$\left\{ \begin{array}{ccc} 
C_i + tD_i &=& (u_i + tv_i)(A_i \pm tB_i) = u_iA_i \mp 3^{-1}v_iB_i + (A_iv_i\pm B_iu_i)t\\
C_i - tD_i & =& (w_i + tz_i)(A_i \mp tB_i) = w_iA \pm 3^{-1}z_iB + (Az_i\mp Bw_i)t
\end{array}\right. $$
As a consequence, 
$$\left\{ \begin{array}{ccc}
C_i = u_iA_i \mp 3^{-1}v_iB_i & and & D_i = A_i v_i \pm B_i u_i \\
C_i = wA_i \pm 3^{-1}zB_i & and &  D_i = -(A_i z_i \mp B_i w_i )
\end{array}\right. $$
Subtracting $C_i = u_iA_i \mp 3^{-1}v_iB_i $ and $C_i = w_iA_i \pm 3^{-1}z_iB_i $, we find 
$$(u_i-w_i) A_i = \mp 3^{-1} (v_i +z_i) B_i $$
As $(A_i,B_i)=1$, $u_i=w_i$ and $v_i=-z_i$, whence
$$1 = u_i^2 +3^{-1} v_i^2.$$
Thus, we can write $Q_i^{e_i}$ as an element in the image of the norm map $N$, that is, $$Q_i^{e_i}=N\left( \prod_{l=1}^{e_i} ( u_{i,l}+ t v_{i,l}) \prod_{l=1}^{e_i} (A_i+ \epsilon_{i,l} t B_i) \right),$$
where $u_i$ and $v_i \in \mathbb{F}_q$ satisfy $1 = u_{i,l}^2 +3^{-1} v_{i,l}^2 $ and $\epsilon_{i,l}\in \{\pm 1\}$. Moreover, if in this product we have at least one $\epsilon_{i,l_0} =1$ and at least one $\epsilon_{i,l_1} =-1$ for some $l_0 \neq l_1$, then it would lead to $A$ and $B$ sharing a divisor, which violates that they are coprime. It follows that the only possibility is to write either $$Q_i^{e_i}=N\left(  \left(\prod_{l=1}^{e_i} ( u_{i,l} + t v_{i,l})\right) (A_i + t B_i)^{e_i} \right)\quad\text{or}\quad Q_i^{e_i}=N\left(\left( \prod_{l=1}^{e_i} ( u_{l,i} + t v_{l,i}) \right) (A_i - t B_i)^{e_i} \right)$$ By the previous Lemma, we write $w$ as $w=N(\mu + t \nu)$ for some $\mu $ and $\nu \in \mathbb{F}_q$, and
$$ Q = N\left( ( \mu +t \nu) \prod_{i=1}^r \left( \prod_{l=1}^{e_i} ( u_{i,l} + t v_{i,l}) \right)(A_i + \epsilon_i t B_i)^{e_i} \right),$$
where $\epsilon_i = \pm 1$ with $\mu , \nu , u_{i,l} , v_{i,l} \in \mathbb{F}_q$. We have $ ( \mu +t \nu) \prod_{i=1}^r  \prod_{l=1}^{e_i} ( u_{i,l} + t v_{i,l}) ) = \alpha + t \beta $ for some $\alpha, \beta \in \mathbb{F}_q$ and $w = N(\alpha + t\beta)$. By Lemma 5.4, there exist $q+1$ such $\alpha$ and $\beta$. Furthermore, $$A+ tB= (\alpha + t \beta)\prod_{i=1}^r   (A_i + \epsilon_i t B_i)^{e_i}.$$ The result follows.
\item Suppose that $p= 2$. We have that $$Q_i  = (A_i + \xi B_i) (A_i + \xi^2 B_i) = A_i^2 + A_i B_i + B_i^2,$$ The factors of $Q_i$ are then $(A_i+\xi B_i)$ and $(A_i +\xi ^2 B_i)$ unique up to a unit. Also, as $Q_i$ is irreducible in $\mathbb{F}_q$, we have that $A_i$ and $B_i$ are coprime. Suppose that $Q_i = N(C_i + \xi D_i) = N(A_i + \xi B_i)$. As $\mathbb{F}_{q^2}[x]$ is a unique factorisation domain, we have 
$$ C_i + \xi D_i= a_i (A_i + \xi^2 B_i)\quad \text{ and } \quad C_i + \xi^2 D_i= b_i (A_i + \xi B_i)$$
or
$$ C_i + \xi D_i= a_i (A_i + \xi B_i)\quad \text{ and } \quad C_i + \xi^2 D_i= b_i (A_i + \xi^2 B_i)$$
for some $a_i , b_i \in \mathbb{F}_{q^2}=\mathbb{F}_q(t)$, such that $a_ib_i=1$, as $Q_i$ is supposed unitary. We write $a_i= u_i +\xi  v_i $, $b_i = w_i + \xi z_i$ where $u_i, v_i , w_i, z_i \in \mathbb{F}_{q}$. As 
$$a_i b_i= (u_i w_i + z_i v_i) + \xi ( u_i z_i + v_i w_i +v_i z_i )= $$
  then 
 $$\left\{ \begin{array}{ccl} 
1&=&u_i w_i + z_i v_i\\ 
0 &=&u_i z_i + v_i w_i +v_i z_i
\end{array}\right. $$
\begin{enumerate}
\item {\sf Case 1: }
$$C_i + \xi D_i=a_i (A_i + \xi^2 B_i)= (u_i A_i + v_i B_i + u_i B_i) + (v_i A_i + u_i B_i )\xi$$ and 
$$C_i + \xi^2 D_i= b_i (A_i + \xi B_i)= (w_i A_i + z_i B_i) + (z_i B_i + z_i A_i + w_i B_i )\xi$$
Thus 
$$\left\{ \begin{array}{ccl} 
C_i &=& u_i A_i + v_i B_i + u_i B_i \\ 
D_i &=& v_i A_i + u_i B_i\\ 
C_i &=& w_i A_i  + z_i A_i + w_i B_i\\ 
D_i &=& z_i B_i + z_i A_i + w_i B_i
\end{array} \right.$$
Hence,
$$\left\{ \begin{array}{ccl} 
(u_i + w_i + z_i)A_i&=& ( v_i + u_i +w_i )B_i\\ 
( v_i + z_i )A_i &=& (u_i + z_i + w_i ) B_i
\end{array} \right.$$
Then, as $(A_i , B_1)=1$,
$$\left\{ \begin{array}{lll} 
 v_i +z_i=0 \\
  u_i + z_i + w_i =0 \\
  u_i w_i + z_i v_i =1 \\ 
u_i z_i + v_i w_i +v_i z_i=0
\end{array} \right.$$
Hence, 
$$\left\{ \begin{array}{lll} 
 v_i +z_i=0 \\
  u_i + z_i + w_i =0 \\
  N(u_i +\xi v_i )= u_i ^2 + u_iv_i +  v_i^2 =1 \\ 
\end{array} \right.$$
\item {\sf Case 2: } 
$$ C_i + \xi D_i = a_i (A_i + \xi B_i)= (u_i A_i + v_i B_i ) + (v_i A_i + u_i B_i +v_i B_i )\xi$$
and 
$$C_i + \xi^2 D_i= b_i (A_i + \xi^2 B_i)= (w_i A_i + z_i B_i+ w_i B_i ) + (z_i A_i + w_i B_i )\xi$$
Thus 
$$\left\{ \begin{array}{ccl} 
C_i &=& u_i A_i + v_i B_i \\ 
D_i &=& v_i A_i + u_i B_i +v_i B_i\\
C_i &=& w_i A_i  + z_i A_i + z_i B_i\\ 
D_i &=& z_i A_i + w_i B_i
\end{array} \right.$$
Hence,
$$\left\{ \begin{array}{ccl} 
(u_i + w_i + z_i)A_i&=& ( v_i + z_i )B_i\\ 
( v_i + z_i )A_i &=& (u_i + v_i + w_i ) B_i
\end{array} \right.$$
Thus, as before
$$\left\{ \begin{array}{lll} 
 v_i +z_i=0 \\
  u_i + z_i + w_i =0 \\
N(u_i +\xi v_i)=  u_i ^2 + u_iv_i +  v_i^2 =1 \\ 
\end{array} \right.$$
\end{enumerate}
As a consequence, we can always write $Q_i^{e_i}$ as an element in the image of the norm map $N$, whence $$Q_i^{e_i}=N\left( \prod_{l=1}^{e_i} ( u_{i,l}+ \xi v_{i,l}) \prod_{l=1}^{e_i} \theta_{i,l}  \right),$$ where $u_{i,l}$ and $v_{i,l} \in \mathbb{F}_q$ satisfies $1 = u_{i,l}^2 + u_{i,l} v_{i,l} + v_{i,l}^2 $ where $\theta_{i,l}  \in \{ A_i + \xi B_i , A_i+\xi^2 B_i \}$. Moreover, if in this product we have at least one $\theta_{i,l_0}= A_i + \xi B_i $ and at least one $\theta_{i,l_1} =A_i + \xi^2 B_i$ for some $l_0 \neq l_1$, then it would lead to $A$ and $B$ that would not be coprime. It follows that the only possibility is to write either $$Q_i^{e_i}=N\left(  \left(\prod_{l=1}^{e_i} ( u_{i,l} + \xi v_{i,l})\right) (A_i + \xi B_i)^{e_i} \right)\quad\text{or}\quad Q_i^{e_i}=N\left(\left( \prod_{l=1}^{e_i} ( u_{l,i} + \xi v_{l,i}) \right) (A_i +\xi^2 B_i)^{e_i} \right)$$ By Lemma 5.4, we write $w$ as $w=N(\mu + \xi \nu)$ for some $\mu$ and $\nu \in \mathbb{F}_q$, and $$ Q = N\left( ( \mu +\xi \nu) \prod_{i=1}^r \left( \prod_{l=1}^{e_i} ( u_{i,l} + \xi v_{i,l}) \right)\theta_i^{e_i} \right),$$ where $\epsilon_i = \pm 1$ with $\mu , \nu , u_{i,l} , v_{i,l} \in \mathbb{F}_q$  where $\theta_{i}  \in \{ A_i + \xi B_i , A_i+\xi^2 B_i \}$. We have $ ( \mu +\xi \nu) \prod_{i=1}^r  \prod_{l=1}^{e_i} ( u_{i,l} + \xi v_{i,l}) ) = \alpha + \xi \beta $ for some $\alpha, \beta \in \mathbb{F}_q$ and $w = N(\alpha + \xi \beta)$, and by Lemma 5.4, there are again $q+1$ such $\alpha$ and $\beta$. Finally, $$A+ \xi B= (\alpha + \xi \beta)\prod_{i=1}^r   \theta_i ^{e_i},$$ which yields the result in this case.
\end{enumerate} 
\end{proof}
\subsection{The irreducibility criterion} 
In the following Theorem and Corollary, we see that if $L/K$ is a Galois cubic extension with a generating equation of the form $y^3 -3y =a$ where $a= \frac{P}{Q}$ and $P$, $Q\in \mathbb{F}_q[x]$ with $(P, Q)=1$, then $Q$ cannot be a cube, up to a unit. 
\begin{theorem} 
Let $q \equiv -1 \mod 3$. Let $P(X)= X^3-3X- a$ be a polynomial, where $a=\frac{P}{Q}$, $P, Q\in \mathbb{F}_q[x]$, and $(P,Q)=1$.
\begin{enumerate}
\item \large If $p\neq 2$, suppose that $P=2 (A^2-3^{-1}B^2)$ and $Q =A^2+3^{-1}B^2$, for some $A, B \in \mathbb{F}_q[x]$ with $(A,B)=1$. Let $t \in \mathbb{F}_{q^2}$ be a root of the polynomial $X^3+ 3^{-1}$. Then $T(X)$ is a reducible polynomial over $\mathbb{F}_q(x)$ if, and only if,  $(A+tB)$ is a cube in $\mathbb{F}_{q^2}[x]$.
\item  \large If $p= 2$, suppose that $P=A^2$ and $Q =A^2 + AB + B^2$, for some $A, B \in \mathbb{F}_q[x]$ with $(A,B)=1$. Let $\xi \in \mathbb{F}_{q^2}$ be a primitive third root of unity. Then $T(X)$ is a reducible polynomial over $\mathbb{F}_q(x)$ if, and only if,  $(B+\xi A)$ is a cube in $\mathbb{F}_{q^2}[x]$.
\end{enumerate} \end{theorem} 
\begin{proof}
First, if $T(X)$ is reducible over $\mathbb{F}_q(x)$ then there is $\frac{f}{g}$ with $f,g \in \mathbb{F}_q[x]$ and $(f,g) = 1$ such that 
$$T\left( \frac{f}{g} \right) =0$$ 
That implies that 
$$ \left( \frac{f}{g}\right)^3- 3 \left(\frac{f}{g}\right) -b=0$$ 
That is 
$$ \frac{f^3  - 3 f g^2 }{g^3}= \frac{P}{Q}$$
As $(f^3  - 3 f g^2, g^3)=1$ and $(P , Q)=1$ then $Q= g^3$ up to a constant.\\

Suppose that $Q$ is a cube up to a unit $c \in \mathbb{F}_q^*$. Then up to multiplication by $c^2$, without loss of generality, we can suppose that $Q$ is a cube say $Q= g^3$ with $\mathbb{F}_q(x)$. 

\begin{enumerate}
\item If $p \neq 2$, let $t$ be a root of $x^2 +3^{-1}$ in $\mathbb{F}_q^2$. 
Suppose $L/\mathbb{F}_q(x) $ has a generating equation of the form $y^3 -3y =a$ where $a=\frac{P}{Q}$ where $P=2 (A^2-3^{-1}B^2)$ and $Q =A^2+3^{-1}B^2$, for some $A, B \in \mathbb{F}_q[x]$ and $(A,B)=1$. By Lemma \ref{decomQ}, $A+tB = (u+tv) (\ca \pm t\cb)^3$ and $A-tB = (u-tv) (\ca \mp t\cb)^3$, for some $u$, $v\in \mathbb{F}_q$ and $\ca $, $\cb \in \mathbb{F}_q(t)$ with $(\ca , \cb)=1$, as $A^2+3^{-1} B^2 = N(A+tB)= (A+tB)(A-tB) =Q= g^3$. 

We now examine the polynomial $X^3 -3X-a$ over $\mathbb{F}_q(t) (x)$. We have 
$$a^2 -4 = \left( 4 t \frac{AB}{A^2+ 3^{-1} B^2}\right)^2 = \delta^2,$$ where 
$$ \delta =4 t \frac{AB}{A^2+ 3^{-1} B^2}.$$ By the proof of Theorem \ref{purelycubic}, we have that as $a^2 -4$ is a square over $\mathbb{F}_q(t)$, whence $\mathbb{F}_q(t)(x) L/ \mathbb{F}_q(t) (x)$ is purely cubic, so that there is $w =  y^2 + k y -2$ such that $w^3 = b$, where 
$$ k =  \frac{-a \pm \delta }{2}=\frac{ - A^2 + 3^{-1} B^2 \pm 2 t AB}{A^2+3^{-1} B^2}=\frac{ - A^2 + 3^{-1} B^2 \pm 2 t AB}{A^2+3^{-1} B^2}=\frac{ - (A \mp tB)^2}{Q},$$ 
and furthermore,
$$ b= \delta^3k.$$ 
As we supposed $Q=A^2+3^{-1} B^2$ to be a cube and $A+tB = (u+tv) (\ca \pm t\cb)^3$, then we have that $X^3-b$ is reducible over $\mathbb{F}_q(t) (x)$ if, and only if, $X^3 -3X-a$ is reducible over $\mathbb{F}_q(t) (x)$, which is true if, and only if, $A + tB$ is a cube in $\mathbb{F}_q(t) (x)$. (This follows from the fact that $A+tB$ is a cube if, and only if, $A-tB$ is a cube, as $(A+tB)(A-tB) = A^2 + 3^{-1} B^2=Q=g^3$.)
Also, $X^3 -3X-a$ is irreducible over $\mathbb{F}_q(t) (x)$ implies $X^3 -3X-a$ is irreducible over $\mathbb{F}_q(x)$. Thus, if $T(X)$ is reducible over $\mathbb{F}_q (x)$, then $(A+ tB)$ is a cube in $\mathbb{F}_{q^2} (x)$.

For the converse, suppose that $Q= A^2+3^{-1}B^2$ and $A+tB$ is a cube in $\mathbb{F}_q(t)[x]$, say, $A+tB= (\ca +t \cb)^3$. Then
$$ A+t B= (\ca + t \cb )^3 = \cb (3\ca^2+(-3)^{-1}\cb^2)t+\ca (\ca^2-\cb^2)$$ and
$$ A = \cb (3\ca^2+(-3)^{-1}\cb^2) \quad and \quad B= \ca (\ca^2-\cb^2).$$ We now let 
$$r= \frac{2(\ca^2-3^{-1} \cb^2)}{\ca^2+3^{-1} \cb^2}\in  \mathbb{F}_q(x).$$  We observe that
$$r^3 -3 r= 2\frac{ (\ca^3 -\ca^2\cb^2)^2 -3^{-1} (3 \ca^2 \cb-3^{-1} \cb^3)^2}{(\ca^3 -\ca^2\cb^2)^2 +3^{-1} (3 \ca^2 \cb-3^{-1} \cb^3)^2}= \frac{P}{Q}.$$ It follows that $T(X)$ is reducible.

\item If $p=2$, let $\xi$ be a primitive $3^{rd}$ root of unity in $\mathbb{F}_{q^2}$. Suppose that $L/K$ has a generating equation of the form $y^3 -3y =a$, where $a=\frac{P}{Q}$, $P=A^2$, and $Q =A^2+AB+B^2$ for some $A$ and $B \in \mathbb{F}_q[x]$ and $(A,B)=1$. From Lemma \ref{decomQ}, $A+\xi B = (u+\xi v) (\ca + \xi \cb)^3$ and $A+\xi^2 B = (u+\xi^2 v) (\ca + \xi^2 \cb)^3$, as $A^2+AB+ B^2 = N(A+\xi B)= (A+\xi B)(A+\xi^2 B) =Q$, for some $u$, $v\in \mathbb{F}_q$ and $\ca $, $\cb \in \mathbb{F}_q(t)$ with $(\ca , \cb)=1$. 

We examine the polynomial $X^3 -3X-a$ over $\mathbb{F}_q(\xi ) (x)$. We have $$a= \frac{Z^2}{W(Z+W)}$$ for  $W=\xi A +B$, $Z =A$ and $(W, Z)=1$ as $(A, B)=1$. By Theorem \ref{purelycubic}, $w= y^2 +ky$ is such that $w^3=  b$, where 
$$ k = \frac{ W^7(Z+W)^5}{Z^{12}} \quad and \quad  b= \frac{W^4 (Z+W)^2}{P^3}= \left( \frac{W}{P}\right)^3 W(\xi^2 A+B)^2= \left( \frac{W}{P}\right)^3 Q(\xi^2 A+B).$$ As $Q$ is a cube by supposition and $A+\xi B = (u+\xi v) (\ca + \xi \cb)^3$ for some $u$, $v\in \mathbb{F}_q$ and $\ca $, $\cb \in \mathbb{F}_q(t)$ with $(\ca , \cb)=1$, we have $X^3-b$ is reducible over $\mathbb{F}_q(t) (x)$ if, and only if, $X^3 -3X-a$ is reducible over $\mathbb{F}_q(t) (x)$, which is true if, and only if, $B + \xi A$ is a cube in $\mathbb{F}_q(t) (x)$. 
Also, irreducibility of $X^3 -3X-a$ over $\mathbb{F}_q(\xi ) (x)$ implies that $X^3 -3X-a$ is irreducible over $\mathbb{F}_q(x)$. Thus, if $T(X)$ is reducible over $\mathbb{F}_q (x)$, then $(B+ \xi A)$ is a cube in $\mathbb{F}_{q^2} (x)$.  (This follows from the fact that $B+\xi A$ is a cube if, and only if, $B+\xi^2A$ is a cube, as $(B+\xi A)(B+\xi^2 A) = A^2 + AB+ B^2=Q=g^3$.)

Conversely, suppose that $B+\xi A$ is a cube in $\mathbb{F}_q(\xi)$, say, $(B+\xi A) = ( \cb + \xi \ca )^3$. Then 
\begin{align*} (B+\xi A)= & (\cb + \xi \ca)^3\\
 = &\ca^3\xi^3+3\cb \ca^2\xi^2+3\cb^2\ca\xi+\cb^3\\
=& \ca^3 +3 \cb \ca^2 (\xi +1)+3 \cb^2 \ca \xi+\cb^3\\
=& \ca^3 + \cb \ca^2 +\cb^3 +\cb \ca ( \ca  + \cb  )\xi,\\
\end{align*}
whence $B= \ca^3+\cb\ca^2+\cb^3\quad\text{and}\quad A= \cb \ca ( \ca  + \cb  )$. We now let $$r = \frac{\ca^2 + \cb^2}{\cb^2 + \cb \ca+ \ca^2} \in  \mathbb{F}_q(x).$$ We then observe that, for this choice of $r$,
\begin{align*} 
r^3 -3 r = \frac{P}{Q},
\end{align*} 
whence $T(X)$ is reducible.
\end{enumerate} 
\end{proof} 

\begin{corollaire} Suppose that $q\equiv -1 \text{ mod } 3 $, and that $L/\mathbb{F}_q(x)$ is cubic and Galois, so that there exists a primitive element $z$ of $L/\mathbb{F}_q(x)$ with irreducible polynomial $T(X)=X^3 - 3X-b$ (see Theorem \ref{qneq1mod3Galois}).  
\begin{enumerate}
\item If $p \neq 2$, then $b=\frac{P}{Q}$ where $P=2 (A^2-3^{-1}B^2)$ and $Q =A^2+3^{-1}B^2 $, for some $A,B \in \mathbb{F}_q[x]$ with $(A,B)=1$. The extension $L/\mathbb{F}_q(x)$ is constant if, and only if, $(A+ tB) = (u+ tv)( \ca + t \cb)^3$, with $u$,$v\in \mathbb{F}_q[x]$ and $u+ tv$ not a cube in $\mathbb{F}_{q^2}$, where $t \in \mathbb{F}_{q^2}$ is a root of the polynomial $X^3+ 3^{-1}$. 
\item When $p=2$, $b=\frac{P}{Q}$ where $P=A^2$ and $Q =A^2+AB+B^2$, for some $A, B \in \mathbb{F}_q[x]$ with $(A,B)=1$. The extension $L/\mathbb{F}_q(x)$ is constant if, and only if, $(B+\xi A) = (u+\xi v)( \cb + \xi \ca)^3$, with $u$, $v\in \mathbb{F}_q[x]$ and $u+\xi v$ not a cube in $\mathbb{F}_{q^2}$, where $\xi \in \mathbb{F}_{q^2}$ is a primitive third root of unity.
\end{enumerate}
In either case, if $L/\mathbb{F}_{q}(x)$ is constant, then $L = \mathbb{F}_{q^3}(x)$.
\end{corollaire}
\begin{proof}
We begin by noting that $\mathbb{F}_{q^2}L/ \mathbb{F}_{q^2}(x)$ is constant if, and only if, $L/\mathbb{F}_{q}(x)$ is constant. 
Indeed, if $L/\mathbb{F}_{q}(x)$ is constant, then clearly $\mathbb{F}_{q^2}L/ \mathbb{F}_{q^2}(x)$ is constant. If $L/\mathbb{F}_{q}(x)$ is not constant, then as there is no unramified extension of $\mathbb{F}_{q}(x)$ (see \cite[Theorem 6.1.3]{Vil}), there would be a ramified place in $\mathbb{F}_q(x)$ in $L$, and then a place of $\mathbb{F}_{q^2}(x)$ above this ramified place would also ramify in $\mathbb{F}_{q^2}L$, so that $\mathbb{F}_{q^2}L/\mathbb{F}_{q^2}(x)$ would also not be constant.\\
\begin{enumerate} 
\item If $p \neq 2$, $b=\frac{P}{Q}$ where $P=2 (A^2-3^{-1}B^2)$ and $Q =A^2+3^{-1}B^2 $, for some $A$ and $B \in \mathbb{F}_q[x]$ and $(A,B)=1$. In $\mathbb{F}_{q^2}(x)$, using the notation of the proof of the previous Theorem, we know that there is $w $ a generator of  $\mathbb{F}_{q^2}L/ \mathbb{F}_{q^2}(x)$ such that $w^3 = b$ where 
 $$ b= \delta^3\left( \frac{ - (A \pm tB)^2}{A^2+3^{-1} B^2}\right) =\delta^3\left( \frac{ - (A \pm tB)^2}{(A+tB) (A-tB)} \right) =\delta^3\left( \frac{ - A \pm tB}{A\mp tB}.\right)$$  
and taking $$z = \frac{-w}{\delta },$$ we have that $z^3 = \frac{A \pm t B}{A \mp t B}$. 
As $\mathbb{F}_q(\xi )[x]$ is a unique factorisation domain and $ ( A + tB)$ and $(A -t B)$ are coprime, we know that $ \frac{A \pm t B}{A \mp t B}$ is a cube if, and only if, $ (A+t B)$ and $(A -t B)$ are a cube. Also, $(A+tB)$ is a cube if, and only if, $A-tB$ is a cube, indeed  $(A+tB) (A-tB) = Q$ and $Q \in \mathbb{F}_q[x]$, thus each factor of $(A+tB)$ have their conjugates appearing to the same power.
Thus, $\mathbb{F}_{q^2}L/ \mathbb{F}_{q^2}(x)$ is constant if, and only if, $A+ t B$ is a cube up to a non-cube constant in $\mathbb{F}_{q^2}$ (see Lemma \ref{Kconstant}). 
\item If $p=2$, $b=\frac{P}{Q}$ where $P=A^2$ and $Q =A^2+AB+B^2$, for some $A$ and $B \in \mathbb{F}_q[x]$ and $(A,B)=1$. In $\mathbb{F}_{q^2}(x)$, using the notation of the proof of the previous Theorem, we know that there is $w $ a generator of  $\mathbb{F}_{q^2}L/ \mathbb{F}_{q^2}(x)$ such that $w^3 = b$ where 
$$b= \left( \frac{W}{P}\right)^3(\xi^2 A+ B) (\xi A + B) (\xi^2 A+B)=\left( \frac{W}{P}\right)^3(\xi^2 A+ B)^2 (\xi A + B) $$ 
and taking $z = \frac{Pw}{W }$, then we obtain $z^3= (\xi^2 A+ B)^2 (\xi A + B) $. 
As $\mathbb{F}_q(\xi )[x]$ is a unique factorisation domain and $ (\xi^2 A+ B)$ and $(\xi A + B)$ are coprime, we know that $ (\xi^2 A+ B)^2 (\xi A + B)$ is a cube if, and only if, $ (\xi^2 A+ B)$ and $(\xi A + B)$ are a cube. Note that $(\xi A+B)$ is a cube if, and only if, $\xi^2 A+ B$ is a cube, indeed  $(\xi A+B) (\xi^2 A+ B) = Q$ and $Q \in \mathbb{F}_q[x]$, thus each factor of $\xi A+B$ have their conjugates appearing to the same power.
Thus, $\mathbb{F}_{q^2}L/ \mathbb{F}_{q^2}(x)$ is constant if, and only if, $\xi A+ B$ is a cube up to a non-cube constant in $\mathbb{F}_{q^2}$ (once again, we refer the reader to Lemma \ref{Kconstant}). 
\end{enumerate} 
\end{proof}
\subsection{The Galois action on the generator} 
The following Theorem describes the Galois action on the generator $z$ of a cubic Galois extension satisfying a minimal equation of the form $X^3 -3X -b=0$. This is not as simple as for the Artin-Schreier or Kummer generators (see Theorems \ref{ASGal} and \ref{KGal}), but may be concisely described, which we now do here. We note that in the following Theorem, we do not assume that $K=\mathbb{F}_q(x)$.
\begin{theorem}\label{explicite}
Suppose $q \equiv -1 \mod 3$. Let $L/K$ be a Galois cubic extension, and let $z$ be the element of Corollary \ref{linearthree} which has minimal polynomial of the form $T(X)=X^3 -3 X - a$. Then, the Galois action of $L/K$ on $z$ is given in the following way:
$$\sigma (z) = -\frac{1+2f}{a} z^2 +f z + \frac{2(1+2f)}{a}$$ and $$\sigma^2 (z)= \frac{1+2f}{a} z^2 +(-1-f) z - \frac{2(1+2f)}{a}$$ for $\sigma \in \text{Gal}(L/K)$, where $f$ is one root of the polynomial $$S(X)=\left(1 - \frac{4}{a^2}\right) X^2 + \left( 1- \frac{4}{a^2}\right) X + \left(1-\frac{1}{a^2}\right). $$
(Note that $-1-f$ is the other root of $S(X)$.)
Moreover,
\begin{enumerate}[(a)] 
\item If $p \neq 2$, then $$f= \frac{- 3\left(a^2 -4 \right)\pm a \delta }{ 6 \left(a^2 - 4 \right)}$$
where $\delta^2 = D$, 
with $D=-27(a^2-4)$ is the discriminant of the polynomial $P(X)=X^3 -3 X-a$.
\item If $p=2$, $af$ is a root of the resolvent polynomial $R(X)=X^2 + a X + (1+a^2)$ of $T(X)$. 
\end{enumerate}
\end{theorem}
\begin{proof} Suppose that $q \equiv -1 \mod 3$. Let $\sigma \in Gal (L/K)$. As $\sigma (z) \in L$, we let 
$$\sigma (z) = e z^2 + fz +g,$$ with $e,f,g \in K$. Then \begin{align*}&\sigma^2(z) \\=& e (e z^2 + fz +g)^2 + f (e z^2 + fz +g) +g \\ =& e^3 z^4 +ef^2 z^2 + e g^2 + 2 e^2f z^3 +2e^2g z^2 + 2 efgz + f ez^2 + f^2 z + fg  + g \\ =& 3 e^3 z^2 + e^3 a z + ef^2 z^2 + eg^2  +  6 e^2f z +  2 e^2f a +2e^2g z^2 + 2 efgz    +fez^2  + f^2 z + fg + g \\ =& (3e^3+ ef^2 + fe+ 2 e^2 g) z^2 + ( e^3 a + f^2 +  6 e^2f + 2 efg)z + (eg^2 + fg +g  +  2 e^2f a).
\end{align*}
Moreover, as $z$, $\sigma(z)$ and $\sigma^2(z)$ are the three roots of $P(X)$ and $P(X)= X^3 -3X-a$, we have $$Tr(z)= z+ \sigma (z) + \sigma^2 (z)=0.$$ That is, $$(3e^3+ ef^2 + fe+ 2 e^2 g +e ) z^2 + (  e^3 a + f^2 +  6 e^2f + 2 efg +f +1 )z + (eg^2 + fg +g  +  2 e^2f a +g)=0.$$ As $1, z, z^2$ is a basis for $L$ over $K$, we obtain the system 
$$\left\{ \begin{array}{llll} 3e^3+ ef^2 + fe+ 2 e^2 g +e & =& 0 & \\ 
e^3 a + f^2 +  6 e^2f + 2 efg +f +1&=&0 & \\ 
eg^2 + fg +g  +  2 e^2f a +g &=& 0 & 
\end{array}\right.  $$
We note that $e\neq 0$, as $K$ does not contain primitive $3^{rd}$ roots of unity, using the second equation of the system. Thus, the previous system simplifies to
$$\left\{ \begin{array}{llll} 3e^2+ f^2 + f+ 2 e g +1 & =& 0 & \\ 
e^3 a + f^2 +  6 e^2f + 2 efg +f +1&=&0 & \\ 
eg^2 + fg +g  +  2 e^2f a +g &=& 0 & 
\end{array}\right. \text{(I)}$$
We denote $z_1 := \sigma (z)$, then $z_2 := \sigma^2 (z)= -z_1  -z$. As the linear coefficient of $P(X)$ is equal to $-3$, we have that \begin{align*} -3 &= z z_1  + z z_2 + z_1 z_2 \\ &=  zz_1 + z(-z_1 -z) + z_1 (-z_1 -z) \\ &= -z^2 - z_1^2 - z_1 z\\ &= -z^2 - (e z^2 + f z + g)^2 -z(e z^2 + f z + g) \\ &= -z^2 - e^2 z^4 - f^2 z^2 - g^2 - 2ef z^3 - 2egz^2 - 2fg z-e z^3 - f z^2 - gz   \\ &= -z^2 - 3 e^2 z^2 - a e^2 z - f^2 z^2 - g^2 - 6 ef z - 2efa  - 2egz^2 - 2fg z-3e z - ea - f z^2 - gz \\ &=-(1+ 3e^2 + f^2 + 2eg + f)z^2 - (ae^2+6ef + 2fg + 3e+g) z - (g^2 + ea +2efa). \end{align*} 
We therefore obtain the following system:
$$\left\{ \begin{array}{llll} 1+ 3e^2 + f^2 + 2eg + f & =& 0 &  \\ 
ae^2+6ef + 2fg + 3e+g&=&0 & \\ 
g^2 + ea +2efa-3 &=& 0 & 
\end{array}\right. \text{(II)}$$
As the constant coefficient is equal to $-a$, we must have \begin{align*} a & = z z_1 z_2 = z z_1 (-z_1 - z)= -z z_1^2 - z^2 z_1 \\ &= - [z(e z^2 + f z + g)^2 + z^2 (e z^2 + f z + g)] \\ &= -(e^2 z^5 + f^2 z^3 + g^2 z+ 2ef z^4 + 2egz^3 + 2fg z^2 + e z^4 + f z^3 + g z^2) \\ &= -(3 e^2 z^3 + e^2 a z^2 + 3f^2 z + f^2 a + g^2 z +6 ef z^2 + 2efaz + 6eg z + 2eg a \\ & \quad\;\;+2fg z^2 + 3e z^2 +ea z + 3 fz + fa + gz^2) \\ &= -(9 e^2 z + 3e^2 a + e^2 a z^2 + 3f^2 z + f^2 a + g^2 z +6 ef z^2 + 2efaz + 6eg z + 2eg a \\ &\quad\;\; +2fg z^2 + 3e z^2 +ea z + 3 fz + fa + gz^2 ) \\ &= -(e^2 a +6ef + 2fg +3e+g)z^2 - (9 e^2 + 3f^2 +g^2 + 2efa + 6eg +ea+3f)z \\ &\quad\;\; - ( 3e^2 a +f^2 a + 2eg a + fa). \end{align*} This leads 
$$\left\{ \begin{array}{llll}e^2 a +6ef + 2fg +3e+g & =& 0 &  \\ 
9 e^2 + 3f^2 +g^2 + 2efa + 6eg +ea+3f &=&0 & \\ 
3e^2 a +f^2 a + 2eg a + fa+a &=& 0 & 
\end{array}\right. $$
As $a\neq 0$, the system becomes
$$\left\{ \begin{array}{llll}e^2 a +6ef + 2fg +3e+g & =& 0 &  \\ 
9 e^2 + 3f^2 +g^2 + 2efa + 6eg +ea+3f &=&0 & \\ 
3e^2  +f^2 + 2eg  + f+1 &=& 0 & 
\end{array}\right. \text{(III)}$$
 Together, the systems $\text{(I)}$, $\text{(II)}$, $\text{(III)}$ yield the system of six equations $$\left\{ \begin{array}{cclc} 3e^2+ f^2 + f+ 2 e g +1& =& 0 & (1) \\ e^3 a + f^2 +  6 e^2f + 2 efg +f +1&=&0 & (2) \\ eg^2 + fg +2g  +  2 e^2f a  &=& 0 &(3) \\ ae^2+6ef + 2fg + 3e+g &=& 0 & (4) \\ g^2 + ea +2efa-3 &=&0 & (5) \\ 
9 e^2 + 3f^2 +g^2 + 2efa + 6eg +ea+3f &=& 0& (6)\\ \end{array}\right. $$ Subtracting Eq. (2) from (1) then yields $$0=3e^2+ f^2 + f+ 2 e g +1-(e^3 a + f^2 +  6 e^2f + 2 efg +f +1)=3e^2+ 2 ge -e^3 a -  6 e^2f - 2 fge.$$
 Thus, as $e \neq 0$ (because $K$ does not contain primitive $3^{rd}$ roots of unity), division by $e$ yields $$3e+ 2 g -e^2 a -  6 ef - 2 fg =0. \ \quad\quad\quad\quad\quad\quad\quad\quad (7)$$ The sum of Eqs. $(4)$ and $(7)$ implies $6e + 3g=0$. As $p \neq 3$, it follows that 
 $$g=-2e. \ \ \ (*)$$ By substitution of this identity in Eq. $(4)$, we obtain
 \begin{align*}ae^2 + 2ef + e=ae^2 + 6ef-4ef +3e - 2e=0.\quad\quad \end{align*} As $e\neq 0$ and $a \neq 0$, it follows that $ae + 2f + 1 = 0$, and hence that 
 $$e=- \frac{1+2f}{a}. \ \ \ (**)$$
  Eq. (1) then reads as \begin{align*} 0 &= 3e^2+ f^2 + f- 4 e^2 +1 = -e^2 + f^2 + f+1 \\ &= -\left(-\frac{1+2f}{a}\right)^2 + f^2 + f+1 = -\frac{(1+4f +4f^2)}{a^2} + f^2 + f+1 \quad\quad \\ & = \left(1-\frac{4}{a^2}\right) f^2 +\left(1-\frac{4}{a^2}\right)f + \left(1-\frac{1}{a^2}\right)   \end{align*}
\begin{enumerate} 
\item[$\circ$] If $p\neq 2$, the discriminant of the polynomial $$S(X):=\left(1-\frac{4}{a^2}\right) X^2 +\left(1-\frac{4}{a^2}\right)X + \left(1 - \frac{1}{a^2}\right)$$ is equal to \begin{align*} \Delta &=\left(1-\frac{4}{a^2}\right)^2 -4\left(1-\frac{4}{a^2}\right)\left(1-\frac{1}{a^2}\right)\\&=\left(1-\frac{4}{a^2}\right)\left( \left(1-\frac{4}{a^2}\right)-4\left(1-\frac{1}{a^2}\right)\right)\\&=-3\left(1-\frac{4}{a^2}\right)\\&= \frac{1}{9a^2}D,\end{align*} where $D=-27(a^2-4)$ is the discriminant of the polynomial $X^3 -3 X - a$, which is a square as $L/K$ is a Galois. Thus $\Delta$ is a square. Thus $S(X)$ has a root and $f$ can be chosen as one of these roots. 
The latter together with $(*)$ and $(**)$ give part $(a)$ of the Theorem. 
\item[$\circ$]  If $p=2$, 
by definition, the resolvent $R(X)$ of $T(X)$ is equal to $$R(X) = X^2 +a X + 1+a^2$$ and has a root that we denote $\gamma$, as $L/K$ is Galois. We observe that
$$S(X)=R(X)/a^2= Y^2 + Y + 1 + 1/a^2$$ where $Y = X/a$. Thus, $\gamma a$ is a root of $S(X)$. The later together with $(*)$ and $(**)$ give part $(b)$ of the Theorem. 
\end{enumerate} 
Lastly, we give a brief verification that $e$, $g$, and $f$ as defined satisfy the equations $(1)-(6)$. Equations $(1)$ and $(4)$ have already been verified before. Via the substitution $g=-2e$, Eq. $(2)$ becomes $$0=e^3 a + f^2 +  6 e^2f + 2 efg +f +1=e^3 a +  6 e^2f + 2 efg  +e^2,$$ which becomes $(4')$ by division of $e^2$. Via the substitutions $g=-2e$ and $e=-\frac{1+2f}{a}$, Eq. $(3)$ becomes \begin{align*} eg^2 + fg +2g  +  2 e^2f a &= 4e^3 -2fe -4e + 2 e^2f a\\ &= e( 4e^2 -2f -4 + 2 ef a)\\&= e(4e^2 -2f-4 -2f-4f^2)\\&= -4e( -e^2 +1+ f +f^2)=0,\end{align*} 
which is satisfied as Eq. $(1)$ is satisfied. Once again via substitution with $g=-2e$ and $e=- \frac{1+2f}{a}$, Eq. $(5)$ becomes $$0 = 4e^2 + ea +2efa -3 =4 e^2 -4f -4f^2 -4= -4(-e^2 +f +f^2 +1),$$ which is satisfied as Eq. $(1)$ is satisfied. Finally, via substitution with $g=-2e$ and $f+ f^2 = e^2 -1$, Eq. $(6)$ becomes $$0=9e^2 +3f^2 + 4e^2 + 2ef -12e^2 + ea + 3f =4e^2 -3 + ea + 2efa,$$ which is satisfied as $(5)$ is satisfied.
\end{proof} 

\subsection{Generators with minimal equation $X^3-3X-a=0$}
The following Theorem gives the form of any generator with minimal equation $X^3 -3X -a=0$ for some $a \in K$. This result is equivalent to Lemma \ref{genAS} and \ref{genK}. In particular, we discover that there are infinitely many of those for a given Galois extension of this form. We give the proof in the Appendix for clarity and brevity, as the argument is quite long. 
\begin{theorem}\label{reallylongtheorem}
Suppose that $q \equiv -1 \mod 3$. Let $L_i = K(z_i)/K$ ($i=1,2$) be two cyclic extensions of degree $3$ such that $z_i^3 -3z_i=a_i\in K$. The following are equivalent: 
\begin{enumerate}
\item $L_1= L_2$;
\item $z_2=\phi  z_1^2 +\chi z_1 -2\phi$, where $\phi ,\chi \in K$ satisfy the equation $\chi^2+a_1\phi \chi+\phi^2=1.$ 
\end{enumerate}
Moreover, if $K = \mathbb{F}_q(x)$, when the above conditions are satisfied, then:
\begin{enumerate}[(a)] 
\item If $p\neq 2$,  there are relatively prime polynomials $C,D\in \mathbb{F}_q[x]$ 
$$\phi = \frac{ 12C DQ_1}{\delta ( D^2 +3^{-1} C^2)}\quad\quad\text{and}\quad\quad \chi =\frac{-6P_1CD\pm \delta(C^2 -3^{-1} D^2)}{\delta( D^2 +3^{-1}  C^2)},$$ 

where $\delta \in K$ is such that $\delta^2=Q_1^2 d_{L/\mathbb{F}_q(x)}$ and $a_1= P_1/Q_1$, $P_1$, $Q_1\in \mathbb{F}_q[x]$ with $(P_1, Q_1)=1$ and $$a_2 =(-2 + a_1^2) \phi^3 + 3 a_1 \phi^2 \chi + 6 \phi \chi^2 + a_1 \chi^3.$$
\item If $p=2$,
\begin{enumerate}[(i)]
\item if $\phi =0$, then $\chi =\pm 1$. If $\chi =1$, then $a_2=a_1$, and if $\chi =-1$, then $a_2=-a_1$.
\item if $\phi \neq 0$, then $\chi /(a_1 \phi)$ is a solution of 
$$X^2-X-\frac{\phi^2-1}{a_1 \phi }=0$$ and 
$$a_2 = a_1(a_1 \phi^3 +   \phi^2 \chi  +  \chi^3).$$
\end{enumerate}
\end{enumerate}\end{theorem} 
\subsection{Ramification and splitting}
The next Theorem studies the ramification for cubic extensions with a primitive element $y$ whose minimal equation is of the form $y^3-3y-a=0$. From it, we obtain the ramification of the extension via the decomposition of the denominator of $a$ into prime factors. Before doing this, we observe via the following Lemma that the valuation of such an element $y$ at an infinite place is always nonnegative. 
\begin{lemme}\label{infinity} Let $p \neq 3$ and $q \neq 1 \mod 3$. Let $L/\mathbb{F}_q(x)$ be a Galois cubic extension with primitive element $z$, which has minimal polynomial over $\mathbb{F}_q(x)$ equal to $T(X) =X^3 - 3X - a$. Then the place $\mathfrak{p}_{\infty}$ corresponding to the pole divisor of $x$ in $\mathbb{F}_q(x)$ satisfies $v_{\mathfrak{p}_\infty}(b)\geq 0$. \end{lemme}
\begin{proof} By contradiction, suppose that the pole divisor $\mathfrak{p}_\infty$ of $x$ in $\mathbb{F}_q(x)$ satisfies $v_{\mathfrak{p}_\infty}(a) < 0$. 
By Theorem \ref{qneq1mod3Galois}, with $a = P/Q$ for relatively prime $P, Q$, there exist $A,B \in \mathbb{F}_q[x]$ relatively prime such that $$P=  2(A^2- 3^{-1} B^2), \quad\quad \text{and}\quad\quad Q =A^2+ 3^{-1} B^2.$$ In particular, as $v_{\mathfrak{p}_\infty}(a) < 0$, it follows that $$\deg\left(\frac{1}{2}[A^2- 3^{-1} B^2]\right) = \deg(P)  > \deg(Q)=\deg\left(\frac{1}{4}[A^2+ 3^{-1} B^2]\right).$$ Let $m = \deg(A)$ and $n = \deg(B)$, with $A = \sum_{i=0}^m a_i x^i$ and $B = \sum_{j=0}^n b_j x^j$. By the previous inequality, we obtain $\deg(A^2+3^{-1} B^2) < \max\{m,n\}$, which implies that $m = n$ and $a_n^2 + 3^{-1} b_n^2 = 0$. Thus $(-3)^{-1}= (b_n/a_n)^2$, which is a square in $\mathbb{F}_q$, contradicting that  $q \equiv -1 \mod 3$. 
\end{proof}

In general, let $L/\mathbb{F}_q(x)$ be a cubic function field. The discriminant $\text{disc}_x(L)$ of $L$ over $\mathbb{F}_q[x]$ is equal to the discriminant of any basis of the integral closure $\mathfrak{O}_L$ of $L$ over $\mathbb{F}_q[x]$. If $\omega$ is a generator of $L/\mathbb{F}_q(x)$ which is integral over $\mathbb{F}_q[x]$, the discriminant $\Delta(\omega)$ of $\omega$ differs from the discriminant $\text{disc}_x(L)$ by an integral square divisor, which we denote by $I = \text{ind}(\omega )$; in particular, we have $\Delta(\omega ) = I^2 \text{disc}_x(L)$. By definition, we also have $(\text{disc}_x(L))_{\mathbb{F}_q[x ]} = (\partial_{L/\mathbb{F}_q(x)})_{\mathbb{F}_q[x]}$, where  $\partial_{L/\mathbb{F}_q(x)}$ is the discriminant ideal of $L/\mathbb{F}_q(x)$ \cite[Definition 5.6.8]{Vil}. In the following Theorem, we prove for our extensions that the place $\mathfrak{p}_\infty$ at infinity for $x$ is unramified in $L$, which in turn implies that $\mathfrak{p}_\infty$ does not appear in $\partial_{L/\mathbb{F}_q(x)}$, whence $\partial_{L/\mathbb{F}_q(x)} = (\partial_{L/\mathbb{F}_q(x)})_{\mathbb{F}_q[x]}$.
 
\begin{theorem}\label{ramification}
Suppose $q \equiv - 1 \mod 3$. Let $L/\mathbb{F}_q(x)$ be a Galois cubic  geometric extension and $y$ a primitive element with $f(y) = y^3-3y-a=0$. Then the ramified places of $\mathbb{F}_q(x)$ in $L$ are precisely the places $\mathfrak{p}$ of $\mathbb{F}_q(x)$ such that $v_{\mathfrak{p}}(a)<0$ and $(v_{\mathfrak{p}}(a), 3)=1$. 
\end{theorem}
\begin{proof} We first prove that all places $\mathfrak{p}$ of $\mathbb{F}_q(x)$ such that $v_{\mathfrak{p}}(a)<0$ and $(v_{\mathfrak{p}}(a), 3)=1$ are fully ramified in $L$. 
Suppose that $\mathfrak{p}$ is a place of $\mathbb{F}_q(x)$, $v_\mathfrak{p}(a) < 0$ and $(v_\mathfrak{p}(a),3) = 1$. Let $\mathfrak{P}$ be a place of $L$ which lies above a place $\mathfrak{p}$ of $K$, we denote $e(\mathfrak{P}|\mathfrak{p})$ the ramification index. As $$v_\mathfrak{P}(a) = v_\mathfrak{P}(y^3 - 3y) < 0,$$ then $v_\mathfrak{P}(y) < 0$. Indeed, if $v_\mathfrak{P}(y)=0$, then $v_\mathfrak{P}(y^3 - 3y)\geq v_\mathfrak{P}(y)=0$, and if $v_P(y)\neq 0$ then $v_\mathfrak{P}(y^3)\neq v_P(-3y)$, and if $v_\mathfrak{P}(y)>0$, then $v_\mathfrak{P}(y^3 - 3y)= v_\mathfrak{P}(y)>0$. It follows that $v_\mathfrak{P}(y^3 - y) = 3v_\mathfrak{P}(y)<0$. As $v_\mathfrak{P}(a) = e(\mathfrak{P}|\mathfrak{p})v_\mathfrak{p}(a)$, we obtain that $3$ divides $e(\mathfrak{P}|\mathfrak{p})v_\mathfrak{p}(a)$. As $(v_\mathfrak{p}(a),3) = 1$, it follows that $3$ must divide $e(\mathfrak{P}|\mathfrak{p})$. In particular, $e(\mathfrak{P}|\mathfrak{p}) \geq 3$. As the other direction of the inequality holds by basic number theory, it follows that $e(\mathfrak{P}|\mathfrak{p}) = 3$, and that $\mathfrak{p}$ is fully ramified in $L$.
We now prove that a place $\mathfrak{p}$ of $\mathbb{F}_q(x)$ is unramified in $L$ whenever $v_\mathfrak{p}(a) \geq 0$. For the minimal polynomial $f(X) = X^3 - 3X - a$, we have $$f'(X) = 3X^2 - 3 = 3(X^2 - 1) = 3(X-1)(X+1).$$ 
Let $\mathfrak{P}$ be a place of $L$ which lies above the place $\mathfrak{p}$ of $K$. By \cite[Theorem 3.5.10(a)]{Sti}, if $\mathfrak{P}$ does not divide $f'(y)=3(y-1)(y+1)$, then $0 \leq d(\mathfrak{P}|\mathfrak{p}) \leq v_\mathfrak{P}(f'(y)) = 0$. Hence $\mathfrak{P}$ would not appear in the different, and $\mathfrak{p}$ would be unramified in $\mathfrak{P}$. The same is also true for any root of $f(X)$; these roots are precisely $y$, $\sigma(y)$, and $\sigma^2(y)$. Furthermore, as $v_\mathfrak{p}(a) \geq 0$ and each of $y$, $\sigma(y)$, and $\sigma^2(y)$ is a root of $f(X)$, it follows that $v_{\mathfrak{P}}(y)$, $v_{\mathfrak{P}}(\sigma(y))$ and $v_{\mathfrak{P}}(\sigma^2(y))$ are all nonnegative. Suppose then for the sake of contradiction that $\mathfrak{P}$ divides $3(y-1)(y+1)$. Then it is necessary that $\mathfrak{P}$ divides $(y-1)$ or $(y+1)$, as $\mathfrak{P}$ is prime. The place $\mathfrak{P}$ must also divide $(\sigma(y)-1)$ or $(\sigma(y)+1)$, and $(\sigma^2(y)-1)$ or $(\sigma^2(y)+1)$ (Ibid.). Let $a_0, a_1, a_2 \in \{-1,1\}$ such that $\mathfrak{P} | (y +a_0 )$,  $\mathfrak{P} | (\sigma(y) +a_1)$, and  $\mathfrak{P} | (\sigma^2(y) +a_2)$. Thus we have $$(y + a_0) + (\sigma(y) + a_1) + (\sigma^2(y) + a_2) \equiv 0 \mod \mathfrak{P}.$$ By construction, we have that the trace of $f(X)$ is equal to $y + \sigma(y) + \sigma^2(y) = 0$. We therefore obtain $$a_0 + a_1 + a_2= (y + a_0) + (\sigma(y) + a_1) + (\sigma^2(y) + a_2) \equiv 0 \mod \mathfrak{P}.$$ As $a_0,a_1,a_2 \in \{-1.1\}$ and $p \neq 3$, this is a contradiction. Thus $\mathfrak{P}$ does not divide $3(y-1)(y+1)$ and $\mathfrak{P}$ does not appear in the different. As the constant field is finite, all places are separable, in particular $\mathfrak{p}$, and so by \cite[Proposition 5.6.9]{Vil}, it follows that $\mathfrak{p}$ is unramified in $L$.

We now suppose that $v_\mathfrak{p}(a) < 0$ and that $3|v_\mathfrak{p}(a)$. 
In particular we have that $\mathfrak{p} \neq \mathfrak{p}_{x,\infty}$. We write $a = \frac{\alpha}{\gamma^3 \beta}$ with $\alpha,\beta,\gamma\in \mathbb{F}_p[x]$,  $\beta$ is cube-free and $(\alpha, \gamma^3 \beta)=1$. It follows that $\mathfrak{p}$ appears in $\gamma$, and that $v_\mathfrak{p}(\alpha) = v_\mathfrak{p}(\beta) = 0$. The transformation $y = z/\gamma$ then yields $$\left(\frac{z}{\gamma}\right)^3-3\left(\frac{z}{\gamma}\right)-\frac{\alpha}{\gamma^3 \beta} = 0, \quad \quad \quad \text{or} \quad\quad\quad z^3 - 3\gamma^2z - \frac{\alpha}{\beta} = 0.$$ Substituting $u/\beta =z$ gives us $$\left(\frac{u}{\beta}\right)^3 - 3\gamma^2\left(\frac{u}{\beta}\right) - \frac{\alpha}{\beta} = 0,\quad \quad \quad \text{or} \quad\quad\quad u^3 - 3 \gamma^2 \beta^2 u - \beta^2 \alpha=0.$$ The element $u$ is integral over $\mathbb{F}_q[x]$. In particular, the $\mathbb{F}_q[x]$-ideal $\mathfrak{I}_u$ generated by $\{1,u,u^2\}$ is contained in the integral closure $\mathcal{O}_L$ of $\mathbb{F}_q[x]$ in $L$. The determinant of the transformation matrix from an integral basis of $\mathcal{O}_L$ to $\mathfrak{I}_u$, which is equal to a constant multiple of $\text{ind}(u)$, is contained in $\mathbb{F}_q[x]$, from which it then follows that $\text{ind}(u) \in \mathbb{F}_q[x]$. 
We also have by definition that $$\text{disc}(u) = \text{ind}(u)^2\text{disc}_x(L).$$ Thus $\text{disc}_x(L)|\text{disc}(u)$ in $\mathbb{F}_q[x]$. By definition, the irreducible polynomial of $u$ over $K$ is equal to $S(X) = X^3 - AX + B$ where $A =  3 \gamma^2 \beta^2$ and $B = - \beta^2 \alpha$. Thus $$\text{disc}(S) := \text{disc}(u) = 4A^3 -27B^2 = 4(3 \gamma^2 \beta^2)^3 - 27( - \beta^2 \alpha)^2.$$ As $v_\mathfrak{p}(\alpha) = v_\mathfrak{p}(\beta) = 0$ and $v_\mathfrak{p}(\gamma) > 0$, it follows by the strict triangle inequality that $v_\mathfrak{p}(\text{disc}(u)) = 0$. As $\text{disc}_x(L)|\text{disc}(u)$ in $\mathbb{F}_q[x]$, it follows that $v_\mathfrak{p}(\text{disc}_x(L)) = 0$. As the constant field is finite, all places are separable, in particular $\mathfrak{p}$, and so by \cite[Proposition 5.6.9]{Vil}, it follows that $\mathfrak{p}$ is unramified in $L$.
\end{proof}
By Lemma 5.9 and Theorem 5.10, we immediately obtain the following Corollary.
\begin{corollaire} Suppose $q \equiv - 1 \mod 3$. Let $L/\mathbb{F}_q(x)$ be a cubic extension and $y$ a primitive element with $f(y) = y^3-3y-a=0$. Then the place $\mathfrak{p}_\infty$ at infinity for $x$ is unramified in $L$.\end{corollaire}
\begin{remarque}
By Theorem \ref{even}, we know that only places corresponding to polynomials of even degree can be ramified in cubic extensions with a primitive element $y$ whose minimal equation is of the form $y^3-3y-a=0$. 
\end{remarque} 

\begin{theorem}[Riemann-Hurwitz] Suppose $q \equiv - 1 \mod 3$. Let $L/\mathbb{F}_q(x)$ be a Galois cubic  geometric extension and $y$ a primitive element with $f(y) = y^3-3y-a=0$. Then the genus $g_L$ of $L$ is given according to the formula $$g_L =  - 2 +  \sum_{\substack{ v_\mathfrak{p}(a) < 0 \\ ( v_\mathfrak{p}(a),3) = 1}} \deg(\mathfrak{p}),$$ where $\deg(\mathfrak{p})$ denotes the degree of a place $\mathfrak{p}$ of $\mathbb{F}_q(x)$.
\end{theorem}
\begin{proof} As $[L:\mathbb{F}_q(x)] = 3$ and $L/\mathbb{F}_q(x)$ is Galois, it follows that all ramification indices are either equal to 1 or 3. Thus for a place $\mathfrak{p}$ of $\mathbb{F}_q(x)$ which ramifies in $L$, we have for $\mathfrak{P}|\mathfrak{p}$ that, with $p = \text{char}(\mathbb{F}_q)$, $(e(\mathfrak{P}|\mathfrak{p}),p) = (3,p) = 1$, whence by \cite[Theorem 5.6.3]{Vil}, the differential exponent $\alpha(\mathfrak{P}|\mathfrak{p})$ satisfies $$\alpha(\mathfrak{P}|\mathfrak{p})  = e(\mathfrak{P}|\mathfrak{p}) - 1 = 2.$$ Furthermore, by Theorem 5.10, the places $\mathfrak{p}$ of $\mathbb{F}_q(x)$ which ramify in $L$ are precisely those for which $ v_\mathfrak{p}(a) < 0$ and $( v_\mathfrak{p}(a),3) = 1$. We let $d_L$ denote the degree function on the divisors of $L$; we therefore also let $d_L(\mathfrak{P})$ denote the degree of a place $\mathfrak{P}$ of $L$. As $L/\mathbb{F}_q(x)$ is of prime degree, it follows that $d_L(\mathfrak{P}) = \deg(\mathfrak{p})$ for all ramified places $\mathfrak{p}$ of $\mathbb{F}_q(x)$ in $L$ and $\mathfrak{P}|\mathfrak{p}$. We thus obtain by the Riemann-Hurwitz formula (see for example [Ibid., Theorem 9.4.2]) \begin{align*} g_L&= 1 +[L:\mathbb{F}_q(x)](g_{\mathbb{F}_q(x)} - 1) + \frac{1}{2} d_L\left(\mathfrak{D}_{L/\mathbb{F}_q(x)}\right)  \\& = -2 + \frac{1}{2} \sum_{\substack{\mathfrak{P}|\mathfrak{p} \\ v_\mathfrak{p}(a) < 0 \\ ( v_\mathfrak{p}(a),3) = 1}} d_L(\mathfrak{P}^2) \\& = -2 + \sum_{\substack{ v_\mathfrak{p}(a) < 0 \\ ( v_\mathfrak{p}(a),3) = 1}} \deg(\mathfrak{p}).\end{align*} Hence the result. \end{proof}

We now examine the valuation of a generator $y$ of a cubic extension whose minimal equation is of the form $y^3-3y-a=0$. 

\begin{theorem}
Suppose that $q\equiv -1 \text{ mod } 3$, and that $L/\mathbb{F}_q(x)$ is cubic and Galois. For any place $\mathfrak{p}$ of $K$, denote by $\mathfrak{P}$ a place above $\mathfrak{p}$ in $L$. We denote by $\mathfrak{p}_\infty$ the place at infinity in $\mathbb{F}_q(x)$ for $x$. Let $y$ of $L/\mathbb{F}_{q}(x)$ be a primitive element with minimal polynomial $f(X) = X^3 - 3X - a$ where $a \in \mathbb{F}_{q}(x)$. Let $a = \alpha/\gamma^3 \beta$ be the factorisation of $a$ in $\mathbb{F}_{q}(x)$, where $\alpha,\beta,\gamma \in \mathbb{F}_{q}[x]$, $(\alpha,\gamma^3 \beta) = 1$, and $\beta$ is cube-free. Let $\mathfrak{P}_1= \sigma (\mathfrak{P})$ and $\mathfrak{P}_2 =  \sigma^2 (\mathfrak{P})$, where $\sigma$ is a generator of $\text{\emph{Gal}}(L/\mathbb{F}_q(x))$. Then:
\begin{enumerate}
\item If $\mathfrak{p}| \beta$, then $v_{\mathfrak{P}} (y) = v_\mathfrak{p}(a) = -v_{\mathfrak{p}} (\gamma^3 \beta)$.
\item  If $\mathfrak{p}| \gamma$ and $\mathfrak{p}\nmid \beta$, then $v_{\mathfrak{P}} (y) = -v_{\mathfrak{p}} (\gamma)$.
\item If $\mathfrak{p}| \alpha$, then the places $\mathfrak{P},\mathfrak{P}_1,\mathfrak{P}_2$ are distinct, and exactly one element of $\{v_{\mathfrak{P}} (y), v_{\mathfrak{P}_2} (y),v_{\mathfrak{P}_1} (  y)\}$ is equal to $v_{\mathfrak{p}} (a)$, and the other valuations in this set are equal to $0$.
\item For the place $\mathfrak{p}_\infty$: \begin{enumerate} \item If $v_{\mathfrak{p}_\infty}(a)=0$, then $v_{{\mathfrak{P}_\infty}}(y)=v_{{\mathfrak{P}_\infty}_1}(y)=v_{{\mathfrak{P}_\infty}_2}(y)=0$.
\item If $v_{\mathfrak{p}_\infty}(a)>0$ then the places $\{\mathfrak{P}_\infty, {\mathfrak{P}_\infty}_1,{\mathfrak{P}_\infty}_2\}$ are distinct, exactly one of $v_{\mathfrak{P}_\infty} (y )$, $v_{{\mathfrak{P}_\infty}_2} ( y)$, $v_{{\mathfrak{P}_\infty}_1} (  y)$ is equal to $v_{\mathfrak{p}_\infty} (a)$, and the two other valuations in this set are equal to $0$.\end{enumerate}
\item For any other place $\mathfrak{p}$ of $K$, $v_{\mathfrak{P}} (y) = 0$.
\end{enumerate}

\end{theorem} 
\begin{proof} Suppose that $\mathfrak{p}| \beta$. By Theorem \ref{ramification}, we know that $\mathfrak{p}$ is (fully) ramified in $L$. By the identity $y^3 - 3y = a$, it follows as in the proof of Theorem 5.11 that, for the unique place $\mathfrak{P}$ of $L$ above $\mathfrak{p}$,  $v_{\mathfrak{\mathfrak{P}}}(y) < 0$. Thus $$3v_\mathfrak{P}(y)=\min\{v_\mathfrak{P}(y^3),v_\mathfrak{P}(3y)\} = v_{\mathfrak{P}}(y^3 - 3y)  = v_{\mathfrak{P}}(a) = 3 v_{\mathfrak{p}}(a),$$ whence $v_\mathfrak{P}(y) = v_\mathfrak{p}(a)$. 

Suppose that $\mathfrak{p}| \gamma$ and $\mathfrak{p}\nmid \beta$. By Theorem \ref{ramification}, we know that $\mathfrak{p}$ is unramified in $L$. Let $\mathfrak{P}$ be a place of $L$ above $\mathfrak{p}$. Analogously to the previous case, we have that as $v_{\mathfrak{p}}(a) < 0$, so must $$3v_\mathfrak{P}(y)=\min\{v_\mathfrak{P}(y^3),v_\mathfrak{P}(3y)\} = v_{\mathfrak{P}}(y^3 - 3y)  = v_{\mathfrak{P}}(a) = v_{\mathfrak{p}}(a) < 0,$$ and thus that $$3v_\mathfrak{P}(y) = v_\mathfrak{p}(a) = - v_\mathfrak{p}(\gamma^3) = -3v_\mathfrak{p}(\gamma).$$ Hence $v_\mathfrak{P}(y) = -v_\mathfrak{p}(\gamma)$. 

Suppose that $\mathfrak{p}|\alpha$.  From the generating polynomial $X^3 - 3X - a$ for $y$ and its conjugates $\sigma(y)$ and $\sigma^2(y)$, we obtain by the non-Archimedean property that if $v_\mathfrak{P}(w) > 0$ for one of $w \in \{y,\sigma(y),\sigma^2(y)\}$, then $v_\mathfrak{P}(w) = v_\mathfrak{p}(a)$.  \\
Furthermore, $v_\mathfrak{p}(a) \geq 0$ implies that $v_\mathfrak{P}(y) \geq 0$, $v_\mathfrak{P}(\sigma(y)) \geq 0$, and $v_\mathfrak{P}(\sigma^2(y)) \geq 0$. By definition of the minimal polynomial for $y$, we have $$X^3 - 3X - a = (X-y)(X-\sigma(y))(X-\sigma^2(y)),$$ and thus that $-y \cdot (-\sigma(y)) \cdot (-\sigma^2(y) )= -a$, whence $y \sigma(y) \sigma^2(y) = a$. Therefore, $$v_\mathfrak{p}(a) = v_\mathfrak{P}(a)  =  v_\mathfrak{P}(y\sigma(y)\sigma^2(y))
=v_\mathfrak{P}(y) + v_\mathfrak{P}(\sigma(y))+v_\mathfrak{P}(\sigma^2(y)).$$ 
As $v_\mathfrak{P}(w)\in \{ 0 , v_\mathfrak{p}(a)\}$, 
therefore, we obtain that one, and only one, $w \in \{y, \sigma(y),\sigma^2(y)\}$ has $v_\mathfrak{P}(w) \neq 0$, for this $w$ that $v_\mathfrak{P}(w)=v_{\mathfrak{p}}(a)$, and that the valuations of the other two conjugates at $\mathfrak{P}$ are equal to 0. Note that
$$v_\mathfrak{P}(\sigma(w))= v_{\sigma (\sigma^{-1}(\mathfrak{P}))}(\sigma(w))= v_{\sigma^{-1}(\mathfrak{P})}(w)= v_{\mathfrak{P}_2}(w)$$
and  
$$v_\mathfrak{P}(\sigma^2 (w))= v_{\sigma^2 (\sigma^{-2}(\mathfrak{P}))}(\sigma^2(w))= v_{\sigma^{-2}(\mathfrak{P})}(w)= v_{\mathfrak{P}_1}(w).$$
Thus, as one of $\{v_{\mathfrak{P}_1}(w),v_{\mathfrak{P}_2}(w)\}$ is distinct from $v_\mathfrak{P}(w)$, it follows that at least two of the places $\{\mathfrak{P},\mathfrak{P}_1,\mathfrak{P}_2\}$ are distinct. By \cite[Corollary 5.2.23]{Vil}, as $L/\mathbb{F}_q(x)$ is Galois, we have $efr = [L:\mathbb{F}_q(x)] = 3$, where $e= e(\mathfrak{P}|\mathfrak{p})$ is the ramification index of $\mathfrak{P}|\mathfrak{p}$, $f = f(\mathfrak{P}|\mathfrak{p})$ the inertia degree of $\mathfrak{P}|\mathfrak{p}$, and $r$ the number of places of $L$ above $\mathfrak{p}$. We have shown that $r > 1$. As $r \mid [L:\mathbb{F}_q(x)] = 3$, it follows that $r = 3$, and that $\mathfrak{P}$, $\mathfrak{P}_1$, and $\mathfrak{P}_2$ are distinct. The result follows.

For the place $\mathfrak{p}=\mathfrak{p}_\infty$, we have by Lemma \ref{infinity}, that $v_{\mathfrak{p}_\infty}(a)\geq 0$. If $v_{\mathfrak{p}_\infty}(a)=0$, then automatically $$v_{\mathfrak{P}_\infty}(y)=v_{{\mathfrak{P}_\infty}_1}(y)=v_{{\mathfrak{P}_\infty}_2}(y)=0=v_{\mathfrak{p}_\infty}(a).$$ If $v_{\mathfrak{p}_\infty}(a)>0$, then the valuation of $y$ is positive at one, and only one, place $\mathfrak{P}_\infty$ of $L$ above $\mathfrak{p}_\infty$; the proof of this is just the same as that for a place $\mathfrak{p}$ dividing $\alpha$, and for the place ${\mathfrak{P}_\infty}$, we obtain $v_{\mathfrak{P}_\infty}(y)=v_{\mathfrak{p}_\infty}(a)$. 
\end{proof}

Finally, we study the splitting behaviour of the unramified places for cubic extensions with a primitive element $y$ whose minimal equation is of the form $y^3-3y-a=0$. (Note that as $[L:\mathbb{F}_q(x)]=3$, splitting is trivial for all (fully) ramified places.)

\begin{theorem}
Let $q \equiv -1 \mod 3$, and let $L/\mathbb{F}_q(x)$ be a Galois cubic geometric extension with generating equation $X^3 - 3X - a = 0$, where $a \in \mathbb{F}_q(x)$. Let $\mathfrak{p}$ be a place of $\mathbb{F}_q(x)$ which is unramified in $L$ and $k(\mathfrak{p})$ the residue field of $\mathbb{F}_q(x)$ at $\mathfrak{p}$. \begin{enumerate}\item If $v_\mathfrak{p}(a) > 0$, then $\mathfrak{p}$ is totally split in $L$. 
\item  If $v_\mathfrak{p}(a)< 0$, then $ v_\mathfrak{p}(a)=-3m$, the place $\mathfrak{p}$ is finite, and $|k(\mathfrak{p})| \equiv 1 \mod 3$. The place $\mathfrak{p}$ splits completely in $L$ if, and only if, the reduction $f_{\mathfrak{p}}^{3m}a \mod \mathfrak{p}$ is a cube in $k(\mathfrak{p})$ where $f_{\mathfrak{p}} \in \mathbb{F}_q[x]$ is the irreducible polynomial associated with $\mathfrak{p}$. Otherwise, $\mathfrak{p}$ is inert in $L$.
\item If $v_\mathfrak{p}(a) = 0$, then
\begin{enumerate} \item If $p > 3$, then  $\mathfrak{p}$ is inert in $L$ if, and only if, \begin{enumerate} \item  $|k(\mathfrak{p})| \equiv 1 \mod 3$ and $\frac{1}{2} (\overline{a} + \delta)$ is not a cube in $k(\mathfrak{p})$, where $\delta^2 = \overline{a}^2-4$, or \item $|k(\mathfrak{p})|=q=5$ and $\overline{a}= \pm 1$. \end{enumerate}
Furthermore, with $a = \frac{P}{Q}$ with $P,Q\in \mathbb{F}_q[x]$, $(P,Q)=1$, Let $A,B \in \mathbb{F}_q[x]$ relatively prime be given as in Theorem 5.1 such that $$P = 2(A^2 - 3^{-1} B^2) \quad \text{and} \quad Q= A^2 + 3^{-1} B^2.$$ If $v_\mathfrak{p}(a) = 0$ and $|k(\mathfrak{p})| \equiv 1 \mod 3$, then $\mathfrak{p}$ is inert if, and only if, $\frac{A + \sqrt{-3}^{-1}B}{A - \sqrt{-3}^{-1}B}$ is not a cube mod $\mathfrak{p}$.
\item If $p = 2$, then with $|k(\mathfrak{p})| = 2^n$ and $tr: k(\mathfrak{p}) \rightarrow \mathbb{F}_2$ the trace map $$tr(\alpha)= \alpha + \alpha^2 + \alpha^{2^2} + \cdots + \alpha^{2^{n-1}},$$ the place $\mathfrak{p}$ is inert in $L$ if, and only if, $tr(1/\overline{a}^2) = tr(1)$ and \begin{enumerate}
\item $|k(\mathfrak{p})| \equiv 1 \mod 3$ and the roots of $T^2 + \overline{a}T +1$ are not cubes in $\mathbb{F}_{2^n} = k(\mathfrak{p})$, or 
\item $|k(\mathfrak{p})| \equiv 2 \mod 3$ and the roots of $T^2 +\overline{a}T +1$ are not cubes in $\mathbb{F}_{2^{2n}}= k(\mathfrak{p})(\sqrt{-3})$.
\end{enumerate}
We note that the roots of $T^2 +\overline{a}T +1$ lie in $\mathbb{F}_{2^n}$, respectively, $\mathbb{F}_{2^{2n}}$, depending on whether $tr(1/\overline{a}^2) = 0$ or $tr(1/\overline{a}^2) = 1$.
\end{enumerate}
  \end{enumerate}
\end{theorem}

\begin{proof} Throughout what follows in this proof, we will use $\mathfrak{P}$ to denote a place of $L$ above $\mathfrak{p}$.
\begin{enumerate}[1.] \item  This is immediate from Theorem 5.15 (3) and (4).
\item  If $v_{\mathfrak{p}} (a) < 0$, then by Theorem 5.11 and Corollary 5.12, $3|v_{\mathfrak{p}} (a)$ and the place $\mathfrak{p}$ must be finite. Moreover, $\mathfrak{p}$ is of even degree by Theorem 5.3, and as a consequence, $|k(\mathfrak{p})|\equiv 1\text{ mod } 3$. Let $f_\mathfrak{p}$ denote the irreducible polynomial corresponding to the place $\mathfrak{p}$, and let $m \in \mathbb{Z}$ be such that $v_{\mathfrak{p}}(a)= -3m$. Then, $z=f_\mathfrak{p}^m y$ is a root of the polynomial $X^3 -3 f_\mathfrak{p}^{2m} X -f_\mathfrak{p}^{3m} a$, and $v_{\mathfrak{p}} (f_\mathfrak{p}^{3m} a)=0$. In particular, we have $$X^3 -3 f_\mathfrak{p}^{2m} X -f_\mathfrak{p}^{3m} a \equiv X^3  -f_\mathfrak{p}^{3m} a \mod \mathfrak{p}.$$ Thus, $\mathfrak{p}$ is inert in $L$ if, and only if, $X^3  -f_\mathfrak{p}^{3m} a$ is irreducible over $k (\mathfrak{p})$, which occurs if, and only if, $f_\mathfrak{p}^{3m} a \mod \mathfrak{p} \in k(\mathfrak{p})$ is not a cube in $k(\mathfrak{p})$, and otherwise, as $L/\mathbb{F}_q(x)$ is Galois of prime degree $3$, $\mathfrak{p}$ is completely split in $L$.

\item If $v_\mathfrak{p}(a)=0$, we let $\overline{a} \in k(\mathfrak{p})$ be the reduction of $a$ modulo $\mathfrak{p}$.\begin{enumerate}[(a)] \item We first suppose that $p > 3$.

 \begin{enumerate}[i.] \item Suppose that $|k(\mathfrak{p})| \equiv 1 \mod 3$. Thus, $-3$ is a square in $k(\mathfrak{p})$ and $k(\mathfrak{p})(\sqrt{-3})=k(\mathfrak{p})$. The discriminant of $X^3 - 3X - a$ is equal to $\Delta = -27( a^2 -4)$, and it is a square as $L/\mathbb{F}_q(x)$ is Galois. As $-3$ is a square, the same is true for $a^2-4$. Thus, there is $\delta \in \mathbb{F}_q(x)$ such that $a^2 -4=\delta^2$. By \cite[Theorem 3]{Dickson}, the reduced polynomial $X^3 - 3X - \overline{a}$ is irreducible over $k(\mathfrak{p})$ if, and only if, (1) its discriminant 
 is a square in $k(\mathfrak{p})$ and (2) the element $\frac{1}{2} (\overline{a} + \delta) $ is not a cube in $k(\mathfrak{p})$. 
As $(1)$ is always true as $\Delta$ is a square, it follows that $\mathfrak{p}$ is inert if, and only if, $\frac{1}{2}(\overline{a} + \delta)$ is not a cube in $k(\mathfrak{p})$.

Moreover, by Theorem 5.1, writing $a = \frac{P}{Q}$ with $P,Q\in \mathbb{F}_q[x]$, $(P,Q)=1$, we know that there exist coprime $A,B \in \mathbb{F}_q[x]$ such that $$P = 2(A^2 - 3^{-1} B^2) \quad \text{and} \quad Q= A^2 + 3^{-1} B^2,$$ whence $$\Delta = 12^2 \frac{A^2 B^2 }{(A+3^{-1}B^2)^2} .$$ By definition, we also have 
$$\delta^2 = \overline{a}^2 - 4 \equiv -\frac{1}{3}\frac{16 A^2 B^2}{(A^2 + 3^{-1} B^2)^2}  \mod \mathfrak{p},$$
and hence 
$$\delta \equiv 4\sqrt{-3}^{-1} \left(\frac{AB}{A^2 + 3^{-1} B^2} \right) \mod \mathfrak{p}.$$ Therefore, 
\begin{align*} \frac{1}{2} (\overline{a} + \delta) &\equiv  \frac{1}{2}\left( \frac{2(A^2 - 3^{-1} B^2)}{A^2 + 3^{-1} B^2} +\frac{4}{\sqrt{-3}} \left(\frac{AB}{A^2 + 3^{-1} B^2} \right)\right)\mod \mathfrak{p} 
\\& \equiv \frac{A + \sqrt{-3}^{-1}B}{A - \sqrt{-3}^{-1}B} \mod \mathfrak{p}. \end{align*}

\item If $|k(\mathfrak{p})| \equiv 2 \mod 3$, then again by \cite[Theorem 3]{Dickson}, the irreducible cubic polynomials in $k(\mathfrak{p})[X]$ are given by \begin{equation}\label{Dickson2mod3}X^3 - 3X \nu^{-\frac{1}{3}(|k(\mathfrak{p})|-2)(|k(\mathfrak{p})| +1)} - \nu - \nu^{|k(\mathfrak{p})|} = 0,\end{equation} where $\nu \in k(\mathfrak{p})(\sqrt{-3})$ is not a cube in $k(\mathfrak{p})(\sqrt{-3})$. Thus, if $X^3 - 3X - \overline{a}$ is an irreducible polynomial in $k(\mathfrak{p})[X]$,  then $$\nu^{-\frac{1}{3}(|k(\mathfrak{p})|-2)(|k(\mathfrak{p})| +1)}  = 1$$ for a non-cube $\nu\in k(\mathfrak{p})(\sqrt{-3})$. We also have $$\frac{1}{3}(|k(\mathfrak{p})|-2)(|k(\mathfrak{p})| +1) =\frac{1}{3}|k(\mathfrak{p})|^2 - \frac{1}{3}|k(\mathfrak{p})| - \frac{2}{3} < |k(\mathfrak{p})|^2-1.$$ As $\nu\in k(\mathfrak{p})(\sqrt{-3})$, it must also be true that $\nu$ is a $(|k(\mathfrak{p})|^2-1)^{st}$ root of unity. Thus $$\frac{1}{3}(|k(\mathfrak{p})|-2)(|k(\mathfrak{p})| +1)  \Big|(|k(\mathfrak{p})|^2-1) = (|k(\mathfrak{p})|-1)(|k(\mathfrak{p})| +1) ,$$ so that $\frac{1}{3}(|k(\mathfrak{p})|-2)$ divides $(|k(\mathfrak{p})|-1)$, i.e., $$(|k(\mathfrak{p})|-2) \Big| 3(|k(\mathfrak{p})|-1).$$ As $(|k(\mathfrak{p})|-2,|k(\mathfrak{p})|-1)=1$, we obtain $(|k(\mathfrak{p})|-2 )| 3$. As $|k(\mathfrak{p})| \geq p > 3$, it follows that $|k(\mathfrak{p})| = p = 5$. Hence $$1 =\nu^{-\frac{1}{3}(|k(\mathfrak{p})|-2)(|k(\mathfrak{p})| +1)}  = \nu^{-6},$$ so that by \eqref{Dickson2mod3}, $\overline{a} = \nu^5 + \nu$ for the non-cube $6^{th}$ root of unity $\nu$. Let $\zeta \in k(\mathfrak{p})(\sqrt{-3}) = \mathbb{F}_{25}$ be a primitive $6^{th}$ root of unity. As $\nu$ is a non-cube in $\mathbb{F}_{25}$, it follows that $\nu = \zeta^i$ for some $i=1,2,4,5$, and thus that $\nu$ is either a primitive $3^{rd}$ or $6^{th}$ root of unity. In the case that $\nu$ is a primitive $3^{rd}$ root of unity, we have $$\overline{a} = \nu + \nu^5= \nu + \nu^2= -1,$$ whereas when $\nu$ is a primitive $6^{th}$ root of unity, as $\nu^2 = \nu-1$, we have $$\nu + \nu^5= \nu + \nu (\nu -1)^2= \nu^3-2\nu^2 +2\nu= \nu (\nu -1) -2(\nu -1) +2\nu = \nu^2-\nu +2=  1.$$ It follows that $\mathfrak{p}$ is inert in $L$ if, and only if, $|k(\mathfrak{p})|=q=5$ and $\overline{a}= \pm 1$.
  \end{enumerate} 
\item   If  $v_\mathfrak{p}(a) = 0$ and $p = 2$, then by \cite[Theorem 1]{Williams}, the polynomial $$X^3 - 3X - \overline{a} = X^3 + X + \overline{a} \in k(\mathfrak{p})[X]$$ is irreducible if over $k(\mathfrak{p})$, and only if, $tr(-27/\overline{a}^2) = tr(1)$ and the roots of $$T^2 - \overline{a}T + (-3)^3 = T^2  + \overline{a}T + 1 \in k(\mathfrak{p})[T]$$ are not cubes in $k(\mathfrak{p})(\sqrt{-3})$. The result then follows by noting that $\sqrt{-3} \in \mathbb{F}_{2^n} = k(\mathfrak{p})$ if, and only if, $|k(\mathfrak{p})| \equiv 1 \mod 3$, and otherwise that $k(\mathfrak{p})(\sqrt{-3}) = \mathbb{F}_{2^{2n}}$.
\end{enumerate}
\end{enumerate} 
\end{proof} 

\subsection{Integral basis}
The next result gives an explicit integral basis for a Galois extension with generating equation $y^3-3y-a=0$. We treat the cases $p \neq 2$ and $p = 2$ separately within this Theorem, as discriminants exhibit different properties in each case.
\begin{theorem} \label{integralbasis} Let $q = -1 \mod 3$. Let $L/\mathbb{F}_q(x)$ be a Galois cubic geometric extension with generator $y$ which satisfies the equation $y^3 - 3y - a=0$, where $a \in \mathbb{F}_q(x)$. As before, we let $a= \alpha/(\gamma^3 \beta)$ where $(\alpha,\beta\gamma) = 1$ and $\beta$ is cube-free. Furthermore, let $\beta = \beta_1 \beta_2^2$, where $\beta_1$ and $\beta_2$ are squarefree, let $\mathfrak{O}_L$ be the integral closure of $\mathbb{F}_q[x]$ in $L$, and let $\omega = \gamma \beta_1 \beta_2 y$. 

\begin{enumerate} \item Suppose that $p \neq 2$. Let $A$ and $B$ be as in Theorem \ref{qneq1mod3Galois}. Then $\theta,\kappa \in \mathbb{F}_q[x]$ may be chosen so that $$\theta \equiv -\alpha(2\gamma^2)^{-1}\beta_2^{-1} \mod (AB)^2\quad\text{and}\quad \theta \equiv \gamma \beta_1 \beta_2 \mod \beta_1^2,$$ $$\kappa \equiv -2(\gamma \beta_1 \beta_2 )^2 \mod (AB)^2 \beta_1^2,$$ $\delta \in \mathbb{F}_q[x]$ may be chosen freely, and the set $$\mathfrak{I}=\left\{1,\omega+\delta,\;(AB)^{-1} \beta_1^{-1}(\omega^2+\theta \omega+ \kappa)\right\}$$ forms a basis of $\mathfrak{O}_L$ over $\mathbb{F}_q[x]$.

 \item Suppose that $p = 2$. Let $A$ and $B$ be as in Theorem \ref{qneq1mod3Galois}. An integral basis of the form $\mathfrak{B}=\{ 1 , \omega +S , (\omega^2 +T\omega  +R)/I\}$ exists for some $S, T, R\in \mathbb{F}_q[x]$, where $T = \gamma \beta_1 \beta_2 + A \beta_1 H$ and $H\in \mathbb{F}_p[x]$ is chosen such that 
$$A+ B  =AG+  \beta_1\gamma^2 \beta_2 H,$$ for some $G \in \mathbb{F}_q[x]$.
\end{enumerate}\end{theorem}
\begin{remarque} Such a choice of $H$ as in Theorem \ref{integralbasis}(2) always exists, as $(B,\beta_1 \gamma \beta_2) = 1$ from $(A,B)=1$. \end{remarque}
\begin{proof}  The element $z = \gamma\beta y$ satisfies the equation $$z^3 - 3\gamma^2 \beta^2 z - \alpha \beta^2 = 0,$$ and $z$ is integral over $\mathbb{F}_q[x]$. Furthermore, for each finite place $\mathfrak{p}$ of $\mathbb{F}_q(x)$, let $f_\mathfrak{p} \in \mathbb{F}_q[x]$ be the polynomial associated with $\mathfrak{p}$, and let $$\beta_1=\prod_{\substack{\mathfrak{p}|\beta \\ v_\mathfrak{p}(\beta) = 1}} f_\mathfrak{p} \quad\text{and}\quad \beta_2=\prod_{\substack{\mathfrak{p}|\beta \\ v_\mathfrak{p}(\beta) = 2}}  f_\mathfrak{p}.$$ Hence $\beta=\beta_1 \beta_2^2$, and the generating equation for $z$ may therefore be written as \begin{align*} z^3 - 3\gamma^2 \beta_1^2 \beta_2^4 z - \alpha \beta_1^2 \beta_2^4 =  0.\end{align*} 
Division of this equation by $\beta_2^3$ and replacing $z$ with $\omega = \beta_2^{-1} z = \gamma \beta_1 \beta_2 y$ yields the equation $$\omega^3 - 3\gamma^2 \beta_1^2 \beta_2^2 \omega - \alpha \beta_1^2 \beta_2 = 0.$$ 

By definition, the discriminant of $\omega$ is equal to $$\Delta(\omega) = 108 \gamma^6 \beta_1^6 \beta_2^6 - 27 \alpha^2 \beta_1^4 \beta_2^2 = 27 \beta_1^4 \beta_2^2 (4\gamma^6 \beta_1^2 \beta_2^4 - \alpha^2)= 27 \beta_1^4 \beta_2^2 (4\gamma^6 \beta^2 - \alpha^2).$$ 

Let $\mathfrak{D}_{L/\mathbb{F}_q(x)}$ be the different of $L/\mathbb{F}_q(x)$ \cite[Section 5.6]{Vil}.
As the residue field extensions of a (fully) ramified prime $\mathfrak{p}$ of $\mathbb{F}_q(x)$ in $K$ are trivial, it follows from the definition of $\mathfrak{D}_{L/\mathbb{F}_q(x)}$ and the fact that the place $\mathfrak{p}_\infty$ is unramified in $K$ that $$\partial_{L/\mathbb{F}_q(x)}=N_{L/\mathbb{F}_q(x)}(\mathfrak{D}_{L/\mathbb{F}_q(x)}) = N_{L/\mathbb{F}_q(x)}\left(\prod_{\mathfrak{P}|\beta} \mathfrak{P}^2 \right) = \prod_{\mathfrak{P}|\mathfrak{p}|\beta} \mathfrak{p}^{2 f (\mathfrak{P}|\mathfrak{p})} = \prod_{\mathfrak{p}|\beta} \mathfrak{p}^2= ((\beta_1 \beta_2)^2)_{\mathbb{F}_q[x]},$$ where we let $\mathfrak{P}$ denote a place of $K$ above $\mathfrak{p}$. 

\begin{enumerate} \item Suppose that $p\neq 2$. By Theorem \ref{qneq1mod3Galois}, there exist coprime $A,B \in \mathbb{F}_q[x]$ such that $$\alpha = 2(A^2 - 3^{-1} B^2) \quad \text{and} \quad \gamma^3 \beta_1 \beta_2^2= \gamma^3 \beta = A^2 + 3^{-1} B^2.$$ It follows that \begin{align*} 4\gamma^6 \beta^2 - \alpha^2 &= 4(\gamma^3 \beta)^2 - \alpha^2 \\& = 4\left(A^2 + 3^{-1} B^2 \right)^2 - \left(2\left[A^2 - 3^{-1} B^2  \right]\right)^2 \\&=\frac{16}{3} (AB)^2. \end{align*} 
Thus, $$\Delta(\omega) = \frac{16}{3} \beta_1^4 \beta_2^2  (AB)^2.$$  We wish to show that $$\mathfrak{I}=\left\{1,\omega,\;(AB)^{-1} \beta_1^{-1}(\omega^2+\theta \omega+ \kappa)\right\}$$ forms an integral basis of $\mathfrak{O}_K$ over $\mathbb{F}_q[x]$, where $\theta$ and $\kappa$ are polynomials in $\mathbb{F}_q[x]$ which are chosen (the former by the Chinese Remainder Theorem) so that $$\theta \equiv -\alpha(2\gamma^2)^{-1}\beta_2^{-1} \mod (AB)^2\quad\text{and}\quad \theta \equiv \gamma \beta_1 \beta_2 \mod \beta_1^2,$$ and $$\kappa\equiv -2(\gamma \beta_1 \beta_2 )^2 \mod (AB)^2 \beta_1^2.$$ We first prove that $(AB)^{-1} \beta_1^{-1}(\omega^2+\theta \omega+ \kappa)$ is integral over $\mathbb{F}_q[x]$, i.e., 
$$ \omega^2+\theta \omega+ \kappa \equiv 0  \mod  AB \beta_1.$$ We have \begin{align*} (\omega^2 + \theta \omega + \kappa)^2 &= \omega^4 + 2\theta \omega^3 + (2\kappa + \theta^2) \omega^2 + 2\theta \kappa \omega + \kappa^2 \\& =  (\omega + 2\theta) \omega^3 + (2\kappa + \theta^2) \omega^2 + 2\theta \kappa \omega + \kappa^2 \\& = (\omega + 2\theta) (3\gamma^2 \beta_1^2 \beta_2^2 \omega + \alpha \beta_1^2 \beta_2) + (2\kappa + \theta^2) \omega^2 + 2\theta \kappa \omega + \kappa^2 \\&= (3\gamma^2 \beta_1^2 \beta_2^2 + 2\kappa + \theta^2) \omega^2 + (\alpha \beta_1^2 \beta_2+ 6\theta \gamma^2 \beta_1^2 \beta_2^2+ 2\theta \kappa) \omega + 2\theta \alpha \beta_1^2 \beta_2 + \kappa^2.  \end{align*} 
By definition of $\theta$, we have  $$\theta \equiv -\alpha(2\gamma^2)^{-1} \beta_2^{-1} \mod (AB)^2,$$ and hence that \begin{align*} \theta^2 &\equiv \alpha^2 (4\gamma^4)^{-1} \beta_2^{-2} \mod (AB)^2 \\&\equiv \gamma^2 \beta^2 \beta_2^{-2} \mod (AB)^2 \\&\equiv \gamma^2 \beta_1^2 \beta_2^4 \beta_2^{-2} \mod (AB)^2 \\& \equiv \gamma^2 \beta_1^2 \beta_2^2 \mod (AB)^2.\end{align*} Therefore \begin{align*} 3\gamma^2 \beta_1^2 \beta_2^2 + 2\kappa + \theta^2 & = 3\gamma^2 \beta_1^2 \beta_2^2 + 2\kappa +\theta^2  \\& \equiv 3\gamma^2 \beta_1^2 \beta_2^2  - 4 (\gamma \beta_1 \beta_2 )^2 + (\gamma \beta_1 \beta_2 )^2 \mod (AB)^2  \\&\equiv 0 \mod (AB)^2.\end{align*}
Also by definition of $\theta$, we obtain that \begin{align*} \theta^2 & \equiv  (\gamma \beta_1 \beta_2 )^2 \mod \beta_1^2  \\&\equiv 0\mod \beta_1^2.\end{align*} Thus, by definition of $\kappa$, it follows that \begin{align*} 3\gamma^2 \beta_1^2 \beta_2^2 + 2\kappa + \theta^2  &=  3\gamma^2 \beta_1^2 \beta_2^2  + 2\kappa +\theta^2 \mod  \beta_1^2 \\&\equiv 3\gamma^2 \beta_1^2 \beta_2^2  -4(\gamma \beta_1 \beta_2 )^2 \mod  \beta_1^2 \\& \equiv   - \gamma^2 \beta_1^2 \beta_2^2 \mod  \beta_1^2  \\& \equiv  0 \mod  \beta_1^2.\end{align*} 
As $(AB,\beta) = 1$, we have therefore proven that $$3 \gamma^2 \beta_1^2 \beta_2^2 + 2\kappa + \theta^2 \equiv 0 \mod  (AB)^2 \beta_1^2.$$
We also find that \begin{align*} \alpha \beta_1^2 \beta_2+ 6\theta \gamma^2 \beta_1^2 \beta_2^2+ 2\theta \kappa &= 2\theta(3 \gamma^2 \beta_1^2 \beta_2^2+\kappa) + \alpha \beta_1^2 \beta_2 \\& \equiv -2\alpha(2\gamma^2)^{-1}\beta_2^{-1}(3 \gamma^2 \beta_1^2 \beta_2^2 -2 (\gamma \beta_1 \beta_2 )^2) + \alpha \beta_1^2 \beta_2 \mod (AB)^2  \\& \equiv 2  (-\alpha(2\gamma^2)^{-1}\beta_2^{-1}(\gamma \beta_1 \beta_2 )^2) + \alpha \beta_1^2 \beta_2 \mod (AB)^2 \\& \equiv 0 \mod (AB)^2,\end{align*} 
and also, \begin{align*} \alpha \beta_1^2 \beta_2+ 6\theta \gamma^2 \beta_1^2 \beta_2^2+ 2\theta \kappa &= 2\theta(3 \gamma^2 \beta_1^2 \beta_2^2+\kappa) + \alpha \beta_1^2 \beta_2 \\& \equiv 2\gamma\beta_1 \beta_2 (3\gamma^2 \beta_1^2 \beta_2^2 -2 (\gamma\beta_1 \beta_2)^2) + \alpha \beta_1^2 \beta_2 \mod \beta_1^2 \\& \equiv 2\gamma^3\beta_1^3 \beta_2^3 + \alpha \beta_1^2 \beta_2 \mod \beta_1^2  \\& \equiv (2\gamma^3\beta_1 \beta_2^2 + \alpha ) \beta_1^2 \beta_2\mod \beta_1^2  \\&=0 \mod \beta_1^2.\end{align*} 
We have therefore shown that $$\alpha \beta_1^2 \beta_2 + 6\theta \gamma^2 \beta_1^2 \beta_2^2+ 2\theta \kappa \equiv 0 \mod (AB)^2 \beta_1^2.$$
Finally, we have \begin{align*} 2\theta \alpha \beta_1^2 \beta_2 + \kappa^2 & \equiv -2\alpha(2\gamma^2)^{-1}\beta_2^{-1}  \alpha \beta_1^2 \beta_2 + (-2(\gamma \beta_1 \beta_2 )^2)^2 \mod (AB)^2 \\& \equiv - \alpha^2 \beta_1^2(\gamma^2)^{-1} + 4\gamma^4\beta_1^4 \beta_2^4 \mod (AB)^2 \\& \equiv \beta_1 ^2(4\gamma^4\beta^2 - \alpha^2 (\gamma^2)^{-1}) \mod (AB)^2  \\& \equiv \beta_1^2(\gamma^2)^{-1}(4\gamma^6\beta^2 - \alpha^2) \mod (AB)^2 \\& \equiv 0 \mod (AB)^2. \end{align*} 
and \begin{align*} 2\theta \alpha \beta_1^2 \beta_2 + \kappa^2 &\equiv 2\gamma \beta_1 \beta_2 \alpha \beta_1^2 \beta_2 + (-2(\gamma \beta_1 \beta_2 )^2)^2 \mod \beta_1^2 \\&\equiv 2\gamma\alpha\beta_1^3 \beta_2^2 + 4\gamma^4\beta_1^4 \beta_2^4 \mod \beta_1^2 \\&\equiv \beta_1^3(2\gamma\alpha\beta_2^2 + 4\gamma^4 \beta_1 \beta_2^4) \mod \beta_1^2 \\&\equiv 0 \mod \beta_1^2.\end{align*} 
We have thus proven that with $\theta$ and $\kappa$ chosen as mentioned, \begin{align*} (\omega^2 + \theta \omega + \kappa)^2 &=  (3\gamma^2 \beta_1^2 \beta_2^2 + 2\kappa + \theta^2) \omega^2 + (\alpha \beta_1^2 \beta_2+ 6\theta \gamma^2 \beta_1^2 \beta_2^2+ 2\theta \kappa) \omega + 2\theta \alpha \beta_1^2 \beta_2 + \kappa^2 \\& \equiv 0 \mod (AB)^2 \beta_1^2,\end{align*} and hence that $$\omega^2 + \theta \omega + \kappa \equiv 0 \mod AB \beta_1.$$ It follows that the element $(AB)^{-1}\beta_1^{-1}(\omega^2+\theta \omega+ \kappa)$ is integral over $\mathbb{F}_q[x]$. As the extension $L/\mathbb{F}_q(x)$ is of degree 3 and $\omega$ generates $L$ over $\mathbb{F}_q(x)$, it follows that the three integral elements $1$, $\omega+\delta$, and $(AB)^{-1} \beta_1^{-1}(\omega^2 + \theta \omega + \kappa)$ are linearly independent over $\mathbb{F}_q(x)$, and hence that $\mathfrak{I}$ is a basis of $L /\mathbb{F}_q(x)$. Finally, the discriminant of the basis $\mathfrak{I}$ is equal to
\begin{align*} \Delta ( \mathfrak{I}) &= \left( \det  \left( \begin{array}{ccc}1&  \delta  & (AB)^{-1} \beta_1^{-1}\kappa \\0 & 1 &  (AB)^{-1} \beta_1^{-1}\theta \\0 & 0 &  (AB)^{-1} \beta_1^{-1}  \end{array} \right)\right)^2 \Delta(\omega ) \\& = \left( \det  \left( \begin{array}{ccc}1&  0  & 0\\0 & 1 & 0 \\0 & 0 & (AB)^{-1} \beta_1^{-1}  \end{array} \right)\right)^2 \Delta(\omega )  \\&= (AB)^{-2} \beta_1^{-2} \left(\frac{16}{3}  \beta_1^4 \beta_2^2  (AB)^2\right)\\&=\frac{16}{3}(\beta_1\beta_2)^2, \end{align*} and hence that $$(\Delta ( \mathfrak{I}))_{\mathbb{F}_q[x]}  = ((\beta_1\beta_2)^2)_{\mathbb{F}_q[x]}  = \partial_{L/\mathbb{F}_q(x)}.$$ By basic theory (see for example \cite[p. 398]{Scheidler}), it follows that $\mathfrak{I}$ is an integral basis for $\mathfrak{O}_L$ over $\mathbb{F}_q[x]$.  
\item Suppose that $p=2$. By Theorem \ref{qneq1mod3Galois}, there exist coprime $A,B \in \mathbb{F}_q[x]$ such that $$\alpha = A^2 \quad \text{and} \quad \gamma^3 \beta_1 \beta_2^2= \gamma^3 \beta = A^2 + AB+ B^2.$$ 
It therefore follows from $p=2$ that $$\Delta(\omega) = 27 \beta_1^4 \beta_2^2 (4\gamma^6 \beta^2 - \alpha^2)  =  \beta_1^4 \beta_2^2 A^4=  (\beta_1A^2)^2 \partial_{L/\mathbb{F}_q(x)}.$$
By \cite[Lemma 3.1, Corollary 3.2]{Scheidler},  a basis of the form $\mathfrak{B}=\{ 1 , y+S , (y^2 +Ty +R)/I\}$ for some $S, T, R\in \mathbb{F}_q[x]$ exists, if and only if, there exists $T \in \mathbb{F}_q[x]$ such that $$T^2 +(\gamma \beta_1 \beta_2)^2\equiv 0 \mod \beta_1A^2\quad \text{ and }\quad T^3 + (\gamma \beta_1 \beta_2)^2T +  \alpha \beta_1^2 \beta_2 \equiv 0 \mod \beta_1^2 A^4,$$
And if so, the set $$\mathfrak{I} = \left\{ 1, \omega+T, \frac{1}{I} (\omega^2+ T\omega + T^2 +(\gamma \beta_1 \beta_2)^2 ) \right\}$$ forms an integral basis of $L/ \mathbb{F}_q(x)$. We therefore investigate when such a $T$ exists. First, we note that the condition $$T^2 +(\gamma \beta_1 \beta_2)^2\equiv 0 \mod \beta_1A^2$$ 
is equivalent to 
$$T \equiv \gamma \beta_1 \beta_2  \mod A \ \text{and } T \equiv 0 \mod \beta_1.$$ 
This is equivalent to the existence of a polynomial $H \in \mathbb{F}_q[x]$ such that 
$$T = \gamma \beta_1 \beta_2 + A \beta_1 H.$$ 
Let us now choose such a $T$. Then, by definition we clearly have
$$ T^3 + (\gamma \beta_1 \beta_2)^2T +  \alpha \beta_1^2 \beta_2 \equiv 0 \mod \beta_1^2. $$
Moreover, $T$ is invertible mod $\beta_1$, since $(A, \gamma \beta_1 \beta_2)=1$. Thus, the condition
$$T^3 + (\gamma \beta_1 \beta_2)^2T +  \alpha \beta_1^2 \beta_2 \equiv 0 \mod \ A^4$$ 
is equivalent to 
$$T^4 + (\gamma \beta_1 \beta_2)^2T^2 +  \alpha \beta_1^2 \beta_2 T \equiv 0 \mod \ A^4.$$ 
We have 
\begin{align*} 
T^4 + (\gamma \beta_1 \beta_2)^2T^2 +  \alpha \beta_1^2 \beta_2  T \equiv & (\gamma \beta_1 \beta_2)^4 + (\gamma \beta_1 \beta_2)^2T^2 +  \alpha \beta_1^2 \beta_2 T \mod A^4\\
\equiv &\beta_1^2 \beta_2 ( \gamma^4 \beta_2^3 \beta_1^2 + \gamma^2 \beta_2 T^2 +  \alpha T))\mod  A^4 
 \end{align*}
Thus, the condition $$T^3 + (\gamma \beta_1 \beta_2)^2T +  \alpha \beta_1^2 \beta_2 \equiv 0 \mod \ A^4$$  
is equivalent to 
$$\gamma^4 \beta_2^3 \beta_1^2 + \gamma^2 \beta_2 T^2 +  \alpha T \equiv 0 \mod  A^4.$$ 
 Since $(\gamma^2 \beta_2, 1)=1$, this condition is in turn equivalent to 
$$\gamma^6 \beta_2^4 \beta_1^2 + \gamma^4 \beta_2^2 T^2 +  \alpha \gamma^2 \beta_2 T \equiv 0 \mod  A^4$$ 
and 
$$ B^4 + A^2 B^2 + \gamma^4 \beta_2^2 T^2 +  A^2 \gamma^2 \beta_2 T \equiv 0 \mod  A^4.\qquad\qquad(*)$$
That is, taking this equivalence mod $A^2$, we find in particular that
$$ B^4  + \gamma^4 \beta_2^2 T^2  \equiv 0 \mod  A^2.$$ Hence $$ \gamma^2 \beta_2 T \equiv B^2 \mod A.$$ We write $\gamma^2 \beta_2 T = B^2 + AG_0$ where $G_0 \in \mathbb{F}_q[x]$. From this and $(*)$, we obtain 
$$ A^2 G_0^2 +  A^3 G_0= A^2 G_0 ( G_0 + A) \equiv 0 \mod  A^4.$$ This implies in particular that $G_0 \equiv 0 \mod  A$.  Thus, $G_0 = A G$ for some $G \in \mathbb{F}_q[x]$, so that 
$$\gamma^2 \beta_2 T = B^2 + A^2G$$
Since 
$$T = \gamma \beta_1 \beta_2 + A \beta_1 H,$$ 
we have
$$ \gamma^2 \beta_2 T = \gamma^3 \beta_1 \beta_2^2 + A \beta_1\gamma^2 \beta_2 H.$$ 
We therefore obtain the reduction 
$$B^2 + \gamma^3 \beta_1 \beta_2^2  =A(AG+  \beta_1\gamma^2 \beta_2 H)$$
$$A^2 + AB  =A(AG+  \beta_1\gamma^2 \beta_2 H)$$
$$A+ B  =AG+  \beta_1\gamma^2 \beta_2 H.$$
As $(A,  \beta_1\gamma^2 \beta_2)=1$, such polynomials $G$ and $H$ must exist. In then follows by a similar argument to that for the discriminant $\Delta(\mathfrak{I})$ in the proof of part (1) that the desired integral basis $\mathfrak{I}$ exists, and that $T = \gamma \beta_1 \beta_2 + A \beta_1 H$ where $H\in \mathbb{F}_q[x]$ and $G \in \mathbb{F}_q[x]$ are chosen to satisfy
$$A+ B  =AG+  \beta_1\gamma^2 \beta_2 H.$$ Hence the result.
 \end{enumerate}
\end{proof}

\section*{Appendix}

Here, we restate Theorem \ref{reallylongtheorem} and give its proof.\\

\noindent {\bf Theorem \ref{reallylongtheorem}}.
{\it Suppose that $q \equiv -1 \mod 3$. Let $L_i = K(z_i)/K$ ($i=1,2$) be two cyclic extensions of degree $3$ such that $z_i^3 -3z_i=a_i\in K$. The following are equivalent: 
\begin{enumerate}
\item $L_1= L_2$;
\item $z_2=\phi  z_1^2 +\chi z_1 -2\phi$, where $\phi ,\chi \in K$ satisfy the equation $\chi^2+a_1\phi \chi+\phi^2=1.$ 
\end{enumerate}
Moreover, if $K = \mathbb{F}_q(x)$, when the above conditions are satisfied, then:
\begin{enumerate}[(a)] 
\item If $p\neq 2$,  there are relatively prime polynomials $C,D\in \mathbb{F}_q[x]$ 
$$\phi = \frac{ 12C DQ_1}{\delta ( D^2 +3^{-1} C^2)}\quad\quad\text{and}\quad\quad \chi =\frac{-6P_1CD\pm \delta(C^2 -3^{-1} D^2)}{\delta( D^2 +3^{-1}  C^2)},$$ 

where $\delta \in K$ is such that $\delta^2=Q_1^2 d_{L/\mathbb{F}_q(x)}$ and $a_1= P_1/Q_1$, $P_1$, $Q_1\in \mathbb{F}_q[x]$ with $(P_1, Q_1)=1$ and $$a_2 =(-2 + a_1^2) \phi^3 + 3 a_1 \phi^2 \chi + 6 \phi \chi^2 + a_1 \chi^3.$$
\item If $p=2$,
\begin{enumerate}[(i)]
\item if $\phi =0$, then $\chi =\pm 1$. If $\chi =1$, then $a_2=a_1$, and if $\chi =-1$, then $a_2=-a_1$.
\item if $\phi \neq 0$, then $\chi /(a_1 \phi)$ is a solution of 
$$X^2-X-\frac{\phi^2-1}{a_1 \phi }=0$$ and 
$$a_2 = a_1(a_1 \phi^3 +   \phi^2 \chi  +  \chi^3).$$
\end{enumerate}
\end{enumerate}}  
\begin{proof} Suppose that $L_1=L_2$. As $\{1, z_1,z_1^2 \}$ is a basis of $L_1=L_2$ over $K$, it follows that there are $\phi,\chi  ,\psi \in  K$ such that 
$$z_2 = \phi z_1^2 +\chi z_1 +\psi.$$ By Theorem \ref{explicite}, we have that $$\sigma (z_1) = -u z_1^2 +f z_1 + 2u$$ and $$\sigma^2 (z_1)= u z_1^2 +(-1-f) z_1 - 2u$$ for a generator $\sigma$ of $\text{Gal}(L/K)$ and where $u=\frac{1+2f}{a_1}$, $u^2 = f^2 +f+1$, and $f$ is one root of the polynomial $$S(X)=\left(1 - \frac{4}{a_1^2}\right) x^2 + \left( 1- \frac{4}{a_1^2}\right) x + \left(1-\frac{1}{a_1^2}\right).$$ 
Thus, we have
\begin{align*} \sigma (z_2) &=  \phi \sigma (z_1)^2 +\chi \sigma(z_1) +\psi 
\\&= \phi (-u z_1^2 +f z_1 + 2u)^2 + \chi (-u z_1^2 +f z_1 + 2u) + \psi 
\\ &= ( 3 \phi u^2 + \phi f^2 - 4  u^2 \phi -\chi u)z_1^2 + ( \phi u^2 a_1 -2 fu\phi + f\chi )z_1 \\&\quad\;\; + (4 u^2 \phi -2ufa_1\phi + 2u\chi+\psi )
\end{align*} and 
\begin{align*} \sigma^2 (z_2) &=  \phi \sigma^2 (z_1)^2 +\chi \sigma^2(z_1) +\psi 
\\&= \phi  (u z_1^2 -(f+1) z_1 - 2u)^2 + \chi (u z_1^2 -(f+1) z_1 - 2u) + \psi 
\\ &= ( 3 \phi u^2 + \phi (f+1)^2 - 4  u^2 \phi +\chi u)z_1^2 + ( \phi u^2 a_1 -2 (f+1)u\phi - \chi (1+f) )z_1 
\\&\quad\;\;+ (4 u^2 \phi  -2u(f+1)a_1\phi - 2u\chi +\psi ).\end{align*}
As $z_2$ satisfies $z_2^3 -3 z_2 =a_2$, we have 
$$Tr(z_2) = z_2 + \sigma (z_2) + \sigma^2 (z_2) =0.$$ 
Thus, as $ua_1= 2f+1$, we obtain 
\begin{align*} 0 &= z_2 + \sigma (z_2) + \sigma^2 (z_2) 
\\ &= (\phi + 3 \phi u^2 + \phi f^2 - 4  u^2 \phi  -\chi u + 3 \phi u^2 + \phi (f+1)^2 - 4  u^2 \phi +\chi u )z_1^2 
\\ & \quad\;\;+ (\chi +\phi u^2 a_1 -2 fu\phi + f\chi + \phi u^2 a_1 -2 (f+1)u\phi - \chi (1+f)) z_1 
\\& \quad\;\;+ (\psi + 4 u^2 \phi -2ufa_1\phi  + 2u\chi +\psi + 4 u^2 \phi  -2u(f+1)a_1\phi - 2u\chi +\psi) 
\\&= (2\phi -2 \phi u^2 + 2\phi f^2+2\phi f )z_1^2  + (2\phi u^2 a_1 -4 fu\phi   -2 u\phi ) z_1 
\\&\quad\;\; + (3\psi + 8 u^2 \phi -4ufa_1\phi  -2ua_1\phi ) 
\\&= 2\phi (- u^2 + f^2+f +1)z_1^2  + 2u\phi (u a_1 -2 f   -1) z_1 
\\& \quad\;\;+ (3\psi + 2u\phi (4 u  -2fa_1  -a_1 ) 
\\&= 2\phi (- u^2 + f^2+f +1)z_1^2  + 2u\phi (u a_1 -2 f   -1)  z_1 
\\& \quad\;\;+ ( 3r\psi + 2u^2\phi (4-a_1^2)).\end{align*}
This yields to the system
$$\left\{ \begin{array}{cccr} 
2\phi (- u^2 + f^2+f +1) &=& 0 & \\ 
 2u\phi (u a_1 -2 f   -1)  &=& 0 &\\
3\psi + 2u^2\phi (4-a_1^2) &=& 0 &  \\
\end{array} \right.$$
The first two simplify to zero identically, which leaves only $ 3\psi  + 2u^2\phi (4-a_1^2)=0$. As $u^2 a_1^2= (2f+1)^2$ and $u^2 = f^2 +f +1$, we obtain
\begin{align}\label{star} u^2 (4-a_1^2)&=4 u^2 -a_1^2 u^2= 4 (f^2 + f +1) - (2f+1)^2=4-1 =3. \end{align}
Thus the last equation of the previous system simplifies to $0=3\cR +6\cP=3(\cR +2\cP )$, whence 
$$\psi =-2\phi \ \ \  ( \star ) $$
Inspection of the linear term yields $$z_2 \sigma(z_2) + z_2 \sigma^2 (z_2)+ \sigma (z_2) \sigma^2 (z_2) = - z_2 \sigma(z_2) - z_2^2 - \sigma (z_2)^2 = -3.$$ 
To simplify the computation, we write $$\sigma (z_2) = \phi ' z_1^2 + \chi' z_1 + \psi',$$ with 
$$\left\{ \begin{array}{ll}
\phi ' &= - \phi u^2 + \phi f^2  -\chi u \\ 
\chi ' &= \phi u^2 a_1 -2 fu\phi + f\chi \\  
\psi ' &= 4 u^2 \phi -2ufa_1\phi + 2u\chi +\psi .
\end{array}\right.  \ (\square )$$
We thus obtain
\begin{align*} 3 &=  z_2 \sigma(z_2) + z_2^2 + \sigma (z_2)^2
\\ &= ( \phi z_1^2 + \chi z_1 + \psi) ( \phi ' z_1^2 + \chi' z_1 + \psi')+ ( \phi z_1^2 + \chi z_1 + \psi)^2 
\\
&\quad\;\;+ ( \phi' z_1^2 + \chi' z_1 + \psi')^2\\ &= (3\phi \phi' + \phi \psi'+\phi \phi' + \psi \phi' +3\phi^2  + \chi^2 +2\phi \psi + 3\phi'^2+ \chi '^2 +2\phi ' \psi' )z_1^2 
\\ &\quad\;\; + (\phi \phi'a_1 + 3\phi \chi' + 3\chi \phi' + \psi' \chi+ \psi \chi'+ \phi^2 a_1 
\\ 
&\quad\;\; + 6\phi \chi + 2\chi \psi +\phi '^2 a_1 + 6\phi '\chi' + 2 \chi' \psi')z_1
\\ &\quad\;\; + (  \phi \chi' a_1 +\chi \phi'a_1 + \psi \psi' +\psi^2 + 2\phi \chi a_1 +\psi'^2 + 2\phi' \chi' a_1). \end{align*}
As a consequence, we obtain the system 
$$\left\{ \begin{array}{llll} 
0&=& 3\phi \phi ' + \phi \psi'+\cQ \cQ' + \cR \cP' +3\cP^2  + \cQ^2 +2\cP \cR + 3\cP'^2+ \cQ'^2 +2\cP' \cR & \ \text{(i)} \\ 
0 &=& \cP \cP'a_1 + 3\cP \cQ' + 3\cQ \cP' + \cR' \cQ+ \cR \cQ'+ \cP^2 a_1 + 6\cP \cQ \\
&&+ 2\cQ \cR+\cP'^2 a_1 + 6\cP' \cQ' + 2 \cQ' \cR' & \ \text{(ii)} \\
0 &=& \cP \cQ' a_1 +\cQ \cP'a_1 + \cR \cR' +\cR^2 + 2\cP \cQ a_1 +\cR'^2 + 2\cP' \cQ' a_1 -3 &  \ \text{(iii)}
\end{array}\right.$$
We now work with Eq. (i). By definition, Eq. (i) simplifies to
\begin{align*} 
0&= 3\cP \cP' + \cP \cR'+\cQ \cQ' + \cR \cP' +3\cP^2  + \cQ^2 +2\cP \cR + 3\cP'^2+ \cQ'^2 +2\cP' \cR'\\ &= 3\cP(- \cP u^2 + \cP f^2  -\cQ u ) + \cR (- \cP u^2 + \cP f^2  -\cQ u ) +\cQ (\cP u^2 a_1 -2 fu\cP + f\cQ) \\&\quad\;\; +\cP(4 u^2 \cP -2ufa_1\cP + 2u\cQ +\cR ) +3\cP^2  + \cQ^2 +2\cP \cR + 3(- \cP u^2 + \cP f^2  -\cQ u )^2
\\&\quad\;\; + (\cP u^2 a_1 -2 fu \cP + f \cQ )^2 +2(- \cP u^2 + \cP f^2 -\cQ u ) (4 u^2 \cP -2ufa_1\cP + 2u \cQ+ \cR)
\\&= \cP^2(u^2 + 3f^2 -2ufa_1 -5u^4 +3+ 3f^4 + u^4a_1^2 +6 u^2 f^2 -4uf^3a_1) \\ &\quad\;\; + \cQ^2(f+ 1-u^2+ f^2) -3u \cQ \cR +3 \cP \cR(1 + f^2- u^2)\\&\quad\;\; +\cP \cQ u(-1  +u a_1 -2 f -6 u^2 -6f^2+6ufa_1).
\end{align*}
Using that $f^2 +f +1 -u^2 =0$ and $ua_1=2f+1$, the previous equation becomes
\begin{align*} 
0&= \cP^2(1+f+f^2 + 3f^2 -2f(2f+1) -5(1+f+f^2)^2  -4f^3(2f+1) \\&\quad\;\;+ (1+f+f^2)(2f+1)^2 +6 (1+f+f^2) f^2 +3+ 3f^4) -3u \cQ \cR -3f \cP \cR \\ &\quad\;\;   +\cP \cQ u(-1  +2f+1 -2 f -6 (f^2+f+1) -6f^2+6f(2f+1))\\&= -6f \cP^2  -3\cQ \cR u -3 f\cP \cR +\cP \cQ u( -6 f^2-6f-6 -6f^2+12f^2 +6f)\\&= -6f \cP^2  -3\cQ \cR u -3 f\cP \cR -6\cP \cQ u.
\end{align*}
Finally, as $\cR =- 2\cP$, the equation $(i)$ $$-6f \cP^2  -3\cQ \cR u -3 f\cP \cR -6\cP \cQ u= -6f \cP^2  +6\cP \cQ u +6 f\cP^2 -6\cP \cQ u =0$$ is always satisfied.\\ 
For Eq. (ii), we find
\begin{align*}
    0 &= \cP \cP'a_1 + 3\cP \cQ' + 3\cQ \cP' + \cR' \cQ+ \cR \cQ'+ \cP^2 a_1 + 6\cP \cQ + 2\cQ \cR+\cP'^2 a_1 + 6\cP'\cQ' + 2 \cQ'\cR'  \\& = \cP(- \cP u^2 + \cP f^2  -\cQ u )a_1 + 3\cP (\cP u^2 a_1 -2 fu\cP + f\cQ) + 3\cQ(- \cP u^2 + \cP f^2  -\cQ u ) 
    \\&\quad\;\; + (4 u^2 \cP -2ufa_1\cP + 2u\cQ+\cR) \cQ+ \cR (\cP u^2 a_1 -2 fu\cP + f\cQ )+ \cP^2 a_1 + 6\cP \cQ \\&\quad\;\;+ 2\cQ \cR +(- \cP u^2 + \cP f^2  -\cQ u )^2a_1 + 6(- \cP u^2 + \cP f^2  -\cQ u )(\cP u^2 a_1 -2 fu\cP + f\cQ ) \\&\quad\;\; + 2 (\cP u^2 a_1 -2 fu\cP + f\cQ )(4 u^2 \cP -2ufa_1\cP + 2u\cQ +\cR )\\& =\cP^2(2 u^2a_1 + f^2a_1 -6 fu+  a_1+ 3u^4a_1 -4u^3 fa_1^2-4 fu^3+ 12f^2u^2a_1  + f^4a_1 -12uf^3  )\\&\quad\;\; +\cQ^2(-u+u^2a_1-2fu) + rq(3+ 3f)+ rp( -6 fu +3u^2 a_1 ) \\&\quad\;\;+\cP \cQ (-ua_1   + 3f +u^2 + 3f^2 + 6  -2ufa_1 + 6fu^2  -6 f^2ua_1  + 6f^3 )\\& =\cP^2(2 u^2a_1 + f^2a_1 -6 fu+  a_1+ 3u^4a_1 -4u^3 fa_1^2-4 fu^3+ 12f^2u^2a_1  + f^4a_1 -12uf^3  )\\&\quad\;\; +\cQ^2 (-u(2f+1)+u^2a_1) + 3\cQ \cR (f+1)+ \cP \cR (3u^2 a_1 -6 fu ) \\&\quad\;\;+\cP \cQ (-ua_1   + 3f +u^2 + 3f^2 + 6  -2ufa_1 + 6fu^2  -6 f^2ua_1  + 6f^3 ) \end{align*}
As $f^2 +f+1 -u^2=0$ and $ua_1=2f+1$, this simplifies to \begin{align*}0 & =\cP^2[2 u(2f+1) + f^2a_1 -6 fu+  a_1+ 3u^4a_1 -4u^3 fa_1^2-4 fu^3+ 12f^2u^2a_1  + f^4a_1 -12uf^3 ]\\& \quad\;\;+ 3\cQ \cR (f+1)+ \cP \cR [3u(ua_1-2f)] +\cP \cQ [-(2f+1)   + 3f +f^2+f +1 + 3f^2 + 6 \\& \quad\;\; -2f(2f+1) + 6f(f^2+f+1) -6 f^2(2f+1)  + 6f^3]\\& =\cP^2[2 u^2a_1 + f^2a_1 -6 fu+  a_1+ 3u^4a_1 -4u^3 fa_1^2-4 fu^3+ 12f^2u^2a_1  + f^4a_1 -12uf^3 ]\\& \quad\;\; + 3\cR \cQ(f+1)+ 3u\cP \cR +\cP \cQ( 6    +6f   ).\\& =\cP^2(2 u^2a_1 + f^2a_1 -6 fu+  a_1+ 3u^4a_1 -4u^3 fa_1^2-4 fu^3+ 12f^2u^2a_1  + f^4a_1 -12uf^3 -6u)\end{align*} where the last equality holds as $\cR=-2\cP$. By $f^2 =-f -1 +u^2$ and $ua_1 = 2f+1$, we therefore obtain
\begin{align*} 0 &= 2 u^2a_1 + f^2a_1 -6 fu+  a_1+ 3u^4a_1 -4u^3 fa_1^2-4 fu^3+ 12f^2u^2a_1  + f^4a_1 -12uf^3 -6u \\ &= 2 u^2a_1 + (-f-1+u^2)a_1 -6 fu+  a_1+ 3u^4a_1 -4u^3 fa_1^2-4 fu^3+ 12(-f-1 +u^2) u^2a_1 \\&\quad\;\; + (-f-1 +u^2)^2a_1 -12uf(-f-1 +u^2) -6u\\ &=  fa_1 -11u^2 a_1 +6 fu+ 16u^4a_1 -4u^3 fa_1^2-16 fu^3 -14u^2 a_1f   + f^2 a_1+a_1    +12uf^2 -6u\\ &=  fa_1 -11u^2 a_1 +6 fu+ 16u^4a_1 -4u^3 fa_1^2-16 fu^3 -14u^2 a_1f   + (-f-1+u^2)a_1+a_1 \\&\quad\;\; +12u(-f-1+u^2) -6u\\ &=   -10u^2 a_1 -6 fu+ 16u^4a_1 -4u^3 fa_1^2-16 fu^3 -14u^2 a_1f  -18u+12u^3 \\&=   -10u (2f+1)-6 fu+ 16u^3(2f+1) -4u f(2f+1)^2-16 fu^3 -14uf(2f+1)  -18u+12u^3 \\&=  -44fu -28u +16u^3 -44f^2u -16f^3 u  +12u^3 \\&=  -44fu -28u +16u(f^2+f+1)f-16f^3 u  -44f^2u   +28u^3\\ &= -28 fu -28u -28 f^2 u + 28 u^3 \\ &= -28 u(f+ f^2 +1 -u^2), \end{align*}
which is always satisfied. \\

For Eq. (iii), we have
\begin{align*} 0&= \cP \cQ' a_1 +\cQ \cP'a_1 + \cR \cR' +\cR^2 + 2\cP \cQ a_1 +\cR'^2 + 2\cP'\cQ' a_1 -3\\ &= \cP a_1( \cP u^2 a_1 -2 fu\cP+ f\cQ ) +\cQ a_1(- \cP u^2 + \cP f^2  -\cQ u) + \cR (4 u^2 \cP -2ufa_1\cP + 2u\cQ +\cR ) \\&\quad\;\;+\cR^2 + 2\cP \cQ a_1 +(4 u^2 \cP -2ufa_1\cP + 2u\cQ +\cR )^2 \\&\quad\;\;+ 2(- \cP u^2 + \cP f^2  -\cQ u)( \cP u^2 a_1 -2 fu\cP + f\cQ ) a_1 -3\\ &= \cP^2 (-2 a_1fu - 4a_1f^3u + a_1^2u^2 + 6 a_1^2f^2u^2 - 12a_1f u^3 + 16u^4 - 2a_1^2u^4)\\& \quad\;\;+ \cP \cR (-6a_1fu+12u^2)+\cP \cQ(2a_1+a_1f+a_1f^2 +2a_1f^3-a_1u^2-6a_1fu^2+16u^3-2 a_1^2u^3)\\&\quad\;\;+\cQ^2(-a_1u-2a_1fu+4u^2)+6\cQ \cR u+3\cR^2-3. \end{align*} 
As $\cR=-2\cP$, this becomes 
\begin{align*}0&= \cP^2 (-2 afu - 4a_1f^3u + a_1^2u^2 + 6 a_1^2f^2u^2 - 12a_1f u^3 + 16u^4 - 2a_1^2u^4)\\&\quad\;\; +\cP \cQ(2a_1+a_1f+a_1f^2 +2a_1f^3-a_1u^2-6a_1fu^2+16u^3-2 a_1^2u^3)\\&\quad\;\;-2 \cP^2(-6a_1fu+12u^2)+\cQ^2(-a_1u-2a_1fu+4u^2)-12\cQ \cP u+12\cP^2-3  \\&= \cP^2 (-2 a_1fu - 4a_1 f^3u + a_1^2u^2 + 6 a_1^2f^2u^2 - 12a_1f u^3 + 16u^4 - 2a_1^2u^4+12a_1fu-24u^2 +12)\\&\quad\;\; +\cP \cQ (2a_1+a_1f+a_1f^2 +2a_1f^3-a_1u^2-6a_1fu^2+16u^3-2 a_1^2u^3-12u)\\& \quad\;\;+\cQ^2(-a_1u-2a_1fu+4u^2)-3. \\&= \cP^2 (- 4a_1f^3u + a_1^2u^2 + 6 a_1^2f^2u^2 - 12a_1f u^3 + 16u^4 - 2a_1^2u^4+10a_1fu-24u^2 +12)\\&\quad\;\; +\cP \cQ(2a_1-a_1u^2-5a_1fu^2+16u^3-2 a_1^2u^3-12u)+u^2(4-a_1^2) \cQ^2-3. \end{align*}
By Eq. \eqref{star}, this is equal to \begin{align*} &\quad\;\;\cP^2 (- 4a_1f^3u + a_1^2u^2 + 6 a_1^2f^2u^2 - 12a_1f u^3 + 16u^4 - 2a_1^2u^4+10a_1fu-24u^2 +12)\\&\quad\;\; +\cP \cQ (2a_1-a_1u^2-5a_1fu^2+a_1f^3+16u^3-2 a_1^2u^3-12u)+3\cQ^2-3=0 \end{align*} As $f^2 +f +1-u^2=0$ and $ua_1 = 2f+1$, the coefficient of $\cP \cQ$ in Eq. (iii) becomes
\begin{align*} 0 &= 2a_1-a_1u^2-5a_1fu^2+a_1f^3+16u^3-2 a_1^2u^3-12u\\ &= 2a_1 -a_1u^2-5a_1fu^2+a_1f(-f-1+ u^2)+16u^3-2 a_1^2u^3-12u \\ &= 2a_1 -a_1u^2-5a_1fu^2-f^2a_1-fa_1+ fu^2a_1+16u^3-2 a_1^2u^3-12u \\ &= 2a_1 -a_1u^2-4a_1fu^2-a_1(-f-1+u^2) -fa_1+16u^3-2 a_1^2u^3-12u \\ &= 3a_1 -2a_1u^2-4a_1fu^2+16u^3-2 a_1^2u^3-12u \\ &= 3a_1 -2a_1u^2-4a_1fu^2+16u^3-2 a_1u^2(2f+1)-12u \\ &= 3a_1 -4a_1u^2-8a_1fu^2+16u^3-12u \\ &= 3a_1 -4u(2f+1)-8(2f+1)fu+16u(f^2+f+1)-12u \\ &= 3a_1. \end{align*}
Also as $f^2 +f +1-u^2=0$ and $ua_1 = 2f+1$, the coefficient of $\cP^2$ in Eq. (vi) is equal to 
\begin{align*} &- 4a_1f^3u + a_1^2u^2 + 6 a_1^2f^2u^2 - 12a_1f u^3 + 16u^4 - 2a_1^2u^4+10a_1fu-24u^2 +12\\ &= - 4f^3(2f+1) + (2f+1)^2 + 6 f^2(2f+1)^2 \\&\quad\;\;- 12f (2f+1)(f^2+f+1) + 16(f^2+f+1)- 2(2f+1)(f^2+f+1)\\&\quad\;\; +10f(2f+1)-24(f^2+f+1) +12\\ &= 3.\end{align*}
Therefore, Eq. (ii) becomes $3\cQ^2+3a_1\cP \cQ+3\cP^2- 3 = 0$, or equivalently, 
$$\cQ^2+a_1\cP \cQ+\cP^2=1 \ \ \ (\star \star )$$

For the norm term in the equation $z_2^3-3z_2-a_2=0$ satisfied by $z_2$, we have $$-a_2=-z_2 \sigma (z_2 ) \sigma^2 (z_2) = z_2\sigma(z_2)(z_2+\sigma(z_2)).$$ Hence
\begin{align*} -a_2 &= z_2\sigma(z_2)(z_2+\sigma(z_2))\\&= (\cP z_1^2 +\cQ z_1 +\cR ) (\cP'z_1^2 +\cQ' z_1 + \cR') ((\cP +\cP')z_1^2+(\cQ +\cQ')z_1+(\cR +\cR'))\\ &= \cP'(\cP +\cP')\cP z_1^6+(\cQ'\cP^2+2\cP'\cP \cQ +2\cP'\cQ'\cP +\cP'^2\cQ' )z_1^5\\&\quad\;\;+(\cR'\cP^2+2\cQ'\cP\cQ +2\cP'\cP\cR +\cP'\cQ^2+2\cP'\cR'\cP +\cQ'^2\cP +2\cP'\cQ'\cQ +\cP'^2\cR )z_1^4\\&\quad\;\;+(2\cR'\cP\cQ +2\cQ'\cP \cR +\cQ'\cQ^2+2\cP'\cQ\cR +2\cR'\cQ'\cP +\cQ'^2\cQ+2\cR'\cP'\cQ +2\cQ'\cP'\cR )z_1^3\\&\quad\;\;+(2\cR' \cP\cR+\cR'\cQ^2+2\cQ'\cQ \cR +\cP'\cR^2+\cR'^2\cP+2\cQ'\cR'\cQ +2\cP'\cR'\cR +\cQ'^2\cR )z_1^2\\&\quad\;\;+(2\cR' \cQ \cR +\cQ'\cR^2 +\cR'^2\cQ +2\cR'\cQ'\cR )z_1+\cR' (\cR +\cR')\cR.
\end{align*}
By construction, $z_1^3 = 3 z_1 +a_1$, so that the previous equation simplifies to
\begin{align*}
0&= (9\cP'\cP^2+a_1\cQ'\cP^2+3\cR'\cP^2+6\cQ'\cP\cQ +2a_1\cP'\cP\cQ +6\cP'\cP\cR +2\cR'\cP\cR +3\cP'\cQ^2\\&\quad\;\;+\cR'\cQ^2+2\cQ'\cQ\cR +\cP'\cR^2+9\cP'^2\cP +2a_1\cP' \cQ'\cP+3\cQ'^2\cP+6\cP'\cR'\cP +\cR'^2\cP \\&\quad\;\;+a_1\cP'^2\cQ+6\cP'\cQ'\cQ +2\cQ'\cR'\cQ +\cQ'^2\cR +3\cP'^2\cR+2\cP'\cR'\cR )z_1^2\\&\quad\;\;+(6\cP'a_1\cP^2+a_1\cR'\cP^2+9\cQ'\cP^2+6\cR'\cP\cQ +18\cP'\cP\cQ +2a_1\cQ'\cP\cQ+6\cQ'\cP\cR\\&\quad\;\;+a_1\cP'\cQ^2+3\cQ'\cQ^2+2\cR'\cQ\cR +6\cP'\cQ\cR +\cQ'\cR^2+18\cP'\cQ'\cP +6\cR'\cQ'\cP \\&\quad\;\;+6\cP'^2a_1\cP+2a_1\cP'\cR'\cP+2a_1\cP'\cP\cR+a_1\cQ'^2\cP +6\cR'\cP'\cQ +2a_1\cP'\cQ'\cQ \\
&\quad\;\;+3\cQ'^2\cQ +9\cP'^2\cQ +\cR'^2\cQ +2\cR'\cQ'\cR +6\cQ'\cP'\cR+a_1\cP'^2\cR )z_1\\&\quad\;\;+3a_1\cQ'\cP^2+6a_1\cP'\cP\cQ +\cP'a_1^2\cP^2+a_1\cQ'\cQ^2+2a_1\cP'\cQ\cR +\cR'\cR^2+6a\cP'\cQ'\cP \\&\quad\;\;+2a\cR'\cQ'\cP +\cP'^2a^2\cP+3a_1\cP'^2\cQ+2a_1\cR'\cP'\cQ +a_1\cQ'^2\cQ+\cR'^2\cR+2a_1\cQ'\cP\c'R\\
&\quad\;\;+2a_1\cR'\cP\cQ +2a_1\cQ'\cP\cR. \end{align*}
As $\{z_1^2,z_1,1\}$ forms a basis of $L/K$, we obtain Eq. (iv),
\begin{align*}0&=9\cP'\cP^2+a_1\cQ'\cP^2+3\cR'\cP^2+6\cQ'\cP\cQ +2a_1\cP'\cP \cQ +6\cP'\cP\cR+2\cR'\cP\cR +3\cP'\cQ^2\\
&\quad\;\;+2\cQ'\cQ\cR +\cP'\cR^2+9\cP'^2\cP +2a_1\cP'\cQ'\cP+3\cQ'^2\cP +6\cP'\cR'\cP +\cR'^2\cP +a_1\cP'^2\cQ \\
&\quad\;\;+\cR'\cQ^2+6\cP'\cQ'\cQ +2\cQ'\cR'\cQ+\cQ'^2\cR +3\cP'^2\cR+2\cP'\cR'\cR,\end{align*}
Eq. (v), 
\begin{align*}0&=6\cP'a_1\cP^2+a_1\cR'\cP^2 +9\cQ'\cP^2 +6\cR'\cP\cQ+18\cP'\cP \cQ +2a_1\cQ'\cP\cQ \\
&\quad\;\;+a_1\cP'\cQ^2+3\cQ'\cQ^2+2\cR'\cQ\cR+6\cP'\cQ \cR +\cQ'\cR^2+18\cP'\cQ'\cP+6\cR' \cQ'\cP\\
&\quad\;\;+a_1\cQ'^2\cP+6\cR'\cP'\cQ+2a_1\cP'\cQ'\cQ +3\cQ'^2\cQ +9\cP'^2\cQ +\cR'^2\cQ +2\cR'\cQ' \cR\\
&\quad\;\;+a_1\cP'^2\cR+6\cP'^2a_1\cP+2a_1\cP'\cP\cR +2a_1\cP'\cR'\cP+6\cQ'\cP'\cR+6\cQ'\cP\cR,\end{align*} 
and Eq. (vi), 
\begin{align*}-a_2&=3a_1\cQ'\cP^2+6a_1\cP'\cP\cQ +\cP'a_1^2\cP^2+a_1\cQ'\cQ^2+2a_1\cP'\cQ\cR +\cR'\cR^2\\
&\quad\;\;+\cP'^2a^2\cP +3a_1\cP'^2\cQ+2a_1\cR'\cP'\cQ +a_1\cQ'^2\cQ +\cR'^2\cR +2a_1\cQ'\cP'\cR\\
&\quad\;\;+6a\cP'\cQ'\cP+2a_1\cR'\cP\cQ +2a_1\cQ'\cP\cR +2a\cR'\cQ'\cP. 
\end{align*} 
By the definitions of $\cP'$, $\cQ'$, and $\cR'$ $(\square )$, Eq. (iv) becomes 
\begin{align*}0&=9\cP'\cP^2+a_1\cQ'\cP^2+3\cR'\cP^2+6\cQ'\cP\cQ +2a_1\cP'\cP\cQ+6\cP'\cP\cR +2\cR'\cP\cR +3\cP'\cQ^2\\
&\quad\;\;+\cR'\cQ^2+2\cQ'\cQ\cR +\cP'\cR^2+9\cP'^2\cP +2a_1\cP'\cQ'\cP +3\cQ'^2\cP +6\cP'\cR'\cP +\cR'^2\cP\\
&\quad\;\;+a_1\cP'^2\cQ+6\cP'\cQ'\cQ+2\cQ'\cR'\cQ +\cQ'^2\cR +3\cP'^2\cR +2\cP'\cR'\cR\\
&=a_1(\cP u^2 a_1 -2 fu\cP + f\cQ )\cP^2+3(4 u^2 \cP -2ufa_1\cP + 2u\cQ +\cR )\cP^2\\
&\quad\;\;+6(\cP u^2 a_1 -2 fu\cP + f\cQ )\cP \cQ +2a_1( - \cP u^2 + \cP f^2  -\cQ u )\cP \cQ\\
&\quad\;\;+9( - \cP u^2 + \cP f^2  -\cQ u )\cP^2+6( - \cP u^2 + \cP f^2  -\cQ u )\cP \cR +( - \cP u^2 + \cP f^2  -\cQ u )\cR^2\\
&\quad\;\;+2( 4 u^2 \cP -2ufa_1\cP + 2u\cQ+\cR )\cP\cR +3( - \cP u^2 + \cP f^2  -\cQ u )\cQ^2\\
&\quad\;\;+( 4 u^2 \cP -2ufa_1\cP + 2u\cQ +\cR )\cQ^2+2(\cP u^2 a_1 -2 fu\cP + f\cQ )\cQ \cR \\
&\quad\;\;+9( - \cP u^2 + \cP f^2  -\cQ u )^2\cP +2a_1( - \cP u^2 + \cP f^2  -\cQ u )(\cP u^2 a_1 -2 fu\cP + f\cQ )\cP\\
&\quad\;\;+3(\cP u^2 a_1 -2 fu\cP + f\cQ )^2\cP +6( - \cP u^2 + \cP f^2  -\cQ u )( 4 u^2 \cP -2ufa_1\cP + 2u\cQ +\cR )\cP \\
&\quad\;\;+( 4 u^2 \cP -2ufa_1\cP + 2u\cQ +\cR )^2\cP +a_1( - \cP u^2 + \cP f^2  -\cQ u )^2\cQ \\
&\quad\;\;+6( - \cP u^2 + \cP f^2  -\cQ u )(\cP u^2 a_1 -2 fu\cP + f\cQ )\cQ+3( - \cP u^2 + \cP f^2  -\cQ u )^2\cR \\
&\quad\;\;+2(\cP u^2 a_1 -2 fu\cP + f\cQ )( 4 u^2 \cP -2ufa_1\cP + 2u\cQ +\cR )\cQ \\
&\quad\;\;+2( - \cP u^2 + \cP f^2  -\cQ u )( 4 u^2 \cP -2ufa_1\cP + 2u\cQ +\cR )\cR +(\cP u^2 a_1 -2 fu\cP + f\cQ )^2\cR \\
&= (a_1^2u^4-12a_1u^3f+18f^2u^2+3u^2-8a_1fu+u^4-16a_1f^3u+6a_1^2f^2u^2+9f^4+a_1^2u^2+9f^2)\cP^3\\
&\quad\;\;+(-3u-4a_1^2u^3f-12fu+12a_1u^2+12a_1f^2u^2-18uf^2-4u^3f-12f^3u-2u^3\\
&\quad\;\;-2a_1^2u^3+4a_1u^2+2a_1f^2+2a_1f^3+3a_1u^4+a_1f^4+a_1f)\cP^2\cQ \\
&\quad\;\;+(4u^2-8a_1fu+12f^2-5u^4+6f^2u^2-4a_1f^3u+3f^4+a_1^2u^4+3)\cP^2\cR\\
&\quad\;\;+(6f^2+2u^2+6fu^2-6a_1uf^2-4a_1fu+6f^3+6f-2a_1u)\cP\cQ^2\\
&\quad\;\;+2u(-3u^2+2a_1u+3a_1fu-2-4f-3f^2)\cP \cQ \cR\\
&\quad\;\;+(-3u^2+3f^2+3)\cP \cR^2+u(a_1u-2f-1)\cQ^3+(-u^2+f^2+4f+1)\cQ^2\cR-3u\cQ \cR^2.
\end{align*} 

As $\cR=-2\cP$, this further simplifies to
\begin{align*} 
0&=(-17u^2-3f^2+6+8a_1fu+11u^4+6f^2u^2-8a_1f^3u+3f^4-a_1^2u^4\\
&\quad\;\;-12a_1u^3f+6a_1^2f^2u^2+a_1^2u^2)\cP^3\\
&\quad\;\;+(10u^3-4a_1u^2-7u+4fu-6uf^2-4a_1^2u^3f+12a_1f^2u^2-4u^3f-12f^3u\\
&\quad\;\;-2a_1^2u^3+2a_1f^2+2a_1f^3+3a_1u^4+a_1f^4+a_1f)\cP^2\cQ\\
&\quad\;\;+(4f^2+4u^2+6fu^2-6a_1uf^2-4a_1fu+6f^3-2f-2a_1u-2)\cP \cQ^2\\
&\quad\;\;+u(a_1u-2f-1)\cQ^3.
\end{align*}
As $f^2 +f+1 =u^2$ and $a_1u=2f+1$, the coefficient of $\cP^3$ in the previous expression is equal to
\begin{align*}
 &\quad\; -17u^2-3f^2+6+8a_1fu+11u^4+6f^2u^2-8a_1f^3u+3f^4-a_1^2u^4\\
&\quad\;\;-12a_1u^3f+6a_1^2f^2u^2+a_1^2u^2\\ 
&= -17(f^2 +f+1)-3f^2+6+8f(2f+1)+11(f^2+f+1)^2\\
&\quad\;\;+6f^2(f^2+f+1)-8f^3(2f+1)+3f^4-(2f+1)^2(f^2+f+1)\\
&\quad\;\;-12f(2f+1)(f^2+f+1)+6f^2(2f+1)^2+(2f+1)^2\\ 
&=0.
\end{align*}
As $f^2 =-f-1+u^2$, the coefficient of $\cP^2\cQ$ is equal to
\begin{align*} &\quad\; 10u^3-4a_1u^2-7u+4fu-6uf^2-4a_1^2u^3f+12a_1f^2u^2-4u^3f-12f^3u\\
&\quad\;\;-2a_1^2u^3+2a_1f^2+2a_1f^3+3a_1u^4+a_1f^4+a_1f\\
&= 10u^3-4a_1u^2-7u+4fu-6u(-f-1+u^2)+12a_1(-f-1+u^2)u^2\\
&\quad\;\;-(12(-f-1+u^2))fu-2a_1^2u^3+2a_1(-f-1+u^2)+2a_1f(-f-1+u^2)\\
&\quad\;\;+3au^4+a_1(-f-1+u^2)^2+a_1f-4a_1^2u^3f -4u^3f \\ 
&=(-a_1+12u)f^2+(-4a_1^2u^3+22u-16u^3-a_1-12a_1u^2)f+4u^3\\
&\quad\;\;-16a_1u^2-u+16a_1u^4-a_1-2a_1^2u^3\\ 
&=(-a_1+12u)(-f-1+u^2)+(-4a_1^2u^3+22u-16u^3-a_1-12a_1u^2)f\\
&\quad\;\;+4u^3-16a_1u^2-u+16a_1u^4-a_1-2a_1^2u^3\\
&= -2u(-5+2a_1^2u^2+8u^2+6a_1u)f-u(17a_1u+13-16u^2-16u^3a_1+2a_1^2u^2)\\
&=-2u(-5+2(2f+1)^2+8(f^2+f+1)+6(2f+1))f-u(17(2f+1)\\
&\quad\;\;+13-16(f^2+f+1)-(16(2f+1))(f^2+f+1)+2(2f+1)^2)\\
&=0
\end{align*}
For the same reason, the coefficient in $\cQ^2\cP$ is equal to
\begin{align*} &\quad\; 4f^2+4u^2+6fu^2-6a_1uf^2-4a_1fu+6f^3-2f-2a_1u-2\\
&=4f^2+4(f^2+f+1)+6f(f^2+f+1)-6f^2(2f+1)\\
&\quad\;\; -4f(2f+1)+6f^3-2f-2(2f+1)-2\\
&=0
\end{align*}
Finally, the coefficient in $\cQ^3$ is also $0$, whence the Eq. (iv) is always satisfied. \\ 
 
For Eq. (v), substitution of $\cP'$, $\cQ'$ and $\cR'$ via $(\square)$ yields
\begin{align*}0&=6\cP'a_1\cP^2+a_1\cR'\cP^2+9\cQ'\cP^2+6\cR'\cP\cQ +18\cP'\cP \cQ +2a_1\cQ'\cP \cQ +6\cQ'\cP\cR +2a_1\cP'\cP\cR\\
&\quad\;\;+a_1\cP'\cQ^2+3\cQ'\cQ^2+2\cR'\cQ\cR +6\cP'\cQ\cR +\cQ'\cR^2+18\cP'\cQ'\cP +6\cR'\cQ'\cP +6\cP'^2a_1\cP \\
&\quad\;\;+a_1\cQ'^2\cP +6\cR'\cP'\cQ +2a_1\cP'\cQ'\cQ +3\cQ'^2\cQ +9\cP'^2\cQ +\cR'^2\cQ +2\cR'\cQ'\cR +6\cQ'\cP'\cR \\
&\quad\;\;+a_1\cP'^2\cR +2a_1\cP'\cR'\cP \\
&=(4a_1u^4+42a_1f^2u^2+6a_1f^4+7a_1u^2-18fu+a_1^3u^4\\
&\quad\;\;-12a_1^2u^3f+6a_1f^2-4a_1^2f^3u-12u^3f-36f^3u-2a_1^2fu)\cP^3\\
&\quad\;\;+(9f+6a_1^2f^2u^2+18f^3-24a_1uf^2-4a_1u-16a_1fu-12a_1u^3f-16a_1f^3u\\
&\quad\;\;+18fu^2+18f^2u^2+u^4+9f^4+18f^2+6u^2+a_1^2u^4+2a_1^2u^2+6u^2a_1^2f-6u^3a_1)\cP^2\cQ \\
&\quad\;\;+(a_1+3a_1u^4+12a_1f^2u^2+a_1f^4-4u^3f-12f^3u+8a_1u^2+4a_1f^2-24fu-4a_1^2u^3f)\cP^2\cR \\
&\quad\;\;+(4a_1u^2-2u^3+12a_1u^2f-6u-18uf^2+2a_1f+2a_1f^2-12fu-2a_1^2u^3+2a_1f^3)\cP\cQ^2\\
&\quad\;\;+(12f+4u^2-8a_1fu+12f^2+6-6a_1uf^2-4a_1u+6fu^2+6f^3)\cP\cQ\cR \\
&\quad\;\;+3u(a_1u-2f)\cP \cR^2+(u^2-2a_1fu+3f^2-a_1u+3f)\cQ^3+u(a_1u-2f-4)\cQ^2\cR\\
&\quad\;\;+(3+3f)\cQ\cR^2
\end{align*} 
As $\cR =-2\cP$, this becomes 
\begin{align*}
0&=(3a_1u^2+6fu-2a_1-2a_1u^4+18a_1f^2u^2+4a_1f^4-4u^3f-12f^3u\\
&\quad\;\;-2a_1f^2-4a_1^2u^3f+a_1^3u^4-4a_1^2f^3u-2a_1^2fu)\cP^3\\
&\quad\;\;+(-3f+6a_1^2f^2u^2+6f^3-12a_1uf^2+4a_1u-12a_1u^3f-16a_1f^3+6fu^2+18f^2u^2\\
&\quad\;\;+u^4+9f^4-6f^2-2u^2+a_1^2u^4+2a_1^2u^2+6u^2a_1^2f-6u^3a_1)\cP^2\cQ \\
&\quad\;\;+(2a_1u^2-2u^3+12a_1u^2f+2u-18uf^2+2a_1f+2a_1f^2-8fu-2a_1^2u^3+2a_1f^3)\cP\cQ^2\\
&\quad\;\;+(u^2-2a_1fu+3f^2-a_1u+3f)\cQ^3
\end{align*} 
Via $u^2 =f^2 +f+1$ and $ua_1 = 2f+1$, we find that the coefficient in $\cP^3$ is equal to
\begin{align*}
& \ \ \ 3a_1u^2+6fu-2a_1-2a_1u^4+18a_1f^2u^2+4a_1f^4-4u^3f-12f^3u\\
&\quad\;\;-2a_1f^2-4a_1^2u^3f+a_1^3u^4-4a_1^2f^3u-2a_1^2fu\\
&=3a_1u^2+6fu-2a_1-2a_1u^4+18a_1f^2u^2+4a_1f^4-4u^3f-12f^3u\\
&\quad\;\;-2a_1f^2-4a_1u^2f(2f+1)+a_1u^2(2f+1)^2-4a_1f^3(2f+1)-2a_1f(2f+1)\\
&=-2a_1u^4-4fu^3+2a_1(7f^2+2)u^2-6f(-1+2f^2)u-2a_1(f^2+f+1)(2f^2+1)\\
&=-2a_1u^4-4fu^3+2a_1(7f^2+2)u^2-6f(-1+2f^2)u-2a_1u^2(2f^2+1)\\
&=-2u^3(2f+1)-4fu^3+2(7f^2+2)u(2f+1)-6f(-1+2f^2)u\\
&\quad\;\;-2u(2f+1)(2f^2+1)\\
&=(-8f-2)u^3+2(4f+1)(f^2+f+1)u\\
&=(-8f-2)u^3+2(4f+1)u^3\\
&=0.\end{align*} 
Via the same relations, the coefficient in $\cP^2 \cQ$ is equal to
\begin{align*}
&\quad\;-3f+6a_1^2f^2u^2+6f^3-12a_1uf^2+4a_1u-12a_1u^3f-16a_1f^3+6fu^2\\
&\quad\;\;+18f^2u^2+u^4+9f^4-6f^2-2u^2+a_1^2u^4+2a_1^2u^2+6u^2a_1^2f-6u^3a_1\\
&=-3f+6f^2(1+2f)^2+6f^3-12f^2(1+2f)+4(1+2f)-16f^3(1+2f)\\
&\quad\;\;-12f(1+2f)(f^2+f+1)+6f(f^2+f+1)+18f^2(f^2+f+1)\\
&\quad\;\;+(f^2+f+1)^2+9f^4-6f^2-2(f^2+f+1)+(1+2f)^2(f^2+f+1)\\
&\quad\;\;+2(1+2f)^2+6f(1+2f)^2-6(1+2f)(f^2+f+1)\\
&=0.
\end{align*} 
Similarly, we find that the coefficient in $\cP \cQ^2$ is equal to
\begin{align*}&\quad\; 2a_1u^2-2u^3+12a_1u^2f+2u-18uf^2+2a_1f+2a_1f^2-8fu-2a_1^2u^3+2a_1f^3\\
&= 2a_1u^2-2u^3+14a_1u^2f+2u-18uf^2-8fu-2a_1^2u^3\\
&=2 (2f+1)u -2 u(f^2+f+1) + 14 (2f+1) uf+ 2u -18uf^2 -8fu -2u(2f+1)^2\\
&=0 
\end{align*}
and finally, that the coefficient in $\cQ^3$ is equal to
\begin{align*} & \ \ \ u^2-2a_1fu+3f^2-a_1u+3f\\ 
&= (f^2 +f+1) -2f(2f+1) + 3f^2 -(2f+1) +3f\\
&=0. \end{align*}
For Eq. (ix), we substitute $\cP'$, $\cQ'$, and $\cR'$ via $(\square )$ to obtain
\begin{align*}-a_2&=3a_1\cQ'\cP^2+6a_1\cP'\cP\cQ +\cP'a_1^2\cP^2+a_1\cQ'\cQ^2 +2a_1\cP'\cQ\cR +\cR'\cR^2+6a\cP'\cQ'\cP \\
&\quad\;\;+\cP'^2a^2\cP+3a_1\cP'^2\cQ +2a_1\cR'\cP'\cQ +a_1\cQ'^2\cQ +\cR'^2\cR+2a_1\cQ' \cP'\cR\\ 
&\quad\;\; +2a\cR'\cQ'\cP +2a_1\cQ'\cP\cR +2a_1\cR'\cP\cQ \\
&=-a_1(-2a_1u^2+6fu-3a_1u^4+4fu^3-12a_1u^2f^2+12f^3u-a_1f^4+4fa_1^2u^3-a_1f^2)\cP^3\\
&\quad\;\;+a_1(6u^2f-a_1u+6f^3-6a_1uf^2+3f+a_1^2u^4+6f^2u^2-5u^4-4a_1f^3u\\
&\quad\;\;+2u^2+6f^2-4a_1fu+3f^4)\cP^2\cQ\\
&\quad\;\;+2u(8u^3-6a_1fu^2+3ua_1^2f^2-a_1^2u^3-2a_1f^3+2a_1^2u-4fa_1)\cP^2\cR\\
&\quad\;\;+2a_1u(-3u^2+a_1u+3a_1fu-3f^2-1-2f)\cP\cQ^2\\
&\quad\;\;+(4fa_1+2a_1-4a_1u^2+4a_1f^2+16u^3-6a_1fu^2-2a_1^2u^3+2a_1f^3)\cP\cQ\cR\\
&\quad\;\;-6u(-2u+fa_1)\cP\cR^2+a_1(f^2-u^2+f)\cQ^3-2u(-2u+2a_1+fa_1)\cQ^2\cR\\ 
& \quad\;\;+6u\cQ\cR^2+2\cR^3
\end{align*}
As $\cR=-2\cP$, we obtain
\begin{align*}-a_2&=(-5fa_1-4a_1+10a_1u^2-2a_1f^2-32u^3+18a_1fu^2+4a_1^2u^3+2a_1f^3-a_1^2u\\
&\quad\;\;-6ua_1^2f^2+a_1^3u^4+6a_1u^2f^2-5a_1u^4-4a_1^2f^3u-4a_1^2fu+3a_1f^4+24u)\cP^2\cQ \\
&\quad\;\;+2u(a_1^2u+3a_1-3a_1u^2-3a_1f^2+3a_1^2fu-4u)PQ^2+a_1(f^2-u^2+f)\cQ^3\\
&\quad\;\;+(48u^2-14a_1fu-32u^4+20a_1u^3f+7a_1^2u^4-4a_1f^3u-6a_1^2u^2+f^4a_1^2\\
&\quad\;\;-4fa_1^3u^3+a_1^2f^2-16)\cP^3\\
\end{align*}
Via $f^2 +f+1=u^2$, $a_1u = 2f+1$, and $ \left(a^2-4\right) f^2 +\left(a^2-4\right)f + \left(a^2 - 1\right)=0$, the coefficient of (ix) in $\cP^3$ is equal to
\begin{align*}&\quad\;\;48u^2-14a_1fu-32u^4+20a_1u^3f+7a_1^2u^4-4a_1f^3u-6a_1^2u^2+f^4a_1^2-4fa_1^3u^3\\
&\quad\;\;+a_1^2f^2-16\\
&= 48(f^2+f+1)-14f(2f+1)-32(f^2+f+1)^2+20f(2f+1)(f^2+f+1)+a_1^2f^2\\
&\quad\;\;+7(2f+1)^2(f^2+f+1)-4f^3(2f+1)-6(2f+1)^2+a_1^2f^4-4f(2f+1)^3-16\\
&=(a_1^2-4)f^4+(a_1^2-1)f^2-3f+1=-f^3(a_1^2-4)-3f+1\\
&=-f^3a_1^2+(f+1)(-1+2f)^2=-f^3a_1^2+(f+1)(a_1u-2)^2\\ 
&=(-f^3+u^2f+u^2)a_1^2-4(f+1)ua_1+4f+4 \\
&= (-f^3+u^2f+u^2)a_1^2-4(a_1u-f)ua_1+4f+4\\
&=(-3u^2-f^3+u^2f)a_1^2+4fua_1+4f+4\\
&=(-3u^2+f(f+1))a_1^2+4f(2f+1)+4f+4\\
&=(-3u^2+f^2+f)a_1^2+8f^2+8f+4=(-3u^2-1+u^2)a_1^2+8(-1+u^2)+4\\
&=-2u^2a_1^2-a_1^2-4+8u^2=-2(2f+1)^2-a_1^2-4+8(f^2+f+1)\\
&=-a_1^2+2,
\end{align*} 
and the coefficient in $\cP \cQ^2$ is equal to
\begin{align*}&\quad\;\;2u(a_1^2u+3a_1-3a_1u^2-3a_1f^2+3a_1^2fu-4u)\\
&= a_1u(2f+1)+3a_1u-3a_1u(f^2+f+1)-3a_1uf^2+3a_1uf(2f+1)-4u^2\\
&= -4u^2+a_1u(2f+1)\\
&=-8(f^2+f+1)+2(2f+1)^2\\
&=-6.
\end{align*} 
The coefficient in $\cP^2\cQ$ is equal to 
\begin{align*}&\quad\;\;-5fa_1-4a_1+10a_1u^2-2a_1f^2-32u^3+18a_1fu^2+4a_1^2u^3+2a_1f^3-a_1^2u\\
&\quad\;\;-6ua_1^2f^2+a_1^3u^4+6a_1u^2f^2-5a_1u^4-4a_1^2f^3u-4a_1^2fu+3a_1f^4+24u\\
&=3a_1f^4-2(2ua_1-1)a_1f^3-2(3ua_1-3u^2+1)a_1f^2-(4ua_1-18u^2+5)a_1f+4a_1^2u^3\\
&\quad\;\; +a_1^3u^4-ua_1^2+10a_1u^2-4a_1-5a_1u^4-32u^3+24u\\
&=3a_1(-f-1+u^2)^2-2(2ua_1-1)af(-f-1+u^2)-2(3ua_1-3u^2+1)a_1(-f-1+u^2)\\
&\quad\;\;-(4ua_1-18u^2+5)a_1f+4a_1^2u^3+a_1^3u^4-ua_1^2+10a_1u^2-4a_1-5a_1u^4-32u^3+24u\\
&=a_1(4ua_1+1)f^2-a_1(4u^3a_1-6ua_1-8u^2-1)f+4a_1u^4-4a_1u^2+a_1-32u^3+24u\\
&\quad\;\;-2a_1^2u^3+5ua_1^2+a_1^3u^4\\
&=a_1(4ua_1+1)(-f-1+u^2)-a_1(4u^3a_1-6ua_1-8u^2-1)f+4a_1u^4-4a_1u^2+a_1\\
&\quad\;\;-32u^3+24u-2a_1^2u^3+5ua_1^2+a_1^3u^4\\
&=a_1^3u^4+(-4f+2)a^2u^3+4au^4+(2f+1)a^2u+(8f-3)au^2-32u^3+24u\\ 
&=a_1^3u^4+(-2ua_1+4)a_1^2u^3+4a_1u^4+a_1ua_1^2u+(4ua_1-7)a_1u^2-32u^3+24u\\
&=-a_1(a_1^2-4)u^4+8(a_1^2-4)u^3+a_1(-7+a_1^2)u^2+24u,
\end{align*}
which by \eqref{star} yields
\begin{align*}
& \ \ \ 3a_1u^2-24u+a_1(-7+a_1^2)u^2+24u \\ 
&= a_1(a_1^2-4)u^2\\
&=-3a_1,
\end{align*} 
and finally, the coefficient in $\cQ^3$ is equal to $a_1(f^2-u^2+f)= a_1$. 

As a conclusion, $\cP$, $\cQ$, and $\cR$ must only satisfy the following equations:
$$\left\{ \begin{array}{cllr}  \cR &=& -2\cP & (1)
\\ 1&=&\cQ^2+a_1\cP \cQ + \cP^2 & (2)
\\a_2&=&(-2 + a_1^2) \cP^3 + 3 a_1 \cP^2 \cQ+ 6 \cP \cQ^2 + a_1 \cQ^3 & (3)
\end{array}\right.$$
When $K= \mathbb{F}_q(x)$, we have the following:
\begin{enumerate}[(a)] \item If $p \neq 2$, then we write $\cP=\frac{U}{V}$ with and $U,V \in \mathbb{F}_q[x]$ relatively prime. The quadratic equation (2) in $\cQ$, $$\cQ^2+a_1\cP \cQ +\cP^2 -1=0,$$ has a solution in $\mathbb{F}_q(x)$ if, and only if, the discriminant $\Gamma$ of the polynomial $h(X)=X^2+a_1\cP X + \cP^2-1$ is a square in $\mathbb{F}_q(x)$. By definition, $$\Gamma =a_1^2\cP^2 -4 (\cP^2 -1)= (a_1^2-4)\cP^2 +4.$$ 
We write $a_1= P_1/ Q_1$ where $P_1$ and $Q_1\in \mathbb{F}_q[x]$ and $(P_1, Q_1)=1$. We also write $U = gcd(U, Q_1) U'$ and $Q_1V= gcd(Q_1 ,U)V'$, with $U'$ and $V' \in \mathbb{F}_q[x]$.

Also, $\Gamma$ is a square if, and only if, there exists $W$ in $\mathbb{F}_q[x]$ such that $$(P_1^2 -4Q_1^2) U'^2 +4 V'^2 = W^2.$$ Multiplication by $-27$ yields $$Q_1^2d_{L/\mathbb{F}_q(x)} U'^2 = -27 ( W^2 -4V'^2)= -27 (W -2V')(W+2V'),$$ where $d_{L/\mathbb{F}_q(x)}=-27 (a_1^2-4)$ is the discriminant of the minimal polynomial of $z_1$ $T_1(X)= X^3-3X -a_1$. As $L/\mathbb{F}_q(x)$ is Galois, it follows by Lemma \ref{disc} that $d_{L/\mathbb{F}_q(x)} \in \mathbb{F}_q(x)^2$. Let $\delta \in \mathbb{F}_q[x]$ be such that $\delta^2 = Q_1^2 d_{L/\mathbb{F}_q(x)}$. Thus $(\delta U')^2 = -27 (W -2V')(W+2V')$. As $U'$ and $V'$ are relatively prime, it follows that $U'$, $V'$, and $W$ are pairwise coprime. Thus $W-2V'$ and $W+2V'$ too are coprime. Therefore, by unique factorisation in $\mathbb{F}_q[x]$, it follows that, up to a unit, $W-2V'$ and $ W+2V'$ are squares in $\mathbb{F}_q[x]$. Hence, there exist $C, D \in \mathbb{F}_q[x]$ relatively prime such that $W-2V' =\xi  C^2 $ and $W+2V' = \zeta D^2$, where $\xi,\zeta \in \mathbb{F}_q^*$ are such that $\xi \zeta$ is equal to $-27^{-1}$ up to a square in $\mathbb{F}_q^*$. Thus 
$$2W = \xi C^2 + \zeta D^2 \quad\text{and}\quad  \quad 4V' = \zeta D^2 - \xi C^2,$$ 
so that
$$2W =\zeta^{-1}(\zeta \xi C^2 + \zeta^2 D^2) \quad\text{and}\quad \quad 4V' =\zeta^{-1} (\zeta^2 D^2 - \zeta \xi C^2).$$
This is equivalent to
$$2W =\zeta^{-1}(-3^{-1} C'^2 +  D'^2) \quad\text{and}\quad \quad 4V' =\zeta^{-1} (D'^2 +3^{-1} C'^2),$$
where $C' = 3^{-1}C$ and $D'= \zeta D$. Therefore, we have
$$2 W -4 V' = -2 7^{-1} \zeta^{-1} C'^2 \quad\text{and}\quad \quad 2W +4V' = 2 \zeta^{-1} D'^2,$$
and it follows that
$$-27 (W-2V')(W+2V') = 9 \zeta^{-2} D'^2 C'^2 = \delta^2 U'^2.$$ 
By construction, we have
$$\delta U' = \pm 3 \zeta^{-1} D' C'.$$ 
As different values of $\zeta$ result in the same rational function $\phi$, up to multiplication of $C$ by $-1$, we have without loss of generality that $$ \delta U' =  3C'' D''  \quad\text{and}\quad \quad  4Q_1 V' = D''^2 + 3^{-1} C''^2,$$
where $C'', D'' \in \mathbb{F}_q[x]$. 

For the converse, suppose that there exist $C,D \in \mathbb{F}_q[x]$ with $C$ and $D$ relatively prime such that $4Q_1V'= D^2 +3^{-1} C^2$ and $\delta U' = 3CD$. Then we find that 
$$\cP =\frac{12 CD Q_1}{\delta(D^2 + 3^{-1} C^2)}$$
and
\begin{align*} \Gamma &= (a_1^2-4)\cP^2 +4\\
& = \frac{ 12^2C^2 D^2 Q_1^2 (P_1^2 -4 Q_1^2 )+ 4\delta^2 (D^2 +3^{-1}C^2)^2}{\delta^2 (D^2 +3^{-1} C^2)^2}\\
& = \frac{ 4\delta^2 (-4 \cdot 3^{-1} C^2 D^2 + (D^2 +3^{-1}C^2)^2)}{\delta^2 (D^2 +3^{-1} C^2)^2}\\
& = \frac{ 4  (D^2 -3^{-1}C^2)^2}{ (D^2 +3^{-1} C^2)^2}
, \end{align*}

which is a square in $K$. Hence, $\Gamma=\gamma^2$ with $$\gamma =\frac{2  (D^2 -3^{-1}C^2)}{ D^2 +3^{-1} C^2} , \quad\quad\quad\text{and}\quad\quad\quad \cQ = \frac{-a_1 \cP \pm \gamma}{2}= \frac{-a_1 \cP \pm \frac{2  (D^2 -3^{-1}C^2)}{ D^2 +3^{-1} C^2}}{2}.$$ 
 \item If $p=2$, then once more Eq. (2) holds. If $\cP =0$, then Eq. (2) yields $\cQ =\pm 1$. By Eq. (3), if $\cQ =1$, then $a_2=a_1$, whereas if $\cQ =-1$, then $a_2=-a_1$. If $\cP \neq 0$, then as $a_1 \neq 0$, division of Eq. (2) by $(a_1\cP )^2$ yields $$\left(\frac{\cQ}{a_1\cP}\right)^2+\left(\frac{\cQ}{a_1\cP}\right)+\frac{\cP^2-1}{(a_1\cP )^2}=0.$$ Thus, $\cQ /(a_1 \cP )$ is a solution of the Artin-Schreier equation $$X^2-X-(\cP^2-1)/(a_1\cP)^2=0,$$ as claimed.
\end{enumerate}
\end{proof} 
\begin{remarque}
\begin{enumerate} 
\item By Eq. (3), for the solution $\cQ = \frac{-a_1 \cP + \gamma}{2}$, we obtain furthermore that
\begin{align*} a_2&=(-2 + a_1^2) \cP^3 + 3 a_1 \cP^2 \cQ + 6 \cP \cQ^2 + a_1 \cQ^3\\
&=(-2 + a_1^2) \cP^3 + 3 a_1 \cP^2 (-a_1\cP +\gamma)/2+ 6 \cP ((-a_1\cP +\gamma)/2)^2 + a_1 ((-a_1\cP +\gamma )/2)^3\\
&=-1/8(a_1^2-4)^2\cP^3+3/8a_1(a_1^2-4)\cP^2\gamma-3/8(a_1^2-4)\cP\gamma^2+1/8a_1\gamma^3\\
&=-1/8(a_1^2-4)^2\cP^3+3/8a_1(a_1^2-4)\cP^2\gamma-3/8(a_1^2-4)\cP((a_1^2-4)\cP^2+4)\\
& \quad +1/8a_1\gamma((a_1^2-4)\cP^2+4)\\
&=-1/2(a_1^2-4)^2\cP^3+1/2a_1(a_1^2-4)\cP^2\gamma-3/2(a_1^2-4)\cP +1/2a_1\gamma \\
&= -1/2(\gamma^2-1)(a_1^2-4)\cP+1/2a_1\gamma(\gamma^2-3).
\end{align*}
For the solution $\cQ = \frac{-a_1 \cP - \gamma}{2}$, we obtain the analogous
\begin{align*} a_2&=(-2 + a_1^2) \cP^3 + 3 a_1 \cP^2 \cQ + 6 \cP \cQ^2 + a_1 \cQ^3\\
&=(-2 + a_1^2) \cP^3 + 3 a_1 \cP^2 (-a_1\cP-\gamma)/2+ 6 \cP ((-a_1\cP -\gamma)/2)^2 + a_1 ((-a_1\cP -\gamma )/2)^3\\
&=-1/8(a_1^2-4)^2\cP^3-3/8a_1(a_1^2-4)\cP^2\gamma-3/8(a_1^2-4)\cP \gamma^2-1/8a_1\gamma^3\\
&=-1/8(a_1^2-4)^2\cP^3-3/8a_1(a_1^2-4)\cP^2\gamma-3/8(a_1^2-4)\cP((a_1^2-4)\cP^2+4)\\
&\quad -1/8a_1\gamma((a_1^2-4)\cP^2+4)\\
&=-1/2(a_1^2-4)^2\cP^3-1/2a_1(a_1^2-4)\cP^2\gamma-3/2(a_1^2-4)\cP-1/2a_1\gamma \\
&= -1/2(\gamma^2-1)(a_1^2-4)\cP-1/2a_1\gamma(\gamma^2-3).
\end{align*}
\item In the language of Theorem 5.8, the base change matrix from $\{1,z_1,z_1^2\}$ to $\{1,z_2,z_2^2\}$ is given by $$\begin{pmatrix} 1 & 0&0 \\
-2\cP&\cQ &\cP\\	
2\cP \cQ a_1+4\cP^2& \cP^2 a_1 + 2\cP \cQ &- \cP^2+\cQ^2\\
\end{pmatrix},$$
and the inverse of this matrix is equal to $$\begin{pmatrix} 1 & 0&0 \\
\frac{-2\cP}{a_1\cP^3+3\cQ \cP^2-\cQ^3}&\frac{\cP^2-\cQ^2}{a_1\cP^3+3\cQ \cP^2-\cQ^3} & \frac{p}{a_1\cP^3+3\cQ \cP^2-\cQ^3}\\	
\frac{2\cP (4\cP \cQ +a_1(\cP^2+\cQ^2)}{a_1\cP^3+3\cQ \cP^2-\cQ^3}&\frac{\cP (a_1\cP +2\cQ )}{a_1\cP^3+3\cQ \cP^2-\cQ^3}&\frac{-\cQ}{a_1\cP^3+3\cQ \cP^2-\cQ^3}\\
\end{pmatrix}.$$ Therefore, $z_1$ may be expressed in terms of $z_2$ as $$z_1 = \frac{1}{a_1\cP^3+3\cQ \cP^2-\cQ^3}\left[-2\cP + (\cP^2-\cQ^2)z_2 + \cP z_2^2 \right].$$
\end{enumerate}
\end{remarque}

\bibliographystyle{plain}
\raggedright
\bibliography{references}

\begin{thebibliography}{10}

\bibitem{Con}
K.~Conrad.
\newblock Galois groups of cubics and quartics in all characteristics.
\newblock {\em Unpublished note}.

\bibitem{Dickson}
L.E. Dickson.
\newblock Criteria for the irreducibility of functions in a finite field.
\newblock {\em Bull. Amer. Math. Soc.}, pages 1--8, 1906.

\bibitem{LiNi}
R.~Lidl and H.~Niederreiter.
\newblock {\em Introduction to finite fields and their applications}.
\newblock Cambridge, 1986.

\bibitem{MadMad}
M.~Madan and D.~Madden.
\newblock The exponent of class groups in congruence function fields.
\newblock {\em Acta Arith.}, 32(2):183--205, 1977.

\bibitem{Scheidler}
R.~Scheidler.
\newblock Algorithmic aspects of cubic function fields.
\newblock {\em 6th International Symposium, ANTS-VI}, pages 395--410, 2004.

\bibitem{modernreference}
J.~Shurman.
\newblock {\em Geometry of the Quintic}.
\newblock Wiley, 1997.

\bibitem{Sti}
H.~Stichtenoth.
\newblock {\em Algebraic Function Fields and Codes}.
\newblock Springer, 2009.

\bibitem{Vil}
G.D. Villa-Salvador.
\newblock {\em Topics in the Theory of Algebraic Function Fields}.
\newblock Birkh\"{a}user, 2006.

\bibitem{Williams}
K.S. Williams.
\newblock Note on cubics over {GF}$(2^n)$ and {GF}$(3^n)$.
\newblock {\em J. Num. Th.}, 7:361--365, 1975.

\bibitem{Cardano}
T.~Richard Witmer.
\newblock {\em Ars Magna or the Rules of Algebra, by Girolamo Cardano}.
\newblock 1968.

\end{thebibliography}
\vspace{.5cm}

$$\begin{array}{ll}
\text{\emph{Sophie Marques}} & \text{\emph{Kenneth Ward}} \\

\text{\emph{Courant Institute of Mathematical Sciences}} & \text{\emph{Department of Mathematics \& Statistics}} \\

\text{\emph{New York University}}& \text{\emph{American University}}\\

\text{\emph{E-mail: marques@cims.nyu.edu}} & \text{\emph{E-mail: kward@american.edu}} \\

\end{array}$$

\end{document}